\documentclass[11pt,oneside,english,reqno]{amsart}
\usepackage[T1]{fontenc}
\usepackage[utf8]{inputenc}
\usepackage[a4paper]{geometry}
\usepackage{mathrsfs}
\usepackage{amstext}
\usepackage{amsthm}
\usepackage{amssymb}
\usepackage{bbm}
\usepackage{shuffle}
\usepackage{color}
\usepackage{float}
\usepackage{enumerate}
\usepackage{verbatim} 
\usepackage{stmaryrd}
\usepackage{enumitem}
\usepackage{graphicx}
\usepackage{todonotes}
\usepackage{cancel}
\usepackage[normalem]{ulem}
\usepackage{subcaption}

\usepackage{hyperref}

\usetikzlibrary{arrows.meta, shapes.geometric}
\usepackage{tikz, pgfplots}
\usepackage{tikz-cd}
\usetikzlibrary{matrix,chains,positioning,decorations.pathreplacing,arrows}
\usetikzlibrary{calc,positioning}

\makeatletter
\numberwithin{equation}{section}
\numberwithin{figure}{section}
\theoremstyle{plain}
\newtheorem{thm}{\protect\theoremname}
\numberwithin{thm}{section}
\theoremstyle{definition}
\newtheorem{defn}[thm]{\protect\definitionname}
\theoremstyle{remark}
\newtheorem{rem}[thm]{\protect\remarkname}
\theoremstyle{plain}
\newtheorem{lem}[thm]{\protect\lemmaname}
\theoremstyle{plain}
\newtheorem{prop}[thm]{\protect\propositionname}
\theoremstyle{plain}
\newtheorem{cor}[thm]{\protect\corollaryname}
\theoremstyle{plain}
\newtheorem{example}[thm]{\protect\examplename}

\makeatother

\usepackage{babel}
\providecommand{\corollaryname}{Corollary}
\providecommand{\definitionname}{Definition}
\providecommand{\lemmaname}{Lemma}
\providecommand{\propositionname}{Proposition}
\providecommand{\remarkname}{Remark}
\providecommand{\theoremname}{Theorem}
\providecommand{\examplename}{Example}

\newcommand{\red}[1]{{\color{red} #1}}
\newcommand{\blue}[1]{{\color{blue} #1}}

\newcommand{\cA}{\mathcal{A}}

\newcommand{\cH}{\mathcal{H}}

\newcommand{\cS}{\mathcal{S}}

\newcommand{\sig}{\mathbf{Sig}}
\newcommand{\bsig}{\mathbf{BSig}}

\newcommand{\trees}{\mathcal{T}}

\newcommand{\BB}{\mathbb{B}}

\newcommand{\NN}{\mathbb{N}}

\newcommand{\RR}{\mathbb{R}}
\renewcommand{\SS}{\mathbb{S}}

\newcommand{\VV}{\mathbb{V}}

\newcommand{\bbX}{\Bar{\mathbb{X}}}
\newcommand{\bfX}{\Bar{\mathbf{X}}}

\newcommand{\bB}{\mathbf{B}}

\newcommand{\br}{\mathbf{r}}

\newcommand{\bW}{\mathbf{W}}
\newcommand{\bX}{\mathbf{X}}
\newcommand{\bx}{\mathbf{x}}
\newcommand{\bY}{\mathbf{Y}}

\newcommand{\bZ}{\mathbf{Z}}

\newcommand{\se}{\mathsf e}

\newcommand{\bu}{\mathbf{u}}
\newcommand{\bv}{\mathbf{v}}

\newcommand{\bw}{\mathbf{w}}

\newcommand{\hopf}{\mathscr{H}}

\newcommand{\Id}{\mathrm{Id}}

\newenvironment{tree}
               {\begin{tikzpicture}[every node/.style={circle,draw,fill,minimum size=3pt,inner sep=0pt,outer sep=0pt},line cap=round,baseline = .1pt]}
               {\end{tikzpicture}}

\newcommand{\innerprod}[1]{\left\langle #1 \right\rangle}

\usepackage[external]{forest}

\def\dtR<#1>{\Forest{[#1]}}
\def\dtI<#1#2>{\Forest{[#1[#2]]}}
\def\dtII<#1#2#3>{\Forest{[#1[#2[#3]]]}}
\def\dtV<#1#2#3>{\Forest{[#1[#2][#3]]}}

\forestset{
  decor/.style = {
    label/.expanded = {[inner sep = 0.1ex, font=\unexpanded{\tiny}]right:{$#1$}}
  },
  root/.style = {minimum size = 0.5ex},
  decorated/.style = {
    for tree = {
      circle, fill, inner sep = 0ex, minimum size = 0.5ex,
      grow' = north, l = 0, l sep = 0.8ex, s sep = 0.5em,
      fit = tight, parent anchor = center, child anchor = center,
      delay = {decor/.option = content, content =}
    }
  },
  default preamble = {decorated, root},
  begin draw/.code={\begin{tikzpicture}[baseline={([yshift=-0.5ex]current bounding box.center)}]},
}
\title{Branched Signature Model}
\author{
Munawar Ali$^*$~~and~~Qi Feng$^\dagger$
}
\thanks{$^*$: Department of
Mathematics, Florida State University, Tallahassee, 32306; email: ma22bm@fsu.edu.}
\thanks{$^\dagger$: Department of
Mathematics, Florida State University, Tallahassee, 32306; email: qfeng2@fsu.edu. This author is partially supported by the National Science Foundation under grant \#DMS-2420029.}
\date{\today}

\begin{document}
\sloppy
\maketitle

\begin{abstract}
In this paper, we introduce the branched signature model, motivated by the branched rough path framework of [Gubinelli, Journal of Differential Equations, 248(4), 2010], which generalizes the classical geometric rough path. We establish a universal approximation theorem for the branched signature model and demonstrate that iterative compositions of lower-level signature maps can approximate higher-level signatures. Furthermore, building on the existence of the extension map proposed in [Hairer-Kelly. Annales de l'Institue Henri Poincar\'e, Probabilit\'es et Statistiques 51, no. 1 (2015)], we show how to explicitly construct the extension of the original paths into higher-dimensional spaces via a map $\Psi$, so that the branched signature can be realized as the classical geometric signature of the extended path. This framework not only provides an efficient computational scheme for branched signatures but also opens new avenues for data-driven modeling and applications. We demonstrate the effectiveness of the proposed framework through two representative applications: rough volatility calibration and optimal execution.
\end{abstract}



\section{Introduction}
\textbf{Background and motivation.} 
The signature of a bounded variation path \(\bX:[0,T] \to \RR^d\) is defined from the iterated integrals of $\bX$. More precisely, the $N$-th order signature of $\bX$ is  defiend as, 
\begin{align}\label{def: sig}
\sig^N(\bX)_{st}= \sum_{k=0}^N \sum_{\{i_1,\cdots,i_k\}\in \{ 1,\cdots d\} }\int_{s< t_1 <\cdots <t_k<t} d\bX^{i_1}_{t_1}\cdots d\bX^{i_k}_{t_k} \se_{i_1}\otimes \cdots\otimes \se_{i_k}, 
\end{align}
for $0\le s\le t< \infty$. The concept of the signature was first introduced by Chen in the 1950s \cite{chen1954iterated}. Since then, the notion of the signature has been extended to a much broader class of paths and has become a fundamental object in Lyons’ rough path theory \cite{lyons1998differential}.  For example, for a $d$-dimensional fractional Brownian motion $B^H$, the $N$-th order signature $\sig^N(B^H)_{st}$ exists almost surely given the Hurst parameter $H>1/4$ (see \cite{Friz_Victoir_2010}[Chapter 14]). From its very definition, the $N$-th order signature lives in the truncated tensor algebra $T_N(\mathbb R^d)$. By viewing $T_N(\mathbb{R}^d)$ as a flat linear space, one can construct $\mathbf{X}$-driven models using linear combinations of the signature components. Such models are commonly referred to as {\it signature models} in the literature. Furthermore, the signature actually lies in a strict subspace $\mathbb{G}^N(\mathbb{R}^d) \subsetneq T_N(\mathbb{R}^d)$, known as the step-$N$ free Carnot group over $\mathbb{R}^d$ (see e.g.: \cite{baudoin2004introduction}), which endows signature with rich geometric and algebraic (or group) structures. The geometric and algebraic structure enables us to study many properties of signature with the help of pre-existing results in the theory of group and algebraic structures. One such property is the  multiplicativity that if we multiply two components of a signature (like a group product), we get another component of the signature or any linear combination of the components of the signature. A simple example is the integration by parts formula (or the Chain rule), which can be recast in terms of signatures as follows,
\begin{align}\label{IBP}
\int_s^t d\bX^i_r \int_s^t d\bX^j_r = \int_s^t \int_s^v d\bX^i_u d\bX^j_v + \int_s^t \int_s^v d\bX^j_u d\bX^i_v.
\end{align}
This nice geometric property has then been used as in the very definition of geometric rough path (see e.g.:\cite{friz2014course}[Chapter 2]). Under this geometric framework, the signature method has become a powerful tool in data science as it helps to study the properties of a data stream (e.g. extraction of characteristic features from the data, \cite{levin2013learning}) and answer many questions associated to data-driven problems \cite{ chevyrev2016primer}. The questions may be related to finding patterns in the data and approximating missing information. In machine learning, some recent applications of signatures are image and texture classification using 2D signatures \cite{zhang2022two}. Also, the sensitivity of signature to the geometric structure of data has made it particularly effective in applications such as Chinese character recognition. For instance, \cite{graham2013sparse} reported a test error of 3.58\% using a sparse signature-based model, outperforming the 5.61\% test error obtained using traditional convolutional neural networks (CNNs) \cite{schmidhuber2012multi}. In mathematical finance, {\it signature models} have been employed in various applications, including the pricing of path-dependent options—also known as signature payoffs \cite{arribas2018derivatives}—model calibration using such payoffs \cite{cuchiero2023signature, cuchiero2025joint}, and the construction of cubature formulas on Wiener space \cite{lyons2004cubature}. Additional applications include optimal execution \cite{kalsi2020optimal}, optimal stopping \cite{bayer2023optimal}, and stochastic optimal control \cite{bank2024stochastic}. 
Moreover, signatures have been employed in generating synthetic data \cite{kidger2019deep} and in topological data analysis \cite{chevyrev2018persistence}. 

Nonetheless, in real-world data-driven settings, the underlying data often possess intrinsic manifold structures of much lower dimension \cite{tenenbaum2000global, roweis2000nonlinear, belkin2003laplacian,fefferman2016testing, pope2021intrinsic}, leading to manifold-valued paths that generally fail to satisfy the geometric property \eqref{IBP} (see e.g.: \cite{armstrong2022non}). In fact, the geometric property—such as the integration by parts identity \eqref{IBP}—does not hold in general for arbitrary paths \(\bX:[0,T] \to \RR^d\). A prominent example is the  Brownian motion in the It\^o's integration form (see e.g.: \cite{friz2014course}[Chapter 2]). The lack of the geometric property in the classical signature suggests the need for an alternative framework capable of handling such cases. To address this, we borrow the concept of the branched signature, which extends the classical notion by accounting for paths whose signatures typically do not satisfy the geometric property. This notion is naturally associated with branched (or non-geometric) rough paths \cite{gubinelli2010ramification}. Accordingly, we propose a new modeling framework termed the {\it branched signature model}, designed to accommodate such irregular or manifold-valued data.

As a toy example, consider the task of estimating functionals of underlying processes whose sample paths do not satisfy the geometric property. In such cases, the classical signature framework becomes inadequate, as it inherently relies on this geometric structure. The branched signature, by contrast, offers a more general representation that aligns with the theory of non-geometric rough paths, enabling the treatment of a broader class of stochastic and manifold-valued paths. Consider a function \(F: \RR^2 \to \RR\) that depends on two underlying processes $(\bX^1_t,\bX_t^2)$ that satisfy the following controlled differential equations driven by are two signal processes \(\xi_i:[0,T] \to \RR\), for i=1,2, in \(\RR\), 
\begin{align}
\begin{cases}
d\bX^1_t &= V_1(\bX^1_t, \bX^2_t)d \xi^1_t,\\
d\bX^2_t &=  V_2(\bX^1_t, \bX^2_t)d \xi^2_t,
\end{cases}
\label{intro: X dynamics}
\end{align}
where \( V_i: \RR^2 \to \RR\), $i=1,2$, are smooth functions. Applying Taylor's expansion around $\bX^{i}_s$ with $t>s$, for  $i=1,2$, multiple times gives the following approximation for the underlying process,
\begin{align}
\bX^i_t - \bX^i_s \approx \text{LOT} + \sum_{j,k=1}^{2} C\left( V_1,  V_2, DV_1, DV_2\right) \int_s^t \left( \int_s^v d\xi^j_u \right) \left( \int_s^v d\xi^k_u \right) d\xi^n_v + \text{HOT}, \nonumber
\end{align}
where LOT ( and HOT respectively) stands for lower order terms ( and higher order terms respectively) and \(C\) is a function of \(V_1, V_2\) or any of their derivatives $(DV_1, DV_2)$ evaluated at initial point $(X_s^1,X_s^2)$. The function $F$ of the two underlying processes \eqref{intro: X dynamics} can be approximated as follows,
\begin{align}
&F(X^1_t,X^2_t) - F(X^1_s,X^2_s)\nonumber \\
\approx&  \text{LOT} + \sum_{i,j,k=1}^{2}  \tilde{C} \left(F, V_1, V_2, DF, DV_1, DV_2 \right) \int_s^t \left(\int_s^v d \xi^i_u\right) \left(\int_s^v d \xi^j_u\right) d \xi^k_v + \text{HOT},
\label{expansion of F}
\end{align}
where \(\tilde{C}\) is a function of \(F, V_1, V_2\) or any of their derivatives. It can be observed that the right-hand side of \eqref{expansion of F}, which approximates 
$F$, contains a branched-type term \(\int_s^t \left(\int_s^v d \xi^i_u\right) \left(\int_s^v d \xi^j_u\right) d \xi^k_v\).
Such a structure does not belong to the classical geometric (shuffle) signature, but instead arises naturally in the context of non-geometric rough paths and is represented in the {\it branched signature} framework. This motivates the definition of the {\it branched signature}, an object that extends the classical signature by including not only all iterated integrals but also their possible products, thereby capturing the full range of terms arising in non-geometric rough paths.
These integrals can be represented using a rooted tree structure (see e.g. \cite{gubinelli2010ramification, hairer2015geometric}), where each tree encodes the combinatorial structure of the corresponding iterated integrals and their products. With a slight abuse of notation, we will write $\bsig(\bX)_{st}$ to denote the full branched signature, i.e., the collection of all possible tree-indexed integrals over $[s,t]$. Then the branched signature of level 
$N$ consists of all tree-indexed integrals and can be expressed in the form
\begin{align}
\bsig^N(\bX)_{st} = \left\{ \innerprod{\bsig(\bX)_{st},\tau}, \tau\in \mathcal T, |\tau| \le N \right\},\nonumber 
\end{align}
where $\mathcal T$  denotes the set of rooted trees, and 
\(\tau\)  is the number of nodes in $\tau$. 
For example, for \(\tau = \Forest{decorated[k[i][j]]}\),  we have $\innerprod{\bsig(\bX)_{st}, \Forest{decorated[k[i][j]]}} = \int_s^t \left(\int_s^{v} d\bX_{u}^{\Forest{decorated[i]}}  \right) \left( \int_s^{v} d\bX_{u}^{\Forest{decorated[j]}} \right)d\bX_{v}^{\Forest{decorated[k]}}$. 
These rooted trees are endowed with the Connes–Kreimer Hopf algebra structure \cite{gubinelli2010ramification}, which is a particular example of a Hopf algebra \cite{manchon2008hopf}. Although there is extensive literature on Hopf algebras and on the Connes–Kreimer Hopf algebra of rooted trees, we will discuss both algebraic structures in detail in the next section to keep the manuscript self-contained. We also refer the reader to the next section for a comprehensive introduction to branched signature models.

\textbf{Main results.} After introducing the {\it branched signature model}, we first establish the universal approximation property for the {\it branched signature model}. A key property of the classical signature model is its universal approximation property, where the geometric nature of classical signatures plays a central role in the proof. Unfortunately, such a geometric property is not available for {\it branched signatures}. To establish a version of such an universal approximation property for {\it branched signature models}, we used the idea of extended path from \cite{hairer2015geometric}, which maps the original paths \(\bX: [0,T] \to \RR^d\) to an extended path \(\Bar{\bX}: [0,T] \to \RR^e\), where \(e\) is much bigger than \(d\). For such an extended paths  \(\Bar{\bX}\) over a given path \(\bX\), for any rooted tree $h$, the {\it branched signature} component corresponding to $h$ coincides with the classical signature component corresponding to a basis element $\Psi(h)$ of a suitable tensor algebra (i.e. Hopf algebra), which can be represented as ,
\begin{equation*}
        \innerprod{\bsig(\bX)_{st}, h} = \innerprod{\sig(\Bar{\bX})_{st}, \Psi(h)}.
\end{equation*}
The existence of such a path and the map $\Psi$ is established in \cite{hairer2015geometric}, although their uniqueness remains unknown and the explicit construction of $\Psi$ is not provided.  Nevertheless, this existence result is sufficient for many applications, as the extended path inherits several desirable properties, such as satisfying Chen’s identity, and allowing the use of classical signature tools for analysis and approximation.

Building on this idea, we next present an explicit construction of the extended paths $\Bar{\bX}$
  and introduce a dimension-reduction algorithm to mitigate the computational complexity. The extended path \(\Bar{X}\) is significantly higher-dimensional than the original path \(X\). As a result, computing the classical signature for $\sig(\Bar{\bX})$ up to a given level $N$ can be computationally expensive. 
  To address this, we adopt a strategy in which a lower-order signature is computed repeatedly. Specifically, if we regard the lower-order signature of order $k$ as a map, we can compose this map multiple times to obtain the signature up to level 
$N$. This algorithm will be made rigorous in the following sections. An illustration of the idea for 
$k=2$ is given below in  Figure~\ref{fig1-intro-1},
\begin{center}
\begin{figure}[h]
\begin{tikzpicture}[
SIR/.style={rectangle, rounded corners,  draw=red!60, fill=red!5, very thick, minimum size=10mm},
]
\node[SIR]    (A1)                     {$\bX_{t}:$ Data};
\node[SIR]    (A2)       [right=of A1] {$\Bar{\bX}_{t}$};
\node[SIR]    (A3)       [right= 1.5cm of A2] {$\mathbf{S}_{\ell}^2(\Bar{\bX})_{0t}$};
\node[SIR]    (A4)       [right= 1.5cm of A3] {$\mathbf{S}^2_{\ell^{(m)}}\left(\mathbf{S}^2_{\ell^{(m-1)}}\left(\dots \mathbf{S}^2_{\ell^{(1)}}\left(\bar{\bX}\right)_{0t}\right)_{0t}\right)_{0t}$};
\draw[->, thick] (A1.east)  to node[above] {$\Psi$} (A2.west);
\draw[->, thick] (A2.east)  to node[above] {$\mathbf{S}_{\ell}^2(\cdot)$} (A3.west);
\draw[->, thick, dashed] (A3.east)  to node[above] {$\mathbf{S}_{\ell}^2(\cdot)$} (A4.west);
\end{tikzpicture}
\caption{Application of Level 2 signature model on extended path $\Bar{\bX}$}
\label{fig1-intro-1}
\end{figure}
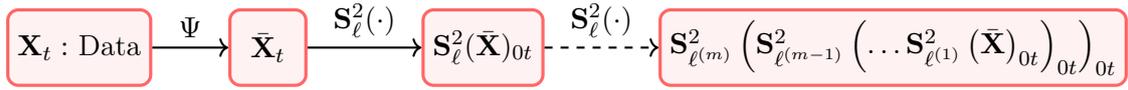
\end{center}
where we denote $\mathbf{S}_{\ell^{(1)}}^2(\bX)_{0t}$ as the second-level signature model paths for the extended path $\Bar{\bX}_t$ and denote $\ell^{(i)}$ as the $i$-th layer signature basis coefficients for $i=1,\cdots,m$. By repeating such second-level signature map $m$ times, we can reach  the desired level $N$ signature model with much lower dimension complexity.

The paper is organized as follows. In Section \ref{section pre}, we introduce the necessary preliminaries for signature and branched signature models. In Section \ref{section universal}, we establish the universal approximation property for
branched signature models, and an iterative version of the branched signature approximation. In Section \ref{section extension}, we provide a constructive method to construct the extended path $\bar{\bX}(t)$ using the map $\Psi$ and provide explicit examples. In Section \ref{section numerical}, we apply our branched signature model to calibrate the stochastic volatility model  with a mixture of estimated driving Brownian motion and fractional Brownian motion from the observational data. In Section \ref{optimal execution section}, we study the problem of optimal execution with branched signatures built up directly from the observed data.

\section{Preliminaries}\label{section pre}
To keep the manuscript self-contained, let us introduce some terminology, discuss the relevant algebraic structures, and define the geometric and non-geometric(branched) rough paths. As we are working with \(d-\)dimensional paths, so we consider the set of underlying letters to be \(\cS = \{1,2,\dots,d\}\) and we call this the alphabet set. Also, we define a \textbf{word} \(\bw\) of length \(|\bw| = n\) to be a sequence \(\bw = w_1 w_2 \dots w_n\) where \(w_i \in \cS\) for \(i = 1,2,\dots, n\). With this, we denote \(\bW\) to be the set of all words and \(\bW_n\) to be the set of all words of length \(n\). Also, for \(n=0\) we define \(\bW_0\) to be the set with empty word \(\emptyset\) . We denote the vector space generated by \(\bW\) as \(\VV (\cS)\) defined as follows
\[\VV (\cS) = \left\{ \chi = \sum_{\bw \in \bW} c_{\bw} \bw | c_{\bw} \in \RR , c_{\bw} \ne 0 \text{ for finitely many } \bw \in \bW\right\}.\]
We also define the concatenation of two words \(\bw = w_1 w_2 \dots w_n\) and \(\bw' = w'_1 w'_2 \dots w'_m\) to be \(\bw \bw' = w_1 w_2 \dots w_n w'_1 w'_2 \dots w'_m\).
\begin{defn}
The set \(\VV (\cS)\) endowed with product defined as the concatenation of words is a non-commutative algebra.
\end{defn}
In the next subsection, we will define signature of a path and discuss its key properties. Though the signature of path is defined by using the idea of alphabet and words that we introduced earlier, yet we define another more general algebraic structure where signature actually lives. Firstly, we define the \textbf{extended tensor algebra} over \(\RR^d\) to be the space
\[T((\RR^d)) =  \prod_{n=0}^{\infty} (\RR^d)^{\otimes n} = \{\bv = (v_0, v_1, \dots , v_n, \dots ) \; | \; v_n \in (\RR^d)^{\otimes n}, n = 0, 1,\dots \},\]
where \((\RR^d)^{\otimes 0} := \RR\).
Another space, closely related to this one and the one generated by words in \textbf{tensor algebra} over \(\RR^d\) that is defined as
\[T(\RR^d) =  \bigoplus_{n=0}^{\infty} (\RR^d)^{\otimes n} = \{ \bv \in T((\RR^d)) \; | \; \forall \;  \bv \; \exists \;  K \in \NN \text{ such that } v_n = 0 \; \forall \; n \ge K\}.\]
Also, the \textbf{truncated tensor algebra} over \(\RR^d\) is defined as
\[T^N(\RR^d) := \left\{ \mathbf{v} \in T(\RR^d) \; | \; v_n = 0 \; \forall \;  n > N\right\}.\]
It can be easily shown that \(\VV (\cS)\) can be viewed as space of linear functionals on extended tensor algebra \(T((\RR^d))\) that is 
\[ \VV (\cS) \cong  T((\RR^d))^{*}.\]
\subsection{Signature and its properties}
Corresponding to a word \(\bw = w_1 \dots w_n\) in \(\bW \subsetneq \VV (\cS) \) we have the following definition
\begin{defn}
The signature of a continuous \(\RR^d\)-valued path of bounded variation \((\bX_t)_{t \in [0,T]}\) is the \(T((\RR^d))\)-valued process \( (s,t) \in \Delta_T^2 \mapsto \sig(\bX)_{st} \in T((\RR^d))\) whose component corresponding to each word \(\bw = w_1 \dots w_n \in \bW\) is defined as
\[\innerprod{\sig(\bX)_{st}, \bw} := \int_{s}^{t} \int_{s}^{t_k} \cdots \int_{s}^{t_2} d\bX^{w_1}_{t_1}\cdots d\bX^{w_n}_{t_n},\]
where \(\Delta_T^2 := \{ (s,t) \in [0,T]^2 : 0 \leq s \leq t \leq T \}\). Similarly, for an empty word \(\emptyset\) we define \(\innerprod{\sig(\bX)_{st}, \emptyset} := 1\).
To be precise, the signature can be identified as an infinite dimensional object given as
\[\sig(\bX)_{st} = \left(1, \int_s^t d\bX^{w_i}_r, \int_s^t \int_s^{r_2} d\bX^{w_i}_{r_1} d\bX^{w_j}_{r_2}, \int_s^t \int_s^{r_3}\int_s^{r_2} d\bX^{w_i}_{r_1} d\bX^{w_j}_{r_2} d\bX^{w_k}_{r_3}, \cdots \right)_{w_i,w_j,w_k,  \dots \in \cS}.\]
\end{defn}
Furthermore, the level \(N\) truncation of the signature is given as below
\[\sig^N(\bX)_{st} = \sum_{\bw \in \bW : |\bw| \le N} \innerprod{\sig(\bX)_{st}, \bw} e_{\bw},\]
where \(e_{\bw}\) is a basis element of \(T((\RR^d))\).\\

\textbf{Notation:} Throughout, for 
\(s=0\) we employ the following convention \(\sig(\bX)_{0t} =  \sig(\bX)_{t}\) and \(\sig^N(\bX)_{0t} = \sig^N(\bX)_{t}\).

\begin{rem}
For a path of bounded 1-variation, all iterated integrals can be defined in the sense of Riemann--Stieltjes integration. For a path of bounded \( p \)-variation with \( p \in (1,2] \), integrals can be defined in the sense of Young. However, for a path of bounded \( p \)-variation with \( p > 2 \), integrals cannot, in general, be defined using either Riemann--Stieltjes or Young integration. In such cases, the existence and interpretation of the integral depend on the nature of the path itself. For detailed constructions, see \cite{Friz_Victoir_2010}.
\end{rem}

\begin{rem}
As two foundamental examples, Brownian motion has bounded \( p \)-variation for all \( p > 2 \), yet integration with respect to it can be defined in the It\^o or Stratonovich sense. Similarly, for fractional Brownian motion with Hurst parameter \( H \in \left(0, 1\right) \), integration can be formulated in the Skorohod (or divergence) sense (see \cite{biagini2008stochastic})using tools from Malliavin calculus. Unless otherwise stated, we will be working throughout with  \(\alpha\)-H\"older paths for any \(\alpha > 0\) as fractional Brownian motion with Hurst parameter \(H\) is \(\alpha\)-H\"older for \(\alpha = H - \varepsilon\) for arbitrarily small \(\varepsilon>0\). But, for the purpose of construction of the rough path and corresponding topology we may restrict ourselves to \(\alpha > \frac{1}{4}\) sometimes.
\end{rem}

Let us discuss some properties of the signature of a bounded \( p \)-variation path, as \(\alpha\)-H\"older paths have bounded \( p \)-variation for \(p > \lfloor \frac{1}{\alpha} \rfloor\). These properties are not universal i.e., they may fail to hold for arbitrary \(p\) or depending on the definition of iterated integrals. For instance, one such property is an extension of the classical integration by parts formula, which does not hold when the path is a Brownian motion and the iterated integrals are defined in the It\^o sense. To formalize this integration by parts property in the context of bounded variation paths, we introduce the \emph{shuffle product} on set of words \(\bW\). This product also encodes the algebraic structure of the space \(T((\mathbb{R}^d))\) using its basis elements 
\(e_{\bw}\) corresponding to each word \(\bw\).

\begin{defn}
For two words \(\bw = w_1 \dots w_n\) and \(\bw' = w'_1 \dots w'_m\), the shuffle product is defined recursively as follows
\[\bw \shuffle \emptyset = \emptyset \shuffle \bw = \bw, \quad \text{and}\]
\[\bw \shuffle \bw' = \left[(w_1 \dots w_{n-1}) \shuffle \bw' \right] w_n + \left[\bw \shuffle (w'_1 \dots w'_{m-1}) \right] w'_m,\]
where \(\emptyset\) is an empty word.
\end{defn}
Similarly, we define the shuffle product of the basis elements of \(T(\RR^d)\) by setting \(e_{\bw} = e_{w_1} \otimes \dots \otimes e_{w_n}\). Thus, for any two \(\mathbf{u},\mathbf{v} \in T(\RR^d)\) we have 
\[\mathbf{u} \shuffle \mathbf{v} = \sum_{|\bw|,|\bw'| \geq 0} u_{\bw} v_{\bw'} (e_{\bw} \shuffle e_{\bw'}).\]
\

Following the above definition, if we endow the space \(T(\RR^d)\) with the shuffle product \(\shuffle\) then the quadruple \(\left(T(\RR^d), +, \cdot, \shuffle \right)\) is an associative algebra.

The shuffle product plays a crucial role in encoding the multiplicative structure of the signature. In particular, it allows us to express products of iterated integrals as linear combinations of other iterated integrals. This leads to the following fundamental identity satisfied by the signature of a bounded variation path, known as the \emph{shuffle property}.

\begin{defn}[Shuffle Property] \cite{lyons2007differential} Let \((\bX_t)_{t \in [0,T]}\) be a continuous \(\RR^d\)-valued path of bounded variation and \(\bw = w_1 \dots w_n\) and \(\bw' = w'_1 \dots w'_m\) be two words then
\[\innerprod{\sig(\bX)_{st}, \bw} \innerprod{\sig(\bX)_{st}, \bw'} = \innerprod{\sig(\bX)_{st}, \bw \shuffle \bw'}.\]
\begin{rem} \label{p_value_remark}
Shuffle property also holds when the path is of bounded \(p\)-variation for \(p \le 2\) and integrals are defined in Young's sense. It is also valid for Brownian  motion and fractional Brownian motion with Hurst \(H > \frac{1}{4}\) when integrals are defined in the Stratonovich sense.
\end{rem}
We now illustrate how the shuffle property reduces to the classical integration by parts identity when \(|\bw| = 1\) and \(|\bw'| = 1\).
\end{defn}
\begin{example}[Integration by parts]
For a continuous \(\RR^d\)-valued path \((\bX_t)_{t \in [0,T]}\) with \(\bX_0 = 0\) and for \(w_i,w_j \in \{1,2, \cdots, d\}\) integration by parts property is
\[\bX_T^{w_i} \bX_T^{w_j} = \int_0^T \bX_t^{w_i} d\bX_t^{w_j} + \int_0^T \bX_t^{w_j} d\bX_t^{w_i}.\]
Therefore,
\begin{align*}
\innerprod{\sig(\bX)_{T}, w_i} \innerprod{\sig(\bX)_{T}, w_j} & = \innerprod{\sig(\bX)_{T}, w_i  w_j} +  \innerprod{\sig(\bX)_{T}, w_j  w_i} = \innerprod{\sig(\bX)_{T}, w_i \shuffle w_j}.
\end{align*}
\end{example}
Hence, integration by parts appears as a special case of the shuffle property. In addition to this, the signature of a \(\mathbb{R}\)-valued path also resembles the structure of the Taylor series basis, as we discuss next.
\begin{example}
For a continuous \(\RR\)-valued path of bounded variation \((\bX_t)_{t \in [0,T]}\), we have \(\innerprod{\sig(\bX)_T, w_1} = \bX_T - \bX_0\). Using the identity
\[
\underbrace{w_1 \shuffle \dots \shuffle w_1}_{k\text{-times}} = k! \, w_1  \dots  w_1,
\]
we can deduce that
\[
\sig(\bX)_T = \left(1, \bX_T - \bX_0 , \frac{1}{2!}(\bX_T - \bX_0)^2, \frac{1}{3!}(\bX_T - \bX_0)^3, \cdots, \frac{1}{k!}(\bX_T - \bX_0)^k, \cdots \right).
\]
\end{example}
\begin{rem}
The statement holds again when paths and integrals are defined as in Remark \ref{p_value_remark} concerning the shuffle product.
\end{rem}
The signature of a data stream is defined as the signature of the piecewise linear interpolation between its data points. From a computational standpoint, this means we compute the signature over each small interval between consecutive data points. To obtain the signature over longer intervals, we iteratively apply a fundamental algebraic rule known as Chen's identity, which describes how to combine signatures over adjacent intervals. We now formally state this identity.
\begin{prop}[Chen's Identity]
\cite{Friz_Victoir_2010}
Let \((\bX_t)_{t \in [0,T]}\) be a continuous, \(\mathbb{R}^d\)-valued path of bounded variation. Then, the concatenated signature over intervals \([s, u]\) and \([u, t]\) satisfies
\[\sig(\bX)_{st} = \sig(\bX)_{su} \otimes \sig(\bX)_{ut},\]
for each \(0 \leq s \leq u \leq t \leq T\). This identity can be equivalently expressed as follows:
\[
\innerprod{\sig(\bX)_{st}, \bw} = \sum_{\bw_1 \bw_2 = \bw} \innerprod{\sig(\bX)_{su}, \bw_1} \innerprod{\sig(\bX)_{ut}, \bw_2} ,
\]
where \(\bw\) is an arbitrary word from \(\bW\).
\end{prop}

Furthermore, speaking naively, if two functions are equal then their integrals are equal too but the converse is not true in general. Same is true in the case of signatures that is if two paths are equal then their signatures are equal too but the converse is not true in general. However, if the signatures of two paths are equal then the paths are equal up to tree-like equivalence. This is proved by the authors of \cite{hambly2010uniqueness} in the case of weakly geometric rough paths. For the universal approximation theorem we need paths to be equal in a more restrictive sense. To acquire this result, we enhance the path \(\bX:[0,T] \to \RR^d\)  with an additional time component and denote it by \(\widehat{\bX}:[0,T] \to \RR^{d+1}\) defined by \(\widehat{\bX}:= (t, \bX)\). For time extended semi-martingales, the equality of the paths from the terminal signatures is proved by \cite{cuchiero2023signature}. We state and prove our result as follows.
\begin{lem}[Uniqueness of the classical signature] \label{sig_uiqueness}
Let $\bX,\bY:[0,T]\to\mathbb R^d$ be continuous $\alpha$-H\"older paths with $\bX_0=\bY_0=0$ for some $\alpha >  \frac{1}{4}$.
Form the time-augmented paths $\widehat{\bX}(t):=(t,\bX_t)$ and $\widehat{\bY}(t):=(t,\bY_t)$ in $\mathbb R^{1+d}$.
Assume their (geometric) terminal signatures coincide at all levels i.e.,
\[
\sig(\widehat{\bX})_{0,T}=\sig(\widehat{\bY})_{0,T}.
\]
Then $\bX_t=\bY_t$ for all $t\in[0,T]$.
\end{lem}

\begin{proof}
Fix a spatial index $i\in\{1,\dots,d\}$ and set $\bZ:=\bX^i-\bY^i$, a continuous $\alpha$-H\"older function with $\bZ_0=0$.
We use the family of signature coordinates of the time-augmented path that contain exactly one spatial letter. Clearly, for every $k,m\in\mathbb N\cup\{\mathbf{0}\}$,
\begin{equation} \label{eq:one-spatial-letter}
\big\langle \sig(\widehat{\bZ})_{0,T},\,\mathbf{0}^k\, i\, \mathbf{0}^m\big\rangle
=\frac{1}{k!\,m!}\int_0^T s^{\,k}(T-s)^{\,m}\,d\bZ_s ,
\end{equation}
where the integral is defined in the Young's sense. This is trivial when $\bZ$ is smooth; for general $\alpha$-H\"older $\bZ$, take smooth approximations $\bZ^n\to \bZ$ in $C^\alpha$, use the classical identity for $\bZ^n$, and pass to the limit: the map $\bZ\mapsto\int s^{k}(T-s)^m\,d\bZ$ is continuous in $C^\alpha$, and the one-spatial-letter signature coordinates are defined by the same limiting procedure.
By the hypothesis $\sig(\widehat{\bX})_{0,T}=\sig(\widehat{\bY})_{0,T}$, identity \eqref{eq:one-spatial-letter} applied to $\bZ=\bX^i-\bY^i$ yields, for all $k,m \ge 0$,
\begin{equation}\label{eq:weighted-Young-zero}
\int_0^T s^{\,k}(T-s)^{\,m}\,d\bZ_s=0 .
\end{equation}
Now choose $m=1$ and $k\ge0$, and set $\phi_k(s):=s^{\,k}(T-s)$.
Since $\phi_k\in C^1$ and $\bZ\in C^\alpha$ with $\alpha>0$, Young integration by parts gives
\[
\int_0^T \phi_k\,d\bZ
= \phi_k(T)\bZ_T-\phi_k(0)\bZ_0 - \int_0^T \bZ(s)\,\phi_k'(s)\,ds
= -\int_0^T \bZ(s)\,\phi_k'(s)\,ds,
\]
because $\phi_k(T)=0$ and $\bZ_0=0$.
By \eqref{eq:weighted-Young-zero} (with $m=1$), the left-hand side vanishes, so for every $k\ge0$,
\begin{equation}\label{eq:leb-zero}
\int_0^T \bZ(s)\,\phi_k'(s)\,ds=0.
\end{equation}
Note $\phi_k'(s)=kT\,s^{k-1}-(k+1)s^k$ is a polynomial with $\phi_0'(s)=-1$.
We claim that the linear span of $\{\phi_k'\}_{k\ge0}$ is the space of all polynomials on $[0,T]$.
Hence \eqref{eq:leb-zero} implies
\[
\int_0^T \bZ(s)\,p(s)\,ds=0,\qquad\text{for every polynomial }p.
\]
Polynomials are dense in $C([0,T])$, and the map $p \mapsto\int_0^T \bZ p$ is continuous in the sup-norm, so
\[
\int_0^T \bZ(s)\,\psi(s)\,ds=0,\qquad\text{for every } \psi \in C([0,T]).
\]
Taking $\psi=\bZ$ gives $\int_0^T |\bZ(s)|^2\,ds=0$, hence $\bZ\equiv0$ on $[0,T]$.
Since this holds for each component $i$, we conclude $\bX \equiv \bY$.
\end{proof}

Finally, we state the universal approximation theorem (UAT) based on classical signature. The main idea of UAT is to approximate the quantity of the form
$$
f\left( (\sig^p(\widehat{\bX})_t)_{t \in [0,T]} \right),
$$
for some $p$, where $f$ is a continuous function, by a linear functional on the full signature, i.e., a quantity of the form \(\innerprod{\sig(\widehat{\bX})_T, \ell}\), where $\ell \in T(\mathbb{R}^d)$. For \(p \in [1,3)\), a universal approximation theorem for weakly geometric c{\`a}dl{\`a}g rough paths is proved in \cite{cuchiero2025universal}. Also, a universal approximation theorem for time-extended semi-martingales is proved in \cite{cuchiero2023signature}, where the rough path signature is constructed via Stratonovich integrals of the semi-martingale process. Let us state and prove our result as follows.

\begin{thm}[UAT for classical signatures of time-extended \(\alpha\)-H\"older paths]\label{thm:UAT-classical}
Let \(\alpha > \frac{1}{4}\) be the H\"older regularity of the path \(\bX\), set \(p=\lfloor 1/\alpha\rfloor\).
Let \(\mathbb{G}^{p}(\mathbb R^{1+d})\) be the step-\(p\) nilpotent Lie group over the alphabet \(\cS = \{0,1,\dots,d\}\) (with \(0\) the time letter), and write \(\langle\cdot,\cdot\rangle\) for the pairing with words of length from \(\bW\).
Define
\[
\mathcal S^{(p)}:=\Big\{\, \sig^{p}(\widehat{\bX})_{t \in [0,T]}\;:\; \bX\in C^\alpha([0,T];\mathbb R^d)\,\Big\}
\subset C\!\big([0,T],\,\mathbb{G}^{(p)}(\mathbb R^{1+d})\big).
\]
Let \(\cH \subset \mathcal S^{(p)}\) be compact and \(f:H\to\mathbb R\) continuous.
Then for every \(\varepsilon>0\) there exists \(\ell\in T^{p}(\mathbb R^{1+d})\) such that
\[
\sup \limits_{(\sig^p(\widehat{\bX})_t)_{t \in [0,T]} \in \cH} \Big|\,f\left((\sig^p(\widehat{\bX})_t(_{t \in [0,T]}\right)-\innerprod{\sig(\widehat{\bX})_T, \ell} \,\Big| <\varepsilon .
\]
\end{thm}

\begin{proof}
Consider the set
\[
\mathcal G:= \mathrm{span}\Big\{\innerprod{\sig(\widehat{\bX})_T, \bw}, \bw \text{ is a word from } \bW \Big\}.
\]
Then \(\mathcal G\) is a unital subalgebra i.e., the empty word gives the constant \(1\), and for words \(\bu, \bv\) the shuffle identity yields
\[
\innerprod{\sig(\widehat{\bX})_T, \bu}\innerprod{\sig(\widehat{\bX})_T, \bv} = \innerprod{\sig(\widehat{\bX})_T, \bu \shuffle \bv} \in \mathcal G.
\]
Also, \(\mathcal G\) vanishes nowhere because \(\innerprod{\sig(\widehat{\bX})_T, \emptyset} \equiv 1\). Finally, \(\mathcal G\) separates points i.e., for any two \(\alpha\)-H\"older paths \(\widehat{\bX}\) and \(\widehat{\bY}\) with \(\widehat{\bX} \ne \widehat{\bY}\) implies \(\innerprod{\sig(\widehat{\bX})_T, \bw} \ne \innerprod{\sig(\widehat{\bY})_T, \bw}\) for any word \(\bw \in \bW\). On contrary, suppose \(\innerprod{\sig(\widehat{\bX})_T, \bw} = \innerprod{\sig(\widehat{\bY})_T, \bw}\) then by uniqueness of Lyon's lift [Theorem 9.5 (i) \cite{Friz_Victoir_2010}] \(\innerprod{\sig^p(\widehat{\bX})_t, \bw} = \innerprod{\sig^p(\widehat{\bY})_t, \bw}\) for any \(t \in [0,T]\), because paths of H\"older regularity \(\alpha\) have bounded \(p\)-variation for \(p > \lfloor \frac{1}{\alpha}\rfloor\). Furthermore, if \(\innerprod{\sig(\widehat{\bX})_T, \bw} = \innerprod{\sig(\widehat{\bY})_T, \bw}\) then \(\widehat{\bX}_t = \widehat{\bY}_t\) for any \(t \in [0,T]\) by the uniqueness of the signature from Lemma \ref{sig_uiqueness}, which is a contradiction to original claim. Therefore \(\mathcal G\) separates points. Hence, the claim follows by Stone-Weierstrass theorem.
\end{proof}

The next subsection provides a brief formal introduction to geometric rough paths. We keep this discussion very brief and then proceed to define branched rough paths in the forthcoming subsection, explore their various properties, introduce the corresponding algebraic structure, and present illustrative examples.

\subsection{Geometric Rough Path}
Having introduced some basic notions related to the path $\bX$ and its signature, which is an infinite dimensional object comprising the iterated integrals of the path $\bX$ in an increasing order, it's time to introduce geometric rough path as given by \cite{hairer2015geometric}. Theoretically, a rough path is an infinite dimensional object. But, in practice, only finitely many components of rough path actually matter. Let $\bX$ be an $\alpha-$H\"older path and let $N$ be the largest integer such that $N\alpha \le 1$, then the components of the path with degree $n > N$ can be determined by those of degree $n \le N$ [see e.g. \cite{Friz_Victoir_2010} Theorem 9.5]. With this, we formally define a geometric rough path as follows.
\begin{defn}
      A map $\sig^{N}(\bX) : [0, T] \times [0, T] \to T(\mathbb{R}^d)$ of regularity $\alpha$ is said to be a (weakly-) geometric rough path (GRP) if it satisfies:
      \begin{enumerate}
          \item $\langle \sig^{N}(\bX)_{st}, \bw \shuffle \bw' \rangle = \langle \sig^{N}(\bX)_{st}, \bw \rangle \langle \sig^{N}(\bX)_{st}, \bw' \rangle $, for each \(\bw, \bw' \in \bW\),
          \item $\sig^{N}(\bX)_{st} = \sig^{N}(\bX)_{su} \otimes \sig^{N}(\bX)_{ut}$,
          \item $\sup\limits_{s \neq t} \frac{\langle \sig^{N}(\bX)_{st}, \bw \rangle}{|t - s|^{\alpha |\bw|}} < \infty$, for every $\bw \in \bW$ with $|\bw| \leq N$.
      \end{enumerate}
\end{defn}

\begin{rem}
We adopt the same notation for signature and the geometric rough path that is $\sig^N(\bX)$ to talk about rough path and signature interchangeably as up to level \(N\) there is no difference in rough path and signature.
\end{rem}

\begin{rem}
The geometric rough path lives in the Lie group $(\mathbb{G}(\mathbb{R}^d), \otimes))$, which is called the free nilpotent group with the tensor product being the group multiplication. This free nilpotent group $\mathbb{G}(\mathbb{R}^d)$ is defined as
$$\mathbb{G}(\mathbb{R}^d) := \exp \left(\mathfrak{g}(\mathbb{R}^d)\right),$$
where $\mathfrak{g}(\mathbb{R}^d) \subset T(\mathbb{R}^d)$ is the formal Lie series of $\mathbb{R}^d$.
\end{rem}

As an illustration, for one-dimensional Brownian motion $\bB$, the Stratonovich lift gives rough path
\[
\sig^2(\bB)_{st} = \left(1, \int_s^t \circ d\bB_r, \int_s^t \int_s^v \circ d\bB_u \circ d\bB_v\right).
\]
By integration by parts,
\[
\innerprod{\sig^2(\bB)_{st}, \bw_1}^2 
= 2\int_s^t \bB_r \circ d\bB_r 
= 2\innerprod{\sig^2(\bB)_{st}, \bw_1 \bw_1}
= \innerprod{\sig^2(\bB)_{st}, \bw_1 \shuffle \bw_1},
\]
where \(\bw_1 = 1\), so the shuffle property holds.  
In contrast, for the It\^o lift
\[
\sig^2(\bB)_{st} = \left(1, \int_s^t d\bB_r, \int_s^t \int_s^v d\bB_u\, d\bB_v\right),
\]
and It\^o’s formula gives
\[
\innerprod{\sig^2(\bB)_{st}, \bw_1}^2
= 2\int_s^t \bB_r\, d\bB_r + (t-s) 
= 2\innerprod{\sig^2(\bB)_{st}, \bw_1  \bw_1} + (t-s),
\]
which differs from $\innerprod{\sig^2(\bB)_{st}, \bw_1 \shuffle \bw_1}$. Hence, the Stratonovich integral yields a geometric rough path, whereas the It\^o integral does not.

\subsection{Branched Rough Path}

Geometric rough path, though not encompassing Brownian motion with It\^o integrals, enjoys many good properties and has many applications in finance, machine learning and data science. And these all applications are due to universal approximation theorem (UAT) which can equivalently be stated that any continuous function of the path or signature can be approximated well by linear combination of the components of the signature that is
\[f (\sig(\bX)_{st}) \approx \sum_{\bw \in \bW} a_{\bw} \innerprod{\sig(\bX)_{st}, \bw},\]
where \(\bw \in \RR\). For the function \(f(x) = x^2\) and an \(\RR^d\)-valued path \(\bX\)

\[f(\innerprod{\sig(\bX)_{st}, w_i}) = \innerprod{\sig(\bX)_{st}, w_i}^2 = \innerprod{\sig(\bX)_{st}, w_i w_i} + \innerprod{\sig(\bX)_{st}, w_i w_i}, \text{ where } w_i \in \{1,\dots,d\}.\]

Similarly,
\[\int_s^t \innerprod{\sig(\bX)_{su}, w_i}\innerprod{\sig(\bX)_{su}, w_j} d\bX^{w_k}_u =  \innerprod{\sig(\bX)_{su}, \bw} + \innerprod{\sig(\bX)_{st}, \bw'}, \]
where \(\bw = w_i w_j w_k\) and \(\bw = w_j w_i w_k\).

On the other hand, if the rough path is not geometric then we do not have such an equality. Also, we know that the geometric rough path lives in the nilpotent Lie group which is associated with the tensor algebra over $\mathbb{R}^d$. While non-geometric (branched) rough path lives in a more general Lie group induced by the Connes-Kreimer Hopf algebra of rooted trees. To discuss this in detail, let us encode the components of the rough path into something different than classical tensor algebra i.e. rooted tree structure. For example the following integral is encoded as 

\[\int_s^t \left(\int_s^u d\bX^{i}_r\right) \left(\int_s^u d\bX^{j}_r\right) d\bX^{k}_u =: \innerprod{\bsig^N(\bX)_{st}, \Forest{decorated[k[i][j]]}},\]
where instead of words we will be using the trees to define the components of the signature/rough path and the vertices of the trees will be decorated from the alphabet set \(\cS = \{1,2,\dots,d\}\). In general, for a rooted tree \(h\) 
\[\innerprod{\bsig^N(\bX)_{st}, h} = \int_s^t \innerprod{\bsig^N(\bX)_{su}, h'} d\bX^{\br}_u, \]
where \(\br\) is the root of tree \(h\) and \(h'\) is the tree that we get after removing root from \(h\).

\begin{rem}
The adoption of the notation \(\bsig^N(\bX)\) instead of \(\sig^N(\bX)\) is to differentiate between geometric and branched rough path and signature.
\end{rem}

The tree structure gives rise to a space called the Connes-Kreimer Hopf algebra of rooted trees $\mathcal{T}$ which is a Hopf algebra of labelled, rooted trees with labels coming from the set $\{1, \cdots, d\}$. This special Hopf algbra is used in \cite{connes1999hopf} in the context of renormalization theory. To be precise, a Hopf algebra is a vector space equipped with a product 
$$\cdot : \hopf \hat{\otimes} \hopf \to \hopf,$$ and a coproduct 
$$\Delta : \hopf \to \hopf \hat{\otimes} \hopf.$$
 This product is the usual commutative product of polynomial where each tree in $\mathcal{T}$ is considered a monomial. The coproduct $\Delta$ is the dual of the convolution product $\star$ which is nothing but all the ways to cut apart a tree like the deconcatenation coproduct of tensors as given by \cite{manchon2008hopf}. A detailed introduction to Hopf algebra and the corresponding properties will be given later.\\
Let's define precisely what a branched rough path is as a reiteration of the definition given by \cite{gubinelli2010ramification}. 
\begin{defn}
    An \(\alpha-\)H\"older map \(\bsig^N(\bX) : [0, T] \times [0, T] \to \hopf^*\)(the graded dual of \(\hopf\))  is said to be a branched rough path if it satisfies the following three properties:
\begin{enumerate}
    \item \(\langle  \bsig^N(\bX)_{st}, h_1 h_2 \rangle = \langle  \bsig^N(\bX)_{st}, h_1 \rangle \langle \bsig^N(\bX)_{st}, h_2 \rangle\), for every \(h_1, h_2 \in \hopf\).
    \item \(\bsig^N(\bX)_{st} = \bsig^N(\bX)_{su} \star \bsig^N(\bX)_{ut}\) or equivalently
     \(\langle \bsig^N(\bX)_{st}, h \rangle = \sum\limits_{(h)} \langle \bsig^N(\bX)_{su}, h^{(1)} \rangle \langle \bsig^N(\bX)_{ut}, h^{(2)} \rangle\),
    where \(\Delta h = \sum\limits_{(h)} h^{(1)} \hat{\otimes} h^{(2)}\) and \(h \in \hopf\).
    \item \(\sup\limits_{s \neq t} \frac{\langle \bsig^N(\bX)_{st}, h \rangle}{|t - s|^{\alpha |h|}} < \infty\), for every \(h \in \hopf\), where \(|h|\) is the degree of \(h\).
\end{enumerate}
\end{defn}
We first recall the definition of a Hopf algebra and then specialize to the Connes–Kreimer Hopf algebra of rooted trees, which is the structure we need for branched rough paths.

\subsection{Hopf Algebra}
Since a Hopf algebra is a special case of a bialgebra, we begin by defining bialgebras. Consider vector spaces \(\hopf\) and \(\hopf^*\) with units \(\mathbf{1}\) and \(\mathbf{1}^*\), products \(\cdot : \hopf \hat{\otimes} \hopf \to \hopf \) and \(\star: \hopf^* \hat{\otimes} \hopf^* \to \hopf^* \) respectively. \(\hopf^*\) is considered as the dual space of \(\hopf\) where the action of functional is defined as \(\innerprod{\cdot,\cdot}: \hopf^* \hat{\otimes} \hopf \to \mathbb{R}\). The structure of \(\hopf^*\) can be superimposed to that of \(\hopf\) using the coproduct \(\Delta\) defined as  
\[ \innerprod{f \star g, h} = \innerprod{f \; \hat{\otimes}\; g, \Delta h},\]
where $f,g \in \hopf^*$, $h \in \hopf$, and $\Delta h = \sum_{(h)} h^{(1)} \; \hat{\otimes } \; h^{(2)}$. We will discuss more about this coproduct later. The definition of bialgbera is given as follows
\begin{defn}
The triple $\left(\hopf, \cdot, \Delta \right)$ is called a bialgebra if it satisfies the following compatibility conditions:
\begin{enumerate}
\item The coproduct $\Delta: \hopf \to \hopf \hat{\otimes} \hopf$ is an algebra homomorphism i.e., \(\Delta(h \cdot k) = \Delta(h) \cdot \Delta(k),  \text{ for all } h, k \in \hopf\) and \(\Delta(\mathbf{1}) = \mathbf{1} \otimes \mathbf{1}\).
\item The counit $\mathbf{1}^*: \hopf \to \RR$ is an algebra homomorphism i.e., \(\mathbf{1}^*(h \cdot k) = \mathbf{1}^*(h) \mathbf{1}^*(k)  \text{ for all } h, k \in \hopf\) and \( \mathbf{1}^*(\mathbf{1}) = 1\).
\end{enumerate}
\end{defn}

Moreover, the map \(\cA^* :  \hopf^* \to \hopf^*\) defined by \(f \star \cA^* f  = \cA^* f \star f = \mathbf{1}\), for \(f \in \hopf^*\) is called the inverse map. The adjoint of this map, \(\cA :  \hopf \to \hopf\) is called the antipode which satisfies the following relation
\begin{equation}
(\Id \; \hat{\otimes} \;\mathcal{A}) \Delta h = (\mathcal{A} \; \hat{\otimes} \; \Id ) \Delta h = \innerprod{\mathbf{1}^*,h}\mathbf{1},
\end{equation}
for $h \in \hopf$ and $\Id: \hopf \to \hopf$ is the identity map. 
\begin{defn}
The quadruple $\left( \hopf, \cdot, \Delta, \mathcal{A}\right)$ is called a Hopf algebra.
\end{defn}
Furthermore, a graded bialgebra is the one that can be decomposed into the direct sum of vector spaces, i.e.: 
$$\hopf = \bigoplus_{n \in \mathbb{N}} \hopf_{(n)}.$$ 
We introduce this grading to recall a fact: any graded bialgebra $\hopf$ with $\hopf_0 = \mathbb{R}$ is automatically a Hopf algebra. Moreover, every Hopf algbera has a unique antipode. For detailed discussions and examples on Hopf algebra, we refer to  \cite{abe2004hopf} and \cite{brouder2004trees}.
Next, we will discuss an special example of Hopf algebra, which is the Connes-Kreimer Hopf Algebra of rooted trees.
\subsubsection{The Connes-Kreimer Hopf Algebra of rooted trees}
The Connes-Kreimer Hopf Algebra is an special example of Hopf algebra that plays a key role in the theory of branched rough paths. It will serve as our primary algebraic framework in this context. Let us define some notations and discuss main properties of this key algebraic structure.

Let the set of all rooted trees (forests) with finite vertices be denoted by $\trees \; (\mathcal{F})$ and that with vertices up to $n$ be denoted by $\trees_n \; (\mathcal{F}_n)$. For example \(\trees_1 = \{\Forest{decorated[]} \}\;, \trees_2 = \{\Forest{decorated[]}, \Forest{decorated[[]]} \}\;, \trees_3 = \{\Forest{decorated[]}, \Forest{decorated[[]]}, \Forest{decorated[[[]]]}, \Forest{decorated[[][]]}\}\; \; \text{etc.}\)
All the trees above are undecorated but can be labeled with letters from some alphabet $\mathcal{S} = \{1,2, \cdots, d\}$.
The recursive construction of the trees is shown as follows
\begin{equation*}
    [\mathbf{1}]_i = \Forest{decorated[i]}, \quad [\Forest{decorated[i]}]_j = \Forest{decorated[j[i]]}, \quad [\Forest{decorated[j[i]]}]_k = \Forest{decorated[k[j[i]]]}, \quad [\Forest{decorated[i]}\Forest{decorated[j]}]_k = \Forest{decorated[k[i][j]]]}, \text{ etc.}
\end{equation*}
Here $\mathbf{1}$ refers to the empty tree. Indeed, every tree in $\mathcal{T}$ can be constructed recursively as \([h_1 h_2, \cdots h_n]_r,\)
for $h_1,  h_2, \cdots h_n \in \mathcal{T} \cup \mathbf{1}$. Furthermore, we will assume that the order of the branches in a tree does not matter i.e., $[h_1 h_2, \cdots h_n]_r = [h_{\sigma(1)} h_{\sigma(2)}, \cdots h_{\sigma(n)}]_r$ for any permutation $\sigma$.

In the case of rooted trees, the Connes-Kreimer Hopf algebra $(\hopf, \cdot, \Delta, \mathcal{S})$ is simply the commutative polynomial algebra generated by the variables coming from the set $\mathcal{T}$. It is equipped with an antipode $\mathcal{A}: \hopf \to \hopf$ and a coproduct $\Delta: \hopf \to \hopf \hat{\otimes} \hopf$. An example of an element of $\hopf$ is \(\Forest{decorated[k[j]]} - 5  \; \Forest{decorated[j]} \; \Forest{decorated[k[j[i]]]} -\frac{\sqrt{3}}{2} \Forest{decorated[i[k][j]]}.\)
The coproduct $\Delta$ can be recursively defined as $\Delta \mathbf{1} = \mathbf{1} \hat{\otimes} \mathbf{1} $ and
\begin{equation}\label{coproduct h}
    \Delta [h_1, \cdots, h_n]_r = [h_1, \cdots, h_n]_r \hat{\otimes} \mathbf{1} + \left(\Id \hat{\otimes} [\; \cdot \; ]_r\right)\Delta (h_1, \cdots, h_n).
\end{equation}
This coproduct is a morphism with respect to polynomial multiplication i.e., \(\Delta (h_1, \cdots, h_n) = \Delta h_1 \cdots \Delta h_n,\)
and is coassociative i.e., \((\Delta \hat{\otimes} \Id) \Delta  = (\Id \hat{\otimes} \Delta) \Delta.\)
Moreover, the antipode $\mathcal{A}$ satisfies \(P((\mathcal{A} \hat{\otimes} \Id) \Delta h) = P((\Id \hat{\otimes} \mathcal{A}) \Delta h) = \innerprod{\mathbf{1}^*, h} \mathbf{1},\) for any \(h \in \hopf\). Here $P$ is product map i.e., $P(h_1 \hat{\otimes} h_2) = h_1h_2$.

\begin{rem}
    If $\bX$ is a path of bounded variation in $\mathbb{R}^d$ then for a tree $\tau = [h_1, \cdots, h_n]_a$ in $\mathcal{T}$ we can write  $\innerprod{\bsig^N(\bX)_{st}, \tau}$ as
    \begin{equation*}
        \innerprod{\bsig^N(\bX)_{st}, \tau} = \int_s^t \innerprod{\bsig^N(\bX)_{sr}, h_1, \cdots, h_n} d \bX_r^{\Forest{decorated[a]}}.
    \end{equation*}
\end{rem}
\begin{rem}[Analogue of Chen's identity for branched rough paths] For the branched rough paths, the corresponding Chen's identity is defined thorough the coproduct. For $h=\Forest{decorated[a]}$, we have \(\int_s^t d \bX_r^{\Forest{decorated[a]}} = \int_s^u d \bX_r^{\Forest{decorated[a]}} + \int_u^t d \bX_r^{\Forest{decorated[a]}}\),  for any \(0 \le s \le u \le t \le T\). For a general tree $\tau = [h_1, \cdots, h_n]_a$, according to \eqref{coproduct h} and induction, we have
\begin{align*}
        \innerprod{\bsig^N(\bX)_{st}, \tau} & = \innerprod{\bsig^N(\bX)_{su}, \tau} + \int_u^t \innerprod{\bsig^N(\bX)_{sr}, h_1, \cdots, h_n} d\bX_r^{\Forest{decorated[a]}}\\
        & = \innerprod{\bsig^N(\bX)_{su}, \tau} + \int_u^t \innerprod{\bsig^N(\bX)_{su} \hat{\otimes} \bsig^N(\bX)_{ur}, \Delta ( h_1, \cdots, h_n)} d\bX_r^{\Forest{decorated[a]}}\\
        & = \innerprod{\bsig^N(\bX)_{su}, \tau} + \innerprod{\bsig^N(\bX)_{su} \hat{\otimes} \bsig^N(\bX)_{ut}, (\Id \hat{\otimes} [\; \cdot \; ]_a) \Delta ( h_1, \cdots, h_n)}\\
        & = \innerprod{\bsig^N(\bX)_{su} \hat{\otimes} \bsig^N(\bX)_{ut}, \tau \; \hat{\otimes} \; 1 +  (\Id \hat{\otimes} [\; \cdot \; ]_a) \Delta ( h_1, \cdots, h_n)}\\
        & = \innerprod{\bsig^N(\bX)_{su} \hat{\otimes} \bsig^N(\bX)_{ut}, \Delta(\tau)}.
\end{align*}
\end{rem}

\subsection{Geometric Realization of Branched Rough Paths}
While branched rough paths possess a more intricate algebraic structure than geometric rough paths, a bridge between them can be established by employing the extension map of \cite{hairer2015geometric} to construct an extended geometric rough path lying above a given branched rough path. In particular, for every branched rough path $\bsig^N(\bX)$ above a path $\bX$, there exists a geometric rough path $\sig^N(\Bar{\bX})$ above an extended path $\Bar{\bX}$ such that $\Bar{\bX}$ is an extension of $\bX$ and $\sig^N(\Bar{\bX})$ contains the information of $\bsig^N(\bX)$.

\begin{center}
\begin{tikzcd}[column sep=4cm, row sep=2cm]
\bX \arrow[r, "\text{Branched Rough path lift}"] \arrow[d, swap, "\text{extended to}"] & \bsig^N(\bX) \arrow[d, "\innerprod{\bsig^N(\bX),  \tau \color{black}} = \innerprod{\sig^N(\Bar{\bX}),  \Psi(\tau) \color{black}}  "] \\
\Bar{\bX} \arrow[r, "\text{Geometric Rough path lift}"] & \sig^N(\Bar{\bX})
\end{tikzcd}
\end{center}

In what follows, we provide a self-contained introduction to the map $\Psi$ as shown in the above diagram. We begin by defining several key notions. Let \(\mathcal{T}\) denote the set of rooted trees, and let \(\mathcal{T}_n\) be the set of rooted trees with at most \(n\) vertices. We define \(\mathcal{V}\) as the real vector space spanned by \(\mathcal{T}\), and \(\mathcal{V}_n\) as the real vector space spanned by \(\mathcal{T}_n\). The tensor algebra generated by \(\mathcal{V}\) is denoted by \(T(\mathcal{V}) := \bigoplus_{i=0}^{\infty} \mathcal{V}^{\otimes i}\), while the tensor algebra generated by \(\mathcal{V}_n\) is denoted by \(T(\mathcal{V}_n) := \bigoplus_{i=0}^{\infty} \mathcal{V}_n^{\otimes i}\). Similarly, the truncated tensor algebra of order \(N\) generated by \(\mathcal{V}\) is written as \(T^{(N)}(\mathcal{V}) := \bigoplus_{i=0}^{N} \mathcal{V}^{\otimes i}\), and the corresponding truncated tensor algebra generated by \(\mathcal{V}_n\) is \(T^{(N)}(\mathcal{V}_n) := \bigoplus_{i=0}^{N} \mathcal{V}_n^{\otimes i}\).

Clearly, $\bX$ lives in the space $\mathcal{V}_1$, where \(\mathcal{V}_1 :=\text{span}  \{\Forest{decorated[a]}: a=1,\cdots,d\} \cong \mathbb{R}^d\). While $\Bar{\bX}$ lives in $\mathcal{V}_N$ such that \(\text{Proj}_{\mathcal{V}_1} \Bar{\bX} = \bX\). Moreover,  $\sig^N(\Bar{\bX})$ lives in the truncated tensor product space $T^{(N)}(\mathcal{V}_N)$ also defined as
\begin{equation*}
    T^{(N)}(\mathcal{V}_N) = \text{span}\{h_1\otimes \cdots  \otimes h_n:h_i \in \mathcal{T}_N, 1 \le n \le N \},
\end{equation*}
such that $\innerprod{\sig^N(\Bar{\bX})_{st}, h} = \Bar{\bX}^h_t - \Bar{\bX}^h_s$ and the tensor components are understood as iterated integrals
\begin{equation*}
    \innerprod{\sig^N(\Bar{\bX})_{st}, h_1\otimes \cdots  \otimes h_n} = \int_{s}^{t} \cdots \int_{s}^{r_2}d\Bar{\bX}_{r_1}^{h_1} \cdots d\Bar{\bX}_{r_n}^{h_n}.
\end{equation*}
The connection between geometric and branched rough paths is built by using a morphism $\Psi: (\hopf, \cdot, \Delta) \to (T(\mathcal{V}), \shuffle, \Bar{\Delta})$. The \emph{existence} of such a morphism guarantees the existence of the extended path and geometric rough path over it. It is defined as
\begin{equation*}
    \Psi(h) = h + \Psi_{n-1}(h),  \text{ and } \Psi(h_1 h_2) = \Psi(h_1) \shuffle \Psi(h_2) \quad \forall \; h, h_1, h_2 \in \mathcal{F}_n, 
\end{equation*}
where $(\hopf, \cdot, \Delta)$ is the Hopf algebra of rooted trees, $(T(\mathcal{V}), \shuffle, \Bar{\Delta})$ is the shuffle Hopf algebra on $T(\mathcal{V})$. Also, $\Psi_{n-1}(h)$ is the projection of $\Psi$ onto $T(\mathcal{V}_{n-1})$ and is all the ways to cut apart $h$. The map $\Psi$ is equivalently defined as \(\Psi(h) = \left(\Psi \hat{\otimes}(\Id - \mathbf{1}^*)\right)\Delta(h)\) in \cite{bruned2022renormalisation}, where $\mathbf{1}^*$ is the co-unit. Also, an equivalent definition of the map is given in \cite{tapia2020geometry} i.e., \(\Psi(h) = h + \left(\Psi \hat{\otimes} \Id\right)\Delta'h\), where $\Delta'h$ is the reduced coproduct.

\begin{example}
One of the above constructions can be used to determine \(\Psi(h)\) for any rooted tree \(h\). For \(h = \Forest{decorated[a]}\), \(\Psi (\Forest{decorated[a]}) = \Forest{decorated[a]}\), for \(h = \Forest{decorated[b[a]]}\), \(\Psi (\Forest{decorated[b[a]]}) = \Forest{decorated[b[a]]} + \Forest{decorated[a]} \otimes \Forest{decorated[b]}\), and for \(h = \Forest{decorated[c[b][a]]}\), \( \Psi (\Forest{decorated[c[b][a]]}) = \Forest{decorated[c[b][a]]} + \Forest{decorated[a]} \otimes \Forest{decorated[c[b]]} + \Forest{decorated[b]} \otimes \Forest{decorated[c[a]]} + \Forest{decorated[b]} \otimes \Forest{decorated[a]} \otimes \Forest{decorated[c]} + \Forest{decorated[a]} \otimes \Forest{decorated[b]} \otimes \Forest{decorated[c]}\).
\end{example}

We now state and prove the following lemma related to the existence of this morphism. 
\begin{lem}
    \cite{hairer2015geometric} There exists a graded morphism of Hopf algebras $\Psi: (\hopf, \cdot, \Delta) \to (T(\mathcal{V}), \shuffle, \Bar{\Delta})$ defined as
\begin{equation*}
    \Psi(h) = h + \Psi_{n-1}(h), \quad \forall \; h \in \mathcal{T}_n,
\end{equation*}
such that $\left(\Psi \hat{\otimes}\Psi\right)\Delta h = \Bar{\Delta}\Psi(h)$.
\end{lem}
\begin{proof}
    For $n=1$ we have $\Psi(\Forest{decorated[a]}) = \Forest{decorated[a]}$ which trivially satisfies $\left(\Psi \hat{\otimes}\Psi\right)\Delta h = \Bar{\Delta}\Psi(h)$.
    Assume the morphism is true for $h$ with $|h|=n-1$. Now, let us prove the claim for $h$ with $|h|=n$.
    We have
    \begin{align*}
        \Bar{\Delta} \Psi(h) & = \Bar{\Delta}\left( \Psi(h_1) \otimes h_2 + h\right)\\
        & = \Bar{\Delta}\left( \Psi(h_1) \otimes h_2\right) + h \hat{\otimes} 1 + 1 \hat{\otimes} h\\
        & = \left( \Psi(h_1) \otimes h_2\right) \Hat{\otimes} 1 + (\Bar{\Delta}\Psi(h_1))\otimes (1 \hat{\otimes} h_2) + h \hat{\otimes} 1 + 1 \hat{\otimes} h\\
        & = \left( \Psi(h_1) \otimes h_2 + h\right) \Hat{\otimes} 1 + (\Psi \hat{\otimes}\Psi)(\Delta h_1)\otimes (1 \hat{\otimes} h_2) +1 \hat{\otimes} h\\
        & = \Psi(h) \Hat{\otimes} 1 + \Psi(h_1)\hat{\otimes}h_2 + 1 \hat{\otimes} (\Psi(h_1) \otimes h_2) + \Psi(h_{11}) \hat{\otimes} (\Psi(h_{12}) \otimes h_2)+1 \hat{\otimes} h\\
        & = \Psi(h) \Hat{\otimes} 1 + 1 \Hat{\otimes} \Psi(h) + \Psi(h_1)\hat{\otimes}h_2 + \Psi(h_{11}) \hat{\otimes} (\Psi(h_{12}) \otimes h_2)\\
        & = \Psi(h) \Hat{\otimes} 1 + 1 \Hat{\otimes} \Psi(h) + \Psi(h_1)\hat{\otimes}h_2 + \Psi(h_{1}) \hat{\otimes} (\Psi(h_{21}) \otimes h_{22})\\
        & = \Psi(h) \Hat{\otimes} 1 + 1 \Hat{\otimes} \Psi(h) + \Psi(h_{1}) \hat{\otimes} (h_2 + \Psi(h_{21}) \otimes h_{22})\\
        & = \Psi(h) \Hat{\otimes} 1 + 1 \Hat{\otimes} \Psi(h) + \Psi(h_{1}) \hat{\otimes} \Psi(h_{2})\\
        & = \left(\Psi \hat{\otimes} \Psi \right) \left( h \hat{\otimes} 1 +  1 \hat{\otimes} h + h_1 \hat{\otimes} h_2\right)\\
        & = \left(\Psi \hat{\otimes} \Psi \right) \Delta h.
    \end{align*}
Here transition from sixth to seventh line is by using coassociativity of the reduced coproduct $(\Delta' \hat{\otimes} \Id) \Delta' h = (\Id \hat{\otimes} \Delta') \Delta' h$.
\end{proof}

\begin{cor}
    For any $h \in \mathcal{T}_n$ we have
    \begin{equation*}
        \left(\Psi \hat{\otimes}\Psi\right)\Delta' h = \Bar{\Delta}'\Psi(h).
    \end{equation*}
\end{cor}
\begin{proof}
    The corollary can be trivially proved by using the previous lemma.
\end{proof}
The following theorem is the main result of this section and will serve as a foundation for several results in the subsequent sections. The result was originally established in \cite{hairer2015geometric}, and a more accessible proof was later provided in \cite{tapia2020geometry}. We present the proof here, following \cite{tapia2020geometry}, while filling in intermediate steps that were previously omitted for clarity and completeness.

\begin{thm} \label{geo}
    Let $\bX$ be a path in $\mathbb{R}^d$ and $\bsig^N(\bX)$ be the $\alpha$-H\"older continuous branched rough path. Then there exists
    \begin{enumerate}
        \item a path $\Bar{\bX}$ which takes values in the space $\mathcal{V}_N$ such that $\text{Proj}_{\mathcal{V}_1}(\Bar{\bX}) = \bX$,
        \item an $\alpha$-H\"older geometric rough path $\sig^N(\Bar{\bX})$ taking values in $T^{(N)}(\mathcal{V}_N)$ such that $\innerprod{\sig^N(\Bar{\bX})_{st}, h} = \Bar{\bX}_t^h - \Bar{\bX}_s^h$ for each $h \in \mathcal{T}_N$,
    \end{enumerate}
    such that
    \begin{equation*}
        \innerprod{\bsig^N(\bX)_{st}, h} = \innerprod{\sig^N(\Bar{\bX})_{st}, \Psi(h)}.
    \end{equation*}
\end{thm}

\begin{proof}
    Let us construct $\sig^N(\Bar{\bX})$ iteratively. Assume $\sig^N(\Bar{\bX})^{(1)}$ be the GRP over $\bX^i_{st} := \innerprod{\bsig^N(\bX)_{st}, \Forest{decorated[i]}}, i = 1,\cdots,d$.  Suppose we have constructed $\sig^N(\Bar{\bX})^{(k)}$ over the path $\Bar{\bX}^h$ such that $\Bar{\bX}_t^h - \Bar{
    \bX}^h_s = \innerprod{\bsig^N(\bX)_{st},h} - \innerprod{\sig^N(\Bar{\bX})_{st}^{(k-1)}, \Psi_{k-1}(h)}$ for $k=1,\cdots,n$. This is clearly true for $n=1$.     If we define $P^h_{st} = \innerprod{\bsig^N(\bX)_{st},h}$, $Q^h_{st} = \innerprod{\sig^N(\Bar{\bX})_{st}^{(n)}, \Psi_{n}(h)}$, and $\delta P^h_{st} = P^h_{st} - P^h_{su} - P^h_{ut}$. Then by Chen's relation for $h \in \mathcal{T}_{n+1}$, we have
    \begin{equation*}
        \delta P^h_{st} = \innerprod{\bsig^N(\bX)_{su} \hat{\otimes} \bsig^N(\bX)_{ut}, \Delta'h} = \innerprod{\sig^N(\Bar{\bX})^{(n)}_{su} \circ \Psi \otimes \sig^N(\Bar{\bX})^{(n)}_{ut} \circ \Psi, \Delta'h}.
    \end{equation*}
    Using the coalgebra morphism property of $\Psi$ we have
    \begin{equation*}
        \delta P^h_{st} = \innerprod{\sig^N(\Bar{\bX})^{(n)}_{su} \otimes \sig^N(\Bar{\bX})^{(n)}_{ut}, \Bar{\Delta}'\Psi(h)}.
    \end{equation*}
    As $h$ is a primitive element in tensor algebra so we have
    \begin{equation*}
        \delta P^h_{st} = \innerprod{\sig^N(\Bar{\bX})^{(n)}_{su} \otimes \sig^N(\Bar{\bX})^{(n)}_{ut}, \Bar{\Delta}'\Psi_n (h)} = \delta Q^h_{st}.
    \end{equation*}
    If we set $M =  P^h_{st} - Q^h_{st}$, then by using above equation we have $\delta M=0$, where $M: [0,1]^2 \to \mathbb{R}$. Then by using formula (5) from \cite{gubinelli2010ramification}, there exists a function $\Bar{\bX}: [0,T] \to \mathbb{R}^d$ such that $\Bar{\bX}^h = M = P^h_{st} - Q^h_{st}$ and
    \begin{equation*}
        |\Bar{\bX}^h_t - \Bar{\bX}^h_s| \le |\innerprod{\bsig^N(\bX)_{st},h}| + |\innerprod{\sig^N(\Bar{\bX})_{st}^{(n)}, \Psi_n(h)}| \lesssim |t-s|^{\alpha |h|}. 
    \end{equation*}
    With this, we have a geometric rough path $\sig^N(\Bar{\bX})^{(n+1)}$ over the path $\Bar{\bX}^h: h \in \mathcal{T}_{n+1}$ whose restriction to $T^{(N)}(\mathcal{V}_N)$ is same as $\sig^N(\Bar{\bX})^{(n)}$.\\
    Hence, for $h \in \mathcal{T}_{n+1}$
    \begin{align*}
        \innerprod{\sig^N(\Bar{\bX})_{st}^{(n+1)}, \Psi_n(h)} & = \innerprod{\sig^N(\Bar{\bX})_{st}^{(|h|)}, h} + \innerprod{\sig^N(\Bar{\bX})_{st}^{(|h|)}, \Psi_{|h|-1}(h)}\\
        & = \Bar{\bX}_t^h - \Bar{\bX}^h_s + \innerprod{\bsig^N(\bX)_{st},h} - \left( \Bar{\bX}_t^h - \Bar{\bX}^h_s\right)\\
        & = \innerprod{\bsig^N(\bX)_{st},h}.
    \end{align*}
    Finally, the geometric rough path we look for is $\sig^N(\Bar{\bX}) = \sig^N(\Bar{\bX})^{(N)}$.
\end{proof}

\section{Branched Signature Model} \label{section universal}
In this section, we will establish the universal approximation theorem for the branched signature. Before proving the main results, we first introduce branched signature and branched signature model.
\subsection{Universal approximation theorem for branched signature}
\begin{defn}[Branched Signature]
Let $\cS=\{1,\dots,d\}$ be an alphabet of decorations for a given \(d\)-dimensional path \(\bX: [0,T] \to \RR^d\). Denote by $\trees$ the set of rooted trees with vertices decorated in $\cS$. Let $\hopf$ be the (decorated) Connes–Kreimer Hopf algebra generated by $\trees$, with product given by disjoint union of forests and unit $\mathbf 1$. We define \emph{branched signature} of \(\bX\) as a functional on \(\hopf\) given by
\begin{equation}
\bsig(\bX)_{st} =  \sum\limits_{\tau \; \in \; \mathcal T, \; |\tau| \; \le \; N} \innerprod{\bsig(\bX)_{st},\tau} \se_{\tau},
\end{equation}
where for each \(\tau \in \hopf\) the component \(\innerprod{\bsig(\bX)_{st},\tau}\) of the branched signature is recursively defined as
\[\innerprod{\bsig(\bX)_{st},\mathbf{1}} = 1, \quad \text{ and } \quad \innerprod{\bsig(\bX)_{st},\tau} = \int_s^t \innerprod{\bsig(\bX)_{su},\tau'} d\bx^{\mathbf{r}}_u,\]
where \(\mathbf{r}\) is the root of \(\tau\) and  \(\tau'\) is the tree we get after removing root \(\mathbf{r}\) from \(\tau\).
\end{defn}
\begin{rem}
The branched signature of a path \(\bX\) is the unique multiplicative extension of a branched rough path, as shown in \cite{gubinelli2010ramification}; this mirrors the classical unique extension of a geometric rough path to its signature \cite{Friz_Victoir_2010}.
\end{rem}
In the classical/geometric setting, the universal approximation theorem for rough paths states that any continuous function of the path can be well-approximated by a linear combination of the iterated integrals i.e., the components of the classical signature. Formally,
\[f (\bX) \approx \sum_{ \bw \; \in \; \bW, \; |\bw| \; \ge \; 0} a_{\bw} \innerprod{\sig(\bX), \bw},\]
where \(\bw \) is a word made from the alphabet set \(\mathcal{S} = \{1,2,\dots,d\}\) and \(a_{\bw} \in \RR\). For the non-trivial empty word \(\bw = \emptyset\), \(|\bw|= 0 \) and \(\innerprod{\sig(\bX), \emptyset} = 1\). A similar result for branched signature would ensure that any continuous function of the path can  be well-approximated using the components of the branched signature i.e.,
\[f (\bX) \approx \ell_{\phi} + \sum_{1 \le |\tau| \le N} \ell_{\tau} \innerprod{\bsig(X)_{st}, \tau},\]
for all \(\tau \in \hopf\). With this, we formally define the \emph{branched signature model} as follows.
\begin{defn}[Branched signature model]\label{def:branched-model}
Let $\hopf$ be the (decorated) Connes–Kreimer Hopf algebra generated by rooted trees whose vertices are decorated by $\cS = \{1,\dots,d\}$ and $\mathcal{F}$ for the set of rooted \emph{forests}, with product given by disjoint union and unit
$\mathbf{1}$ (the empty forest). Denote the subspace
$\hopf_{\le N}:=\mathrm{span}\{\tau\in\mathcal F:\,|\tau|\le N\}$. Then for coefficients $\ell=\ell_{\mathbf 1}\mathbf 1+\sum_{1\le|\tau|\le N}\ell_\tau\,\tau\in\hopf_{\le N}$,
the \emph{branched signature model} (truncated to level $N$) for a \(d\)-dimensional path \(\bX: [0,T] \to \RR^d\) is the linear functional of the branched signature i.e.,
\begin{equation}
\mathbf{M}^N_\ell(\mathbf X)_{st}
\;:=\;\big\langle \bsig(\mathbf X)_{st},\,\ell\big\rangle
\;=\;\ell_{\mathbf 1}+\sum_{1\le|\tau|\le N}\ell_\tau\,\big\langle \bsig(\mathbf X)_{st},\,\tau\big\rangle.
\end{equation}
\end{defn}
An example of branched signature model for \(N=2\) is given as follows.
\begin{example}
For \(N=2\) the branched signature model is given as 
\[\mathbf{M}^2_\ell(\mathbf X)_{st} = \ell_{\mathbf{1}} + \sum_{i  \in  \cS}\ell_{\Forest{decorated[i]}} \innerprod{\bsig(\bX)_{st}, \Forest{decorated[i]}} + \sum_{j,k  \in  \cS} \left( \ell_{\Forest{decorated[j]}\Forest{decorated[k]}} \innerprod{\bsig(\bX)_{st}, \Forest{decorated[j]}\Forest{decorated[k]}} + \ell_{\Forest{decorated[k[j]]}} \innerprod{\bsig(\bX)_{st}, \Forest{decorated[k[j]]}} \right).\]
Equivalently,
\[\mathbf{M}^2_\ell(\mathbf X)_{st} = \ell_{\mathbf{1}} + \sum_{i  \in  \cS} \ell_{\Forest{decorated[i]}} \int_s^t d\bX_{r_1}^{\Forest{decorated[i]}} + \sum_{j,k  \in  \cS} \left( \ell_{\Forest{decorated[j]}\Forest{decorated[k]}} \int_s^t d\bX_{r_1}^{\Forest{decorated[j]}} \int_s^t d\bX_{r_1}^{\Forest{decorated[k]}} + \ell_{\Forest{decorated[k[j]]}} \int_s^t \int_s^{r_2} d\bX_{r_1}^{\Forest{decorated[j]}} d\bX_{r_2}^{\Forest{decorated[k]}} \right).\]
\end{example}
We begin by stating and proving a uniqueness principle for branched signatures, needed for our universal approximation theorem. In general, equality of branched signatures implies equality of paths modulo tree-like equivalence. For time-extended paths, however, the monotone time component rules out nontrivial tree-like loops, so equality of branched signatures actually forces equality of the paths themselves. Therefore, let us work with time extended paths from now on. The iterated integrals with a component of time can be defined in the sense of Young's.
\begin{lem}[Uniqueness of the branched signature] \label{bsig_uiqueness}
Let $\bX,\bY:[0,T]\to\mathbb R^{d+1}$ be two time extended paths with the first component to be the time component and $\bX_0=\bY_0=0$. The corresponding alphabet set is \(\cS = \{0,1,\dots,d\}\), Let $\bX^-,\bY^-:[0,T]\to\mathbb R^{d}$ be the paths without time component and are continuous $\alpha$-H\"older paths for some $\alpha >  \frac{1}{4}$. Assume their branched terminal signatures coincide at all levels i.e.,
\[
\bsig(\bX)_{0,T}=\bsig(\bY)_{0,T}.
\]
Then $\bX_t=\bY_t$ for all $t\in[0,T]$.
\end{lem}

\begin{proof}
Fix a spatial index $i\in\cS \backslash \{0\}$ and set $\bZ:=\bX^i-\bY^i$, a continuous $\alpha$-H\"older path with $\bZ_0=0$.
We use the family of signature coordinates of the path that contain exactly one spatial letter. Consider for every $k,m\in\mathbb N\cup\{0\}$, \(\tau = [[[\tilde{\tau}\underbrace{]_0] \dots]_0}_{m \text{-times}}\) with 
\(\tilde{\tau} = [[[\mathbf{i}\underbrace{]_0] \dots]_0}_{k \text{-times}}\),
\begin{equation} \label{bsig_one_spatial}
\big\langle \bsig(\bZ)_{0,T},\, \tau \big\rangle
=\frac{1}{k!\,m!}\int_0^T s^{\,k}(T-s)^{\,m}\,d\bZ_s ,
\end{equation}
where the integral is defined in the Young's sense. This is trivial when $\bZ$ is smooth; for general $\alpha$-H\"older $\bZ$, take smooth approximations $\bZ^n\to \bZ$ in $C^\alpha$, use the classical identity for $\bZ^n$, and pass to the limit: the map $\bZ\mapsto\int s^{k}(T-s)^m\,d\bZ$ is continuous in $C^\alpha$, and the one-spatial-letter signature coordinates are defined by the same limiting procedure.
By the hypothesis $\bsig(\bX)_{0,T}=\sig(\bY)_{0,T}$, identity \eqref{bsig_one_spatial} applied to $\bZ=\bX^i-\bY^i$ yields, for all $k,m \ge 0$,
\begin{equation}\label{bsigweighted-Young-zero}
\int_0^T s^{\,k}(T-s)^{\,m}\,d\bZ_s=0 .
\end{equation}
Beyond this point the proof is similar to the uniqueness of the classical signature.
\end{proof}
Let us state and prove universal approximation theorem for branched signature model now.

\begin{thm}[UAT for branched signatures of time-extended $\alpha$-H\"older paths]\label{thm:UAT-branched}
Let $\alpha>\tfrac14$ and set $p=\lfloor 1/\alpha\rfloor$. Let $\hopf_{\le p}$ be the Connes--Kreimer Hopf algebra of rooted (time/space–decorated) trees truncated at degree $p$, and let $\BB^{(p)}(\RR^{1+d})$ denote its character group (the step-$p$ Butcher group) over the alphabet $\cS=\{0,1,\dots,d\}$, where $0$ is the time letter. Write $\langle\cdot,\cdot\rangle$ for the canonical pairing between $\BB^{(p)}(\RR^{1+d})$ and $\hopf_{\le p}$. For a path \(\bX: [0,T] \to \RR^{d+1}\) such that \(d\)-dimensional path without time component \(\bX^-\) is \(\alpha\)-H\"older i.e., $\bX^- \in C^\alpha([0,T];\RR^d)$, define
\[
\mathcal S^{(p)} \;:=\; \Big\{\, \big(\bsig^{p}(\bX)_{t}\big)_{t\in[0,T]}\;:\; \bX^- \in C^\alpha([0,T];\RR^d)\,\Big\}
\;\subset\; C \ \big([0,T],\,\BB^{(p)}(\RR^{1+d})\big).
\]
Let $\mathcal{H}\subset \mathcal S^{(p)}$ be compact and $f:\mathcal{H}\to\RR$ continuous. Then for every $\varepsilon>0$ there exists $h \in \hopf$ such that
\[
\sup_{\,(\bsig^{p}(\bX)_{t})_{t\in[0,T]}\in \mathcal{H}}
\Big|\, f\big((\bsig^{p}(\bX)_{t})_{t\in[0,T]}\big)
\;-\; \langle \bsig(\bX)_{T},\, h \rangle \,\Big| \;<\; \varepsilon.
\]
\end{thm}

\begin{proof}
Consider the set
\[
\mathcal G\;:=\;\mathrm{span}\Big\{\;(\bsig^{p}(\bX)_t)_{t\in[0,T]}\mapsto \langle \bsig(\bX)_T,\,h\rangle \;:\; h\in\hopf\Big\}\;\subset\; C(\mathcal{H}).
\]
Then $\mathcal G$ is a unital subalgebra; the constant $1$ corresponds to $h=\mathbf{1}$ (empty forest) i.e., \(\langle \bsig(\bX)_T,\,\mathbf{1}\rangle = 1\) , and for $h_1,h_2\in\hopf$, \(\langle \bsig(\bX)_T,h_1\rangle \langle \bsig(\bX)_T,h_2\rangle = \langle \bsig(\bX)_T,h_1 h_2\rangle \in \mathcal{G}\), since $\bsig^{p}(\bX)_T\in\BB^{(p)}(\RR^{1+d})$ is a character on $\hopf$. Finally, \(\mathcal G\) separates points i.e., for any two paths \(\bX\) and \(\bY\) with \(\bX \ne \bY\) implies \(\innerprod{\bsig(\bX)_T, h} \ne \innerprod{\bsig(\bY)_T, h}\) for any word \(h \in \hopf\). On contrary, suppose \(\innerprod{\bsig(\bX)_T, h} = \innerprod{\bsig(\bY)_T, h}\), then by uniqueness of branched rough path lift  \cite{gubinelli2010ramification}, \(\innerprod{\bsig^p(\bX)_t, h} = \innerprod{\bsig^p(\bY)_t, h}\) for any \(t \in [0,T]\). Furthermore, if \(\innerprod{\bsig(\bX)_T, h} = \innerprod{\bsig(\bY)_T, h}\) then \(\bX_t = \bY_t\) for any \(t \in [0,T]\) by the uniqueness of the signature by Theorem \ref{bsig_uiqueness}, which is a contradiction to original claim. Therefore \(\mathcal G\) separates points. Hence, the claim follows by Stone-Weierstrass theorem.
\end{proof}
Every component of a branched rough path \(\bsig(\bX)\) is identified with the corresponding component of geometric rough path i.e., \(\sig(\Psi(\bX))\), where \(\Psi\) is Hairer-Kelly morphism. Therefore, we have the following version of the universal approximation theorem for branched signatures.
\begin{cor} \label{cor for UAT}
Let $\alpha>\tfrac14$ and set $p=\lfloor 1/\alpha\rfloor$. For a path \(\bX: [0,T] \to \RR^{d+1}\) such that its \(d\)-dimensional components without time component \(\bX^-\) is \(\alpha\)-H\"older, i.e. $\bX^- \in C^\alpha([0,T];\RR^d)$, and we define \(\mathcal S^{(p)}\) as before. Let $\mathcal{H}\subset \mathcal S^{(p)}$ be compact and $f:\mathcal{H}\to\RR$ continuous. Then for every $\varepsilon>0$, there exists $h \in \hopf$ such that
\[ \sup_{\,(\bsig^{p}(\bX)_{t})_{t\in[0,T]}\in \mathcal{H}}
\Big|\, f\big((\bsig^{p}(\bX)_{t})_{t\in[0,T]}\big)
\;-\; \langle \sig(\bX)_{T},\, \Psi(h) \rangle \,\Big| \;<\; \varepsilon. \]
\end{cor}
\subsection{Iterative application of signature model} \label{layerwise section}
Evaluating the classical signature of a path up to a fixed level \(N\) is computationally expensive, especially for large \(N\), even when using optimized software packages such as \textit{iisignature} 
 \cite{reizenstein2018iisignature} or \textit{signatory} \cite{kidger2020signatory}. To address this computational challenge, we adopt an approach based on the iterative application of signature models of lower degree. This iterative procedure allows us to approximate the signature model of a higher degree \(N\) by composing models of smaller depth \(k\), and can naturally be interpreted as stacking layers in a neural network. In this interpretation, the coefficients of the signature models serve as the parameters that can be learned. To begin with, we will show that every component of the higher order signature model (say \(N\)) can be expressed as a lower level signature (say \(m < N\)) applied to some lower order signature model (say \(k<N\)). The following result formalizes the validity of this approach.

\begin{lem} \label{classical_iterated}
Let \(\bX: [0, T] \to \RR^d\) be a path for which some which a geometric rough path lift is well defined. For $0<s<t<T$, define its \emph{classical signature model} (truncated to level $N$) as
\begin{equation}
\mathbf{S}^N_\ell(\mathbf X)_{st}
\;:=\;\big\langle \sig(\mathbf X)_{st},\,\ell\big\rangle
\;=\;\ell_{\emptyset}+\sum_{1\le|\bw|\le N}\ell_\bw\,\big\langle \sig(\mathbf X)_{st},\,\bw\big\rangle,
\end{equation}
where \(\bw\) is a word from the alphabet set \(\cS = \{1,2,\dots,d\}\). For \(k < N\) , denote \(m:= \lceil N/k \rceil\) be the integer part. Define the path \(\Phi_t\) by \(\Phi_t := \mathbf{S}^k_\ell(\mathbf X)_{st}\) i.e., the classical signature model truncated to level \(k\). Then, every component of the level-\(N\) signature model \(\mathbf{S}^N_\ell(\mathbf X)_{st}\) can be recovered by the level-\(m\) signature applied to \(\Phi_t\).

Consequently, any level-\(N\) signature model of the path \(\bX\) can be exactly replicated by a level-\(m\) signature model of the path \(\Phi_t\).
\end{lem}
\begin{proof}
For all \(t \in [0,T]\) and \(s \le t\), \(\Phi_t\) is a one dimensional path i.e., \(\Phi_t:[0,T] \to \RR\) and is defined as
\[\Phi_t := \mathbf{S}^k_\ell(\mathbf X)_{st}= \ell_{\emptyset}+\sum_{1\le|\bw|\le k}\ell_\bw\,\big\langle \sig(\mathbf X)_{st},\,\bw\big\rangle.\]
For any word \(\bv = v_1 v_2 \dots v_m\) with \(v_i = 1, \; i = 1,2,\dots,m\), the component of the signature of \(\Phi_t\) corresponding to \(\bv\) i.e., \(\innerprod{\sig(\Phi_t)_{st}, \bv}\) is given as
\[\innerprod{\sig(\Phi_t)_{st}, \bv} = \int_s^t\int_s^{r_m} \dots \int_s^{r_2} d\Phi_{r_1}^{v_1} \dots d\Phi_{r_m}^{v_m}.\]
Since each \(v_i = 1\), therefore using the identity \(\underbrace{1 \shuffle \dots \shuffle 1}_{m\text{-times}} = m! \, 1  \dots  1\), the right hand side becomes \(\innerprod{\sig(\Phi_t)_{st}, \bv} = \frac{(\Phi_t - \Phi_s)^m}{m!}.\) Substituting the expression for \(\Phi_t\) gives
\[\innerprod{\sig(\Phi_t)_{st}, \bv} = \frac{\left(\ell_{\emptyset}+\sum_{1\le|\bw|\le k}\ell_\bw\,\big\langle \sig(\mathbf X)_{st},\,\bw\big\rangle\right)^m}{m!}.\]
The right hand side again can be expanded using the binomial theorem to get
\[
\begin{aligned}
\innerprod{\sig(\Phi_t)_{st},\bv}
&= \ell_{\emptyset}^m
 + m\ell_{\emptyset}^{m-1}\!\sum_{1\le|\bw|\le k}\!\ell_\bw\,\big\langle \sig(\mathbf X)_{st},\bw\big\rangle \\
&\quad + \frac{m(m-1)}{2!}\ell_{\emptyset}^{m-2}\!\left(\sum_{1\le|\bw|\le k}\ell_\bw\,\big\langle \sig(\mathbf X)_{st},\bw\big\rangle\right)^2
 + \dots +  \left(\sum_{1\le|\bw|\le k}\ell_\bw\,\big\langle \sig(\mathbf X)_{st},\bw\big\rangle\right)^m .
\end{aligned}
\]
The very last term of the expansion i.e., \(\left(\sum_{1\le|\bw|\le k}\ell_\bw\,\big\langle \sig(\mathbf X)_{st},\bw\big\rangle\right)^m\) can be expanded by using a multinomial expansion formula. The highest degree term observed is \(\langle \sig(\mathbf X)_{st},\bw\big\rangle^m\). Using geometric property of the signature we get \(\langle \sig(\mathbf X)_{st},\bw\big\rangle^m = \langle \sig(\mathbf X)_{st},\underbrace{\bw\shuffle \dots \shuffle \bw}_{m\text{-times}}\big\rangle\). Since \(\bw\) is a word of length at most \(k\), therefore \(\bw\shuffle \dots \shuffle \bw\) is a word with length at most \(km\) that is \(|\bw\shuffle \dots \shuffle \bw| \le km  = k \lceil N/k \rceil\), and hence \(|\bw\shuffle \dots \shuffle \bw| \le N\). Also, since the term \(\left(\sum_{1\le|\bw|\le k}\ell_\bw\,\big\langle \sig(\mathbf X)_{st},\bw\big\rangle\right)^m\) covers all possible products with different words \(\bw^1,\dots, \bw^m\) of each length \(k\), therefore \(\bw^1 \shuffle \dots \shuffle \bw^m\) is a word of the all possible choices from the alphabet set \(\cS\) with length \(km\). This concludes that if \(|\bv| \le m\), then any component of the level-\(N\) signature model \(\mathbf{S}^N_\ell(\mathbf X)_{st}\) can be recovered by the level-\(m\) signature applied to level-\(k\) signature model \(\mathbf{S}^k_\ell(\mathbf X)_{st}\) for any \(k<N\) and \(m = \lceil N/k \rceil\).
\end{proof}

\begin{cor}
For a given a path \(\bX: [0, T] \to \RR^d\), its level-\(N\) signature model \(\mathbf{S}^N_\ell(\mathbf X)_{st}\) can be fully replicated by applying level-\(k\) signature model \(m\)-times i.e.,
\[
\mathbf{S}^N_\ell(\mathbf X)_{st}= \mathbf{S}^k_{\ell^{(m)}}\left(\mathbf{S}^k_{\ell^{(m-1)}}\left(\dots \mathbf{S}^k_{\ell^{(1)}}\left(\mathbf X\right)_{st}\right)_{st}\right)_{st},
\]
where \(m \ge  \lceil \frac{\ln N}{\ln k} \rceil\).
\end{cor}
\begin{proof}
Applying the level-\(k\) signature model  on the signature path recovered from a signature model of level-\(k\), we recover such a model of level-\(k^2\) by Lemma ~\eqref{classical_iterated}. Repeating this process \(m\) times gives signature model of level-\(k^m\). We want \(k^m \ge N\) which gives \(m \ge  \lceil \frac{\ln N}{\ln k} \rceil\) whenever \(k>1\).
\end{proof}
 A direct analogue of Lemma ~\eqref{classical_iterated} (which relies on the shuffle product algebra) does not hold for the branched signature model. The shuffle property, fundamental to the classical signature, breaks down for the non-geometric rough paths captured by the branched signature. Consequently, composing branched signature operators behaves differently. While applying a level-$k$ \emph{classical} signature operator twice effectively creates a level-$k^2$ feature set, applying a level-$k$ \emph{branched} signature operator to the path generated by another level-$k$ branched operator does not necessarily replicate all features up to level $k^2$. Instead, due to the non-geometric nature and the specific algebraic structure (related to the Connes-Kreimer Hopf algebra), such composition primarily extends the captured dependencies incrementally, one each time. 
 
 Reproducing the result analogous to Lemma ~\eqref{classical_iterated} is not possible because branched signatures do not satisfy shuffle property. For simplification, avoiding to be too complex,  we restrict level of the the branched signature to be \(2\). The reason is to make the computation and overall complexity to be as small as possible. The following result formalizes this approach and shows additive nature of composition in the branched setting, contrasting sharply with the multiplicative effect seen in the classical case.

\begin{lem} \label{branched_iterated}
Let \(\bX: [0, T] \to \RR^d\) be a path for which some notion of branched rough path exists. Let \(k = N-1\) be an integer and \(m = 2\). Define the path \(\Phi_t\) by \(\Phi_t := \mathbf{M}^k_\ell(\mathbf X)_{st}\) i.e., the branched signature model truncated to level \(k\). Then, every component of the level-\(N\) branched signature model \(\mathbf{M}^N_\ell(\mathbf X)_{st}\) can be recovered by the level-\(m\) branched signature applied to \(\Phi_t\).

Consequently, any level-\(N\) branched signature model of the path \(\bX\) can be exactly replicated by a level-\(m\) branched signature model of the path \(\Phi_t\).
\end{lem}

\begin{proof}
The path \(\Phi_t\) identified by the branched signature model truncated to level \(k\) is given as
\begin{equation*}
\Phi_t := \mathbf{M}^N_\ell(\mathbf X)_{st}
\;=\;\ell_{\mathbf 1}+\sum_{1\le|\tau_1|\le k}\ell_\tau\,\big\langle \bsig(\mathbf X)_{st},\,\tau_1\big\rangle,
\end{equation*}
where \(\tau_1\) is a rooted forest of degree at most \(k\). For now we will restrict ourselves to trees only not the forests as forest will involve the product terms. Consider a rooted tree \(\tau_2\) of degree at most \(m\). Since every tree is constructed recursively, so consider \(\tau_1 = [h_1 \cdots h_p]_{r_1}\) and \(\tau_2 = [l_1 \cdots l_q]_{r_2}\) with \(|h_1 \cdots h_n| = k-1\) and \(|l_1 \cdots l_n| = m-1\) and \(r_1\) and \(r_2\) are roots of the tree \(\tau_1\) and \(\tau_2\) respectively.

The component of the branched signature of level \(m\) of \(\Phi_t\) is denoted by \(\innerprod{\bsig(\Phi_t)_{st}, \tau_2}\) with \(|\tau_2| \le m\) and is given as
\begin{equation} \label{bsig level 2}
\innerprod{\bsig(\Phi_t)_{st}, \tau_2} = \int_s^t \innerprod{\bsig(\Phi_u)_{su}, l_1 \cdots l_q} d\Phi_{u}^{\Forest{decorate[r_2]}}.
\end{equation}
If we pick the highest degree component from \(\bsig(\Phi_t)_{st}\) say \(\big\langle\bsig(\mathbf X)_{st},\,\tau_1\big\rangle\) with \(|\tau_1| \le k\) and set  \(\Phi(t) = \innerprod{\bsig(\bX)_{st}, \tau_1}\), then \(d\Phi(t) = \innerprod{\bsig(\bX)_{st}, h_1 \cdots h_p} d\bX_t^{\Forest{decorate[r_1]}}\). Now, since \(m=2\), so let us take \(\tau_2 = \Forest{decorated[j[i]]}\) with both \(i,j=1\) as the path \(\Phi_t\) is only one dimensional. Using equation~\eqref{bsig level 2}, we get
\begin{align*}
\innerprod{\bsig(\Phi_t)_{st}, \Forest{decorated[j[i]]}} & = \int_s^t \innerprod{\bsig(\Phi_u)_{su}, \Forest{decorated[i]}} d\Phi_{u}^{\Forest{decorate[j]}}\\
& = \int_s^t \innerprod{\bsig(\bX)_{su}, \tau_1} \innerprod{\bsig(\bX)_{su}, h_1 \cdots h_p} d\bX_u^{\Forest{decorate[r_1]}}\\
& = \innerprod{\bsig(\bX)_{st}, [\tau_1 h_1 \cdots h_p]_{r_1}}.
\end{align*}
Which is just an extra root \(r_1\) introduced in the branched signature of the underlying path \(\bX\). Therefore, applying level-\(m\) branched signature on the branched signature path generated by level-\(k\) branched signature model only gives information up to level-\((k+1)\) which is \(N\) in this case. 
\end{proof}
\begin{cor}
For a given a path \(\bX: [0, T] \to \RR^d\), level-\(N\) branched signature model \(\mathbf{M}^N_\ell(\mathbf X)_{st}\) can be exactly replicated by applying level-\(2\) signature model \(N-1\)-times i.e.,
\[
\mathbf{M}^N_\ell(\mathbf X)_{st}= \mathbf{M}^2_{\ell^{(N-1)}}\left(\mathbf{M}^2_{\ell^{(N-2)}}\left(\dots \mathbf{M}^2_{\ell^{(1)}}\left(\mathbf X\right)_{st}\right)_{st}\right)_{st}.
\]
\end{cor}
\begin{proof}
Applying branched signature model of level-\(2\) on the path recovered from a signature model of level-\(2\), we recover such a model of level-\(2+1=3\) by Lemma.~\eqref{branched_iterated}. Repeating this process \(N-1\) times gives signature model of level-\(N-1+1=N\).
\end{proof}
Computing a level-\(N\) signature in \(d\) dimensions scales as \(O(d^{N})\), which is computationally intractable when \(d\) is large. By contrast, a level-\(2\) signature costs \(O(d^{2})\). The lemma shows we can recover the level-\(N\) model by iterating a level-\(2\) signature on the model path: after each step, the model is effectively one-dimensional (update cost \(O(d)\)). Repeating this \(N\!-\!1\) times keeps the dependence on \(d\) quadratic i.e., \(O(d^{2})\) rather than \(O(d^{N})\), making the approach especially effective for very high-dimensional data streams.

The following version of universal approximation theorem extends the universal approximation theorem for branched rough paths ~\eqref{thm:UAT-branched} to this layer-wise application of branched signature model with learnable parameters.

\begin{thm}
Let \(\bX: [0,T] \to \RR^{d+1}\) be the time-extended path such that \(d\)-dimensional path without time component \(\bX^-\) is \(\alpha\)-H\"older i.e., $\bX^- \in C^\alpha([0,T];\RR^d)$. Let  \(\mathbf{M}^2_\ell(\mathbf X)_{st}\) be the level-2 branched signature model. Set \(s=0\) and  Let 
\(\mathbf{M}^{\circ m}_\ell(\mathbf X)_{0t}= \mathbf{M}^2_{\ell^{(m)}}\left(\mathbf{M}^2_{\ell^{(m-1)}}\left(\dots \mathbf{M}^2_{\ell^{(1)}}\left(\mathbf X\right)_{0t}\right)_{0t}\right)_{0t}\) be the application of branched signature model \(m\) times. Define \(\mathcal S^{(p)}\) as before
\[
\mathcal S^{(p)} \;:=\; \Big\{\, \big(\bsig^{p}(\bX)_{t}\big)_{t\in[0,T]}\;:\; \bX^- \in C^\alpha([0,T];\RR^d)\,\Big\}
\;\subset\; C \ \big([0,T],\,\BB^{(p)}(\RR^{1+d})\big).
\]
Let $\mathcal{H}\subset \mathcal S^{(p)}$ be compact and $f:\mathcal{H}\to\RR$ continuous. Then for every $\varepsilon>0$ there exists \(\ell\) such that
\[
\sup_{\,(\bsig^{p}(\bX)_{t})_{t\in[0,T]}\in \mathcal{H}}
\Big|\, f\big((\bsig^{p}(\bX)_{t})_{t\in[0,T]}\big)
\;-\; \mathbf{M}^{\circ m}_\ell(\mathbf X)_{0T} \,\Big| \;<\; \varepsilon,
\]
where \(m\) is sufficiently large and \(\ell = \{\ell^{(1)}, \dots, \ell^{(m)}\}\) is the set of learnable parameters with \(\ell^{(1)}\) to be the parameters for the first model, \(\ell^{(2)}\) for the second and so on.
\end{thm}

\begin{proof}
Using Theorem ~\eqref{thm:UAT-branched} with the usual notation and let $\mathcal{H}\subset \mathcal S^{(p)}$ be compact and $f:\mathcal{H}\to\RR$ continuous. Then for every $\varepsilon>0$ there exists $h \in \hopf$ such that
\[
\sup_{\,(\bsig^{p}(\bX)_{t})_{t\in[0,T]}\in \mathcal{H}}
\Big|\, f\big((\bsig^{p}(\bX)_{t})_{t\in[0,T]}\big)
\;-\; \langle \bsig(\bX)_{T},\, h \rangle \,\Big| \;<\; \varepsilon.
\]
By Lemma ~\eqref{branched_iterated} every component of the branched signature of some higher level can be replicated by iteratively applying branched signature model of level-\(2\). Therefore,
\[\langle \bsig(\bX)_{T},\, h \rangle= \mathbf{M}^2_{\ell^{(m)}}\left(\mathbf{M}^2_{\ell^{(m-1)}}\left(\dots \mathbf{M}^2_{\ell^{(1)}}\left(\mathbf X\right)_{0T}\right)_{0T}\right)_{0T}.\]
When \(m\) is sufficiently large. Then for any $h \in \hopf$, \(\langle \bsig(\bX)_{T},\, h \rangle\) can be recovered by learning the parameters \(\ell^{(1)}, \dots, \ell^{(m)}\). Hence, the claim follows from this.
\end{proof}

From a practical standpoint, there is no existing library for computing the branched signature of an arbitrary path, because the nature of the driving signal (and hence the appropriate integration rule) is typically unknown. When the underlying path is Brownian motion, the It\^o iterated integrals do induce a branched signature; for a general path, this construction is not available. To address this, we follow the extension principle of Hairer–Kelly \cite{hairer2015geometric}, recalled in Theorem~\ref{geo}: extend the observed data \(\bX\) to a higher-dimensional process \(\bar{\bX}\), and then apply low-order classical signature models iteratively. The next result formalizes this idea.

\begin{thm}
Let \(\bX: [0,T] \to \RR^{d+1}\) be the time-extended path such that \(d\)-dimensional path without time component \(\bX^-\) is \(\alpha\)-H\"older i.e., $\bX^- \in C^\alpha([0,T];\RR^d)$. Let \(\bar{\bX}\) be the extended path such that   
\(\innerprod{\bsig(\bX), \tau} = \innerprod{\sig(\Bar{\bX}^),\Psi(\tau)}\) for all \(\tau \in \hopf\). Let  \(\mathbf{S}^k_\ell(\Bar{\bX})_{st}\) be the level-\(k\) classical signature model applied to the extended path \(\Bar{\bX}\). Set \(s=0\) and  Let 
\(\mathbf{S}^{\circ m}_\ell(\mathbf X)_{0t}= \mathbf{S}^k_{\ell^{(m)}}\left(\mathbf{S}^k_{\ell^{(m-1)}}\left(\dots \mathbf{S}^k_{\ell^{(1)}}\left(\mathbf X\right)_{0t}\right)_{0t}\right)_{0t}\) be the application of signature model \(m\) times. Let the dimension of \(\Bar{\bX}\) be \(\tilde{d}\) and \(\tilde{\cS}\) be the alphabet set over \(\Bar{\bX}\). Let \(\mathbb{G}^{p}(\mathbb R^{\tilde{d}})\) be the step-\(p\) nilpotent Lie group over the alphabet \(\tilde{\cS}\). Define \(\mathcal S^{(p)}\) as before
\[
\mathcal S^{(p)} \;:=\; \Big\{\, \big(\bsig^{p}(\bX)_{t}\big)_{t\in[0,T]}\;:\; \bX^- \in C^\alpha([0,T];\RR^d)\,\Big\}
\;\subset\; C \ \big([0,T],\,\BB^{(p)}(\RR^{1+d})\big).
\]
Let $\mathcal{H}\subset \mathcal S^{(p)}$ be compact and $f:\mathcal{H}\to\RR$ continuous. Then for every $\varepsilon>0$ there exists $\ell$ such that
\[
\sup_{\,(\bsig^{p}(\bX)_{t})_{t\in[0,T]}\in \mathcal{H}}
\Big|\, f\big((\bsig^{p}(\bX)_{t})_{t\in[0,T]}\big)
\;-\; \mathbf{S}^{\circ m}_\ell(\Bar{\bX})_{0T} \,\Big| \;<\; \varepsilon,
\]
where \(m\) is sufficiently large and \(\ell = \{\ell^{(1)}, \dots, \ell^{(m)}\}\) is the set of learnable parameters with \(\ell^{(1)}\) to be the parameters for the first model, \(\ell^{(2)}\) for the second and so on.
\end{thm}

\begin{proof}
Using Corollary ~\eqref{cor for UAT} with the usual notation and let $\mathcal{H}\subset \mathcal S^{(p)}$ be compact and $f:\mathcal{H}\to\RR$ continuous. Then for every $\varepsilon>0$, there exists $h \in \hopf$ such that
\[ \sup_{\,(\bsig^{p}(\bX)_{t})_{t\in[0,T]}\in \mathcal{H}}
\Big|\, f\big((\bsig^{p}(\bX)_{t})_{t\in[0,T]}\big)
\;-\; \langle \sig(\Bar{\bX})_{T},\, \Psi(h) \rangle \,\Big| \;<\; \varepsilon. \]
By Lemma ~\eqref{classical_iterated}, every component of the classical signature of some higher level can be replicated by iteratively applying classical signature model of lower level, e.g., \(k\). Therefore,
\[\langle \sig(\Bar{\bX})_{T},\, \Psi(h) \rangle= \mathbf{S}^k_{\ell^{(m)}}\left(\mathbf{S}^k_{\ell^{(m-1)}}\left(\dots \mathbf{S}^k_{\ell^{(1)}}\left(\Bar{\bX}\right)_{0T}\right)_{0T}\right)_{0T}.\]
When \(m\) is sufficiently large. Then for any $h \in \hopf$, \(\langle \sig(\Bar{\bX})_{T},\, \Psi(h) \rangle\) can be recovered by learning the parameters \(\ell^{(1)}, \dots, \ell^{(m)}\). Hence, we conclude the claim and show the algorithm in Figure \eqref{fig1-intro}.
\end{proof}

\begin{center}
\begin{figure}[h]
\begin{tikzpicture}[
SIR/.style={rectangle, rounded corners,  draw=red!60, fill=red!5, very thick, minimum size=10mm},
]
\node[SIR]    (A1)                     {$\bX_{t}:$ Data};
\node[SIR]    (A2)       [right=of A1] {$\Bar{\bX}_{t}$};
\node[SIR]    (A3)       [right= 1.5cm of A2] {$\mathbf{S}_{\ell}^2(\Bar{\bX})_{0t}$};
\node[SIR]    (A4)       [right= 1.5cm of A3] {$\mathbf{S}^2_{\ell^{(m)}}\left(\mathbf{S}^2_{\ell^{(m-1)}}\left(\dots \mathbf{S}^2_{\ell^{(1)}}\left(\bar{\bX}\right)_{0t}\right)_{0t}\right)_{0t}$};
\draw[->, thick] (A1.east)  to node[above] {$\Psi$} (A2.west);
\draw[->, thick] (A2.east)  to node[above] {$\mathbf{S}_{\ell}^2(\cdot)$} (A3.west);
\draw[->, thick, dashed] (A3.east)  to node[above] {$\mathbf{S}_{\ell}^2(\cdot)$} (A4.west);
\end{tikzpicture}
\caption{Application of Level 2 signature model on extended path $\Bar{\bX}$}
\label{fig1-intro}
\end{figure}
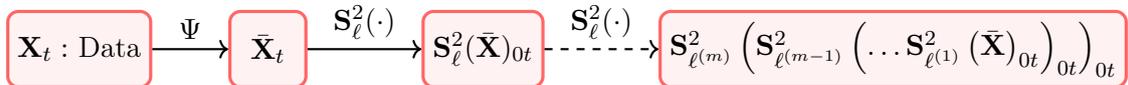
\end{center}

\section{Construction of the extended path} \label{section extension}
In this section, we present a systematic procedure for constructing the extended path \(\Bar{\bX}\). For an \(\alpha\)-H\"older path \(\bX:[0,T]\to\RR^d\), we develop two complementary routes: (i) an analytic specification of integration rules that yields a non-geometric (branched) rough-path enhancement, and (ii) a non-analytic, data-driven construction learned via a neural network. The subsections below treat each approach in turn.
\subsection{Explicit construction}
Let \(\bX:[0,T]\to\RR^d\) be an \(\alpha\)-H\"older path with \(\alpha\in\bigl(\tfrac{1}{4},\tfrac{1}{3}\bigr]\).
Set \(N:=\lfloor 1/\alpha\rfloor=3\); thus only components up to order \(3\) are relevant. We construct the extended path \(\Bar{\bX}\) using solely the Hairer--Kelly morphism i.e., fix the Hopf algebra morphism $\Psi: (\hopf, \cdot, \Delta) \to (T(\mathcal{V}), \shuffle, \Bar{\Delta})$ and define \(\Bar{\bX}\) so that
\[
\big\langle \sig(\Bar{\bX})_{st},\,\psi(h)\big\rangle
=\big\langle \bsig(\bX)_{st},\,h\big\rangle
,\qquad\text{for all rooted trees } h \text{ with }|h|\le 3.
\]
In words, \(\Bar{\bX}\) is chosen so that its geometric iterated integrals coincide with the \(\psi\)-image of the branched signature of \(\bX\) up to level \(3\). We now describe the resulting level-1, level-2, and level-3 coordinates.

To begin with, for level-1,  let $h = \Forest{decorated[a]}$, where \(a \in \cS\)-the alphabet set. Using the definition of the \(\Psi\) we have $\Psi(\Forest{decorated[a]}) = \Forest{decorated[a]}$ i.e., \(\innerprod{\bsig(\bX)_{st}, \Forest{decorated[a]}} = \innerprod{\sig(\Bar{\bX})_{st},\Forest{decorated[a]}}\), which gives \(\Bar{\bX}^{\Forest{decorated[a]}}_t - \Bar{\bX}^{\Forest{decorated[a]}}_s = \bX^{\Forest{decorated[a]}}_t - \bX^{\Forest{decorated[a]}}_s\). Similarly, for $h = \Forest{decorated[b[a]]}$, $\Psi(\Forest{decorated[b[a]]}) = \Forest{decorated[b[a]]} + \Forest{decorated[a]} \otimes \Forest{decorated[b]}$ i.e., \(\innerprod{\bsig(\bX)_{st},\Forest{decorated[b[a]]}} = \innerprod{\sig(\Bar{\bX})_{st}, \Psi(\Forest{decorated[b[a]]})} = \innerprod{\sig(\Bar{\bX})_{st}, \Forest{decorated[b[a]]} + \Forest{decorated[a]} \otimes \Forest{decorated[b]}}\), which gives
\begin{equation*}
\Bar{\bX}^{\Forest{decorated[b[a]]}}_t - \Bar{\bX}^{\Forest{decorated[b[a]]}}_s = \int_s^t  \int_s^{r_2} d\bX_{r_1}^{\Forest{decorated[a]}}d\bX_{r_2}^{\Forest{decorated[b]}} - \int_s^t  \int_s^{r_2} d\Bar{\bX}_{r_1}^{\Forest{decorated[a]}}d\Bar{\bX}_{r_2}^{\Forest{decorated[b]}}.
\end{equation*}
Now, for $h = \Forest{decorated[c[a][b]]}$, \(\Psi(\Forest{decorated[c[a][b]]}) = \Forest{decorated[c[a][b]]} + \Forest{decorated[a]} \otimes \Forest{decorated[c[b]]} + \Forest{decorated[b]} \otimes \Forest{decorated[c[a]]} + \Forest{decorated[a]} \otimes\Forest{decorated[b]} \otimes\Forest{decorated[c]} + \Forest{decorated[b]} \otimes\Forest{decorated[a]} \otimes\Forest{decorated[c]}\), that is
\begin{equation*}
\innerprod{\bsig(\bX)_{st},\Forest{decorated[c[a][b]]}} = \innerprod{\sig(\Bar{\bX})_{st}, \Psi(\Forest{decorated[c[a][b]]})} = \innerprod{\sig(\Bar{X})_{st}, \Forest{decorated[c[a][b]]} + \Forest{decorated[a]} \otimes \Forest{decorated[c[b]]} + \Forest{decorated[b]} \otimes \Forest{decorated[c[a]]} + \Forest{decorated[a]} \otimes \Forest{decorated[b]} \otimes \Forest{decorated[c]} + \Forest{decorated[b]} \otimes \Forest{decorated[a]} \otimes \Forest{decorated[c]}}.
\end{equation*}
This gives
\begin{align*}
\Bar{X}^{\Forest{decorated[c[a][b]]}}_t - \Bar{X}^{\Forest{decorated[c[a][b]]}}_s = & \int_s^t \left( \int_s^{r_2} dX_{r_1}^{\Forest{decorated[a]}}\right) \left(\int_s^{r_2} dX_{r_1}^{\Forest{decorated[b]}}\right) dX_{r_2}^{\Forest{decorated[c]}} - \int_s^t  \int_s^{r_2} d\Bar{X}_{r_1}^{\Forest{decorated[a]}}d\Bar{X}_{r_2}^{\Forest{decorated[c[b]]}} - \int_s^t  \int_s^{r_2} d\Bar{X}_{r_1}^{\Forest{decorated[b]}}d\Bar{X}_{r_2}^{\Forest{decorated[c[a]]}}\\
& - \int_s^t \int_s^{r_3} \int_s^{r_2} d\Bar{X}_{r_1}^{\Forest{decorated[a]}}d\Bar{X}_{r_2}^{\Forest{decorated[b]}} d\Bar{X}_{r_3}^{\Forest{decorated[c]}}  - \int_s^t \int_s^{r_3} \int_s^{r_2} d\Bar{X}_{r_1}^{\Forest{decorated[b]}}d\Bar{X}_{r_2}^{\Forest{decorated[a]}} d\Bar{X}_{r_3}^{\Forest{decorated[c]}}.
\end{align*}
Finally, for $h = \Forest{decorated[c[b[a]]]}$, \(\Psi(\Forest{decorated[c[b[a]]]}) = \Forest{decorated[c[b[a]]]} + \Forest{decorated[a]} \otimes \Forest{decorated[c[b]]} + \Forest{decorated[b[a]]} \otimes \Forest{decorated[c]} + \Forest{decorated[a]} \otimes\Forest{decorated[b]} \otimes\Forest{decorated[c]}\) i.e.,
\begin{equation*}
\innerprod{\bsig(\bX)_{st},\Forest{decorated[c[b[a]]]}} = \innerprod{\sig(\Bar{\bX})_{st}, \Psi(\Forest{decorated[c[b[a]]]})} = \innerprod{\sig(\Bar{\bX})_{st}, \Forest{decorated[c[b[a]]]} + \Forest{decorated[a]} \otimes \Forest{decorated[c[b]]} + \Forest{decorated[b[a]]} \otimes \Forest{decorated[c]} + \Forest{decorated[a]} \otimes\Forest{decorated[b]} \otimes\Forest{decorated[c]}}.
\end{equation*}
So the required component is
\begin{align*}
\Bar{\bX}^{\Forest{decorated[c[b[a]]]}}_t - \Bar{\bX}^{\Forest{decorated[c[b[a]]]}}_s = & \int_s^t \int_s^{r_3} \int_s^{r_2} d\bX_{r_1}^{\Forest{decorated[a]}}  d\bX_{r_2}^{\Forest{decorated[b]}} d\bX_{r_3}^{\Forest{decorated[c]}} - \int_s^t  \int_s^{r_2} d\Bar{\bX}_{r_1}^{\Forest{decorated[a]}}d\Bar{\bX}_{r_2}^{\Forest{decorated[c[b]]}} - \int_s^t  \int_s^{r_2} d\Bar{\bX}_{r_1}^{\Forest{decorated[b[a]]}}d\Bar{\bX}_{r_2}^{\Forest{decorated[c]}}\\
& - \int_s^t \int_s^{r_3} \int_s^{r_2} d\Bar{\bX}_{r_1}^{\Forest{decorated[a]}}d\Bar{\bX}_{r_2}^{\Forest{decorated[b]}} d\Bar{\bX}_{r_3}^{\Forest{decorated[c]}}.
\end{align*}
Hence the extended path $\Bar{\bX}$ is
\begin{equation*}
\Bar{\bX} = \left(\Bar{\bX}^{\Forest{decorated[a]}}, \Bar{\bX}^{\Forest{decorated[b[a]]}}, \Bar{\bX}^{\Forest{decorated[c[a][b]]]}}, \Bar{\bX}^{\Forest{decorated[c[b[a]]]}}\right)_{a,\,b,\,c \, \in \, \cS},
\end{equation*}
where the path components are constructed explicitly. Next, we will give a couple of examples where this explicit construction works.
\subsubsection{Multi-dimensional Brownian motion}
Let the underlying path \(\bX\) be an \(d\)-dimensional Brownian motion \(\bB\). 
Since, Brownian motion is $\alpha$-H\"older for any \(\alpha < \frac{1}{2}\) so the number of components that actually matter to get the extended path is  \(N=2\). With the help of the previous explicit construction, the components of the extended path are
\begin{align*}
\Bar{\bB}_t^{\Forest{decorated[a]}} - \Bar{\bB}_s^{\Forest{decorated[a]}} &= \bB_t^{\Forest{decorated[a]}} - \bB_s^{\Forest{decorated[a]}}\\
\Bar{\bB}_t^{\Forest{decorated[b[a]]}} - \Bar{\bB}_s^{\Forest{decorated[b[a]]}} &= \int_s^t  \int_s^{r_2} d\bB_{r_1}^{\Forest{decorated[a]}}d\bB_{r_2}^{\Forest{decorated[b]}} - \int_s^t  \int_s^{r_2} d\Bar{\bB}_{r_1}^{\Forest{decorated[a]}}d\Bar{\bB}_{r_2}^{\Forest{decorated[b]}} = - \frac{1}{2}[\bB^{\Forest{decorated[a]}}_t-\bB^{\Forest{decorated[a]}}_s, \bB^{\Forest{decorated[b]}}_t-\bB^{\Forest{decorated[b]}}_s],
\end{align*}
for \(a,b \in \cS\). Here the iterated integral \(\int_s^t  \int_s^{r_2} d\bB_{r_1}^{\Forest{decorated[a]}}d\bB_{r_2}^{\Forest{decorated[b]}}\) is defined in It\^o sense , while \(\int_s^t  \int_s^{r_2} d\Bar{\bB}_{r_1}^{\Forest{decorated[a]}}d\Bar{\bB}_{r_2}^{\Forest{decorated[b]}}\) is defined in Stratonovich sense. The term \(\frac{1}{2}[\bB^{\Forest{decorated[a]}}_t-\bB^{\Forest{decorated[a]}}_s, \bB^{\Forest{decorated[b]}}_t-\bB^{\Forest{decorated[b]}}_s]\) is nothing but the co-variation of the components of the \(d\)-dimensional Brownian motion on the interval \([s,t]\). Hence, the extended Brownian motion path $\Bar{\bB}$ is
\begin{equation*}
\Bar{\bB} = \left(\Bar{\bB}^{\Forest{decorated[a]}}, \Bar{\bB}^{\Forest{decorated[b[a]]}}\right)_{a, \, b,\, \in \, \cS}.
\end{equation*}
In case of $1$-dimensional Brownian motion this reduces to 
\begin{equation*}
\Bar{\bB}_{st} = \left(\Bar{\bB}^{\Forest{decorated[a]}}, \Bar{\bB}^{\Forest{decorated[a[a]]}}\right)_{st} = \left(\bB^{\Forest{decorated[a]}}_t - \bB^{\Forest{decorated[a]}}_s, -\frac{1}{2}(t-s)\right),
\end{equation*}
where the second component is nothing but the It\^o-Stratonovich correction.

\subsubsection{Multi-dimensional fractional Brownian motion}
Let the underlying path \(\bX\) be a \(d\)-dimensional fractional Brownian motion \(\bB^H=(B^{H,\Forest{decorated[a]}})_{a\in\cS}\) with Hurst index \(H\in(\tfrac14,\tfrac13]\) and correlation matrix \(\rho=(\rho_{ab})_{a,b\in\cS}\), so that \(\mathrm{Cov}(B^{H,\Forest{decorated[a]}}_t,B^{H,\Forest{decorated[b]}}_t)=\rho_{ab}\,t^{2H}\).
Since fBm is \(\alpha\)-Hölder for any \(\alpha<H\), the number of components that actually matter to get the extended path is \(N=3\).
With the help of the previous explicit construction, the components of the extended path are
\[
\Bar{\bB}^{H,\Forest{decorated[a]}}_t-\Bar{\bB}^{H,\Forest{decorated[a]}}_s
= \bB^{H,\Forest{decorated[a]}}_t-\bB^{H,\Forest{decorated[a]}}_s,
\]
and
\[
\begin{aligned}
\Bar{\bB}^{H,\Forest{decorated[b[a]]}}_t-\Bar{\bB}^{H,\Forest{decorated[b[a]]}}_s
&= \int_s^t\!\!\int_s^{r_2} d\bB^{H,\Forest{decorated[a]}}_{r_1}\,d\bB^{H,\Forest{decorated[b]}}_{r_2}
\;-\;
\int_s^t\!\!\int_s^{r_2} d\Bar{\bB}^{H,\Forest{decorated[a]}}_{r_1}\,d\Bar{\bB}^{H,\Forest{decorated[b]}}_{r_2}
\\
&= -\tfrac12\,\rho_{ab}\,(t^{2H}-s^{2H}),
\end{aligned}
\]
for \(a,b\in\cS\). Here the first iterated integral is the canonical (Gaussian/Wick–Skorohod) double integral, while the second is the corresponding Stratonovich/rough integral along \(\Bar{\bB}^H\), and the difference is the normal-ordering correction determined by the covariance \cite{Janson_1997}. See also \cite{hu2013stratonovich} for explicit connection between Stratonovich and Skorohod integrals of multidimensional Gausssian processes.
For third order, write \(R_{ac}(t,t)=\rho_{ac}\,t^{2H}\) and note \(\partial_t R_{ac}(t,t)=2H\,\rho_{ac}\,t^{2H-1}\). Then, for \(a,b,c\in\cS\),
\[
\begin{aligned}
\Bar{\bB}^{H,\Forest{decorated[c[b[a]]]}}_t-\Bar{\bB}^{H,\Forest{decorated[c[b[a]]]}}_s
&= \int_s^t\!\!\int_s^{r_3}\!\!\int_s^{r_2}
d\bB^{H,\Forest{decorated[a]}}_{r_1}\,
d\bB^{H,\Forest{decorated[b]}}_{r_2}\,
d\bB^{H,\Forest{decorated[c]}}_{r_3}
\\[-2pt]
&\quad -\;
\int_s^t\!\!\int_s^{r_2}
d\Bar{\bB}^{H,\Forest{decorated[a]}}_{r_1}\,
d\Bar{\bB}^{H,\Forest{decorated[c[b]]}}_{r_2}
\;-\;
\int_s^t\!\!\int_s^{r_2}
d\Bar{\bB}^{H,\Forest{decorated[b[a]]}}_{r_1}\,
d\Bar{\bB}^{H,\Forest{decorated[c]}}_{r_2}
\\
&\quad -\;
\int_s^t\!\!\int_s^{r_3}\!\!\int_s^{r_2}
d\Bar{\bB}^{H,\Forest{decorated[a]}}_{r_1}\,
d\Bar{\bB}^{H,\Forest{decorated[b]}}_{r_2}\,
d\Bar{\bB}^{H,\Forest{decorated[c]}}_{r_3}
\\
&= -\,H\,\rho_{ac}\,\int_s^t \bB^{H,\Forest{decorated[b]}}_{r}\, r^{2H-1}\,dr,
\end{aligned}
\]
and for the tree \(h = \Forest{decorated[c[a][b]]}\),
\[
\begin{aligned}
\Bar{\bB}^{H,\Forest{decorated[c[a][b]]}}_t-\Bar{\bB}^{H,\Forest{decorated[c[a][b]]}}_s
&= \int_s^t
\Big(\int_s^{r_2} d\bB^{H,\Forest{decorated[a]}}_{r_1}\Big)
\Big(\int_s^{r_2} d\bB^{H,\Forest{decorated[b]}}_{r_1}\Big)
d\bB^{H,\Forest{decorated[c]}}_{r_2}
\\[-2pt]
&\quad -\;
\int_s^t\!\!\int_s^{r_2}
d\Bar{\bB}^{H,\Forest{decorated[a]}}_{r_1}\,
d\Bar{\bB}^{H,\Forest{decorated[c[b]]}}_{r_2}
\;-\;
\int_s^t\!\!\int_s^{r_2}
d\Bar{\bB}^{H,\Forest{decorated[b]}}_{r_1}\,
d\Bar{\bB}^{H,\Forest{decorated[c[a]]}}_{r_2}
\\
&\quad -\;
\int_s^t\!\!\int_s^{r_3}\!\!\int_s^{r_2}
d\Bar{\bB}^{H,\Forest{decorated[b]}}_{r_1}\,
d\Bar{\bB}^{H,\Forest{decorated[a]}}_{r_2}\,
d\Bar{\bB}^{H,\Forest{decorated[c]}}_{r_3}
\\
&\quad -\;
\int_s^t\!\!\int_s^{r_3}\!\!\int_s^{r_2}
d\Bar{\bB}^{H,\Forest{decorated[a]}}_{r_1}\,
d\Bar{\bB}^{H,\Forest{decorated[b]}}_{r_2}\,
d\Bar{\bB}^{H,\Forest{decorated[c]}}_{r_3}
\\
&= -\,H\,\rho_{bc}\,\int_s^t \bB^{H,\Forest{decorated[a]}}_{r}\, r^{2H-1}\,dr
\;-\;H\,\rho_{ac}\,\int_s^t \bB^{H,\Forest{decorated[b]}}_{r}\, r^{2H-1}\,dr.
\end{aligned}
\]

Hence, the extended fractional Brownian motion path \(\Bar{\bB}^H\) (up to level \(3\)) is
\[
\Bar{\bB}^H
=
\Big(\Bar{\bB}^{H,\Forest{decorated[a]}},
\ \Bar{\bB}^{H,\Forest{decorated[b[a]]}},
\ \Bar{\bB}^{H,\Forest{decorated[c[b[a]]]}},
\ \Bar{\bB}^{H,\Forest{decorated[c[a][b]]}}\Big)_{a,b,c\in\cS}.
\]
In case of \(1\)-dimensional fractional Brownian motion this reduces to
\[
\Bar{\bB}^H_{st}
=
\left(
\bB^H_t-\bB^H_s,\;
-\tfrac12\,(t^{2H}-s^{2H}),\;
-\,H\!\int_s^t \bB^H_r\, r^{2H-1}\,dr
\right),
\]
where the second and third components are the covariance-driven normal-ordering corrections associated with \(R(t,t)=t^{2H}\).

\subsection{Data-driven construction learned via a neural network} \label{data_driven}
Because the driving noise of the primary process is unknown i.e., whether it is Brownian motion, fractional Brownian motion, or some other stochastic input, we cannot prescribe in advance how the extended path should be constructed. Instead, we adopt a supervised-learning approach within a neural-network framework. Concretely, we observe a response \(\bY(t)\) (e.g., the solution of an SDE/CDE/RDE) that is driven by a signal \(\bX(t)\). Learning \(\bY(t)\) directly from the classical (geometric) signature of \(\bX(t)\) may be insufficient, since certain interactions are only captured by branched signatures and are not recoverable from purely geometric features.

Rather than constructing a branched signature explicitly, we learn a parametric extension of the primary signal, \(t \mapsto \bar{\bX}_\theta(t)\), with \(\bar{\bX}_\theta(t) \in \mathbb{R}^{m}\), where these \(m\) latent coordinates are learned to encode the non-geometric information that a branched signature would otherwise carry. Once we have access to the neural-network output \(\bar{\bX}_\theta(t)\), we concatenate this with the actual path \(\bX(t)\) and define \(\bbX_{\theta}(t) := (\bX(t), \bfX_{\theta}(t)) \in \RR^{d+m}\). The reason to concatenate the actual path is to be consistent with the extension map defined in the previous section i.e., the extension given by \cite{hairer2015geometric}. After this concatenation, we apply the classical signature of some order \(k > 1\) to this extended path and fit \(\bY(t)\) from these features. The training is performed with a loss function that is a combination of the physics-informed loss and the shuffle property loss. The physics-informed loss balances the data fit and is given as 
\begin{equation} \label{physics_informed loss}
\mathcal{L}_{\mathrm{physics-informed}}(\theta,\phi) := \frac{1}{N}\sum_{t_i \in \pi[0,T]} \Big\| \bY_{t_i} - g_\phi\!\big(\sig^{k}(\bbX_\theta)_{0t_i}\big)\Big\|^2,
\end{equation}
where \(\pi[0,T]\) is some partition of the interval of consideration and \(g_{\phi}\) is some predictor function like a linear layer etc. and \(N\) is the length of the partition \(\pi[0,T]\). Shuffle property loss ensures that the iterated integrals satisfy the integration by parts (shuffle) property. Here we don't rely on the signature computed via \textit{iisignature} or \textit{signatory} as they are always geometric because they use idea of Stratonovich integration. Instead, we compute the integrals using left-hand point Riemann sum similar to It\^o integration. The corresponding loss function is given as follows
\begin{equation} \label{shuffle product loss}
\mathcal{L}_{\mathrm{shuffle}}(\theta) := \frac{1}{N} \sum_{t_i \in \pi[0,T]} \sum_{j,k \in \Bar{\cS}} \Big\| \Delta \Bar{\bX}^{j}_{\theta}(t_0,t_i)\Delta \Bar{\bX}^{k}_{\theta}(t_0,t_i) - \int_{t_0}^{t_i} \Delta \Bar{\bX}^{j}_{\theta}(t_0,s)d \bar{\bX}^k_\theta(s) - \int_{t_0}^{t_i} \Delta \Bar{\bX}^{k}_{\theta}(t_0,s)d \bar{\bX}^j_\theta(s) \Big\|^2,
\end{equation}
where \(\Delta \Bar{\bX}^{k}_{\theta}(t_0,t_i) := \bar{\bX}^k_\theta(t_i) - \bar{\bX}^k_\theta(t_0)\), \(\Bar{\bX}^{k}_{\theta}\) is \(k\)-th component of \(\Bar{\bX}_{\theta}\), and \(\Bar{\cS}\) is set of cardinality \(m\) and is the alphabet set over the components of \(\Bar{\bX}_{\theta}\). Also, the integral inside the shuffle loss \(\mathcal{L}_{\mathrm{shuffle}}(\theta)\) is defined as follows
\[\int_{t_0}^{t_i} \Delta \Bar{\bX}^{j}_{\theta}(t_0,s)d \bar{\bX}^k_\theta(s) = \sum_{l=1}^m \left(\bar{\bX}^j_\theta(s_{l-1}) - \bar{\bX}^j_\theta(s_0) \right) \left(\bar{\bX}^k_\theta(s_{l}) - \bar{\bX}^k_\theta(s_{l-1}) \right), \]
where the sum is over the partition of \([0,t_i]\) i.e., \(\{0=s_0 < s_1 < \dots < s_m = t_i\}\).

With this, the total loss function for the training becomes
\begin{equation} \label{total loss}
\mathcal{L}(\theta,\phi) = \lambda_p \mathcal{L}_{\mathrm{physics-informed}}(\theta,\phi) + \lambda_s \mathcal{L}_{\mathrm{shuffle}}(\theta),
\end{equation}
where \(\lambda_p\) and \(\lambda_s\) are weights corresponding to physics-informed loss and shuffle loss respectively.  These weights can be chosen wisely to train the model efficiently. Finally, this strategy allows the extended geometric signature \(\sig^{k}(\bbX_\theta)\) to emulate the expressive content of a branched signature while remaining trainable end-to-end from data.

\section{Application I: rough volatility calibration}\label{section numerical}


In this section, we will present an experiment to validate our data-driven construction method of path extension. Our experiment will primarily be related to stock and variance path calibration. The most general form of the volatility model that we will consider for our experiments follows the coupled dynamics as below,
\begin{equation} \label{volitility model}
\left\{
\begin{aligned}
dS_t & = f_1(S_t,V_t,t)\,dt \;+\; f(S_t,V_t,t)\,d\bB_t ,\\[6pt]
dV_t & = g_1(S_t,V_t,t)\,dt \;+\; g_2(S_t,V_t,t)\,d\bB_t \;+\; g(S_t,V_t,t)\,d\bB^H_t,
\end{aligned}
\right.
\end{equation}
where \(f_1, f, g_1, g_2, g : \RR^2 \times [0,\infty) \to \RR\) are sufficiently smooth functions. Here \(S_t\) shows the asset price process while \(V_t\) is the variance process and their combined dynamics is driven by the process \((t, \bB, \bB^H)\), where \(\bB\) is the Brownian motion and \(\bB^H\) is the fractional Brownian motion with Hurst parameter \(H\).

In practice, the asset price process and the variance process are typically correlated. In Eq.~\eqref{volitility model}, this dependence is built by introducing the same Brownian motion term \(\bB\). Moreover, the dependence of \(V_t\) on \(S_t\) is a particular instance of volatility model where  (spot) volatility depends on the past of the price trajectory. Such type of path dependent volatility models are considered in \cite{guyon2023volatility}. If the fractional Brownian motion \(\bB^H\) is replaced by another independent Brownian motion \(\Bar{\bB}\) then we recover many classical models - e.g., the Stein-Stein Model \cite{stein1991stock}, the Heston model \cite{heston1993closed}, the Bergomi model \cite{bergomi2005smile}  etc. However, there are many rough volatility models where fractional Brownian motion \(\bB^H\) appears and drives the variance process. For example, when there is no Brownian term in the dynamics of \(V_t\) i.e., \(g_2 = 0\), we recover the rough Heston model \cite{el2019characteristic}, the rough Bergomi model \cite{bayer2016pricing}, the quadratic rough Heston model \cite{gatheral2020quadratic} etc.

As a special case of the general coupled system ~\eqref{volitility model} and motivated by different volatility models discussed above, we consider its driftless version obtained by setting \(g_1=g_2=0\) and by choosing the variance in quadratic form. More precisely, we let the stock price process \(S_t\) and the auxiliary volatility process \(Z_t\) satisfy
\begin{equation}\label{rough vol}
\left\{
\begin{aligned}
\frac{dS_t}{S_t}
&=
\sqrt{a(Z_t-b)^2+c} \, d\bB_t
-\frac12\bigl(a(Z_t-b)^2+c\bigr)\,dt,\\[6pt]
dZ_t
&=
\lambda\eta\sqrt{a(Z_t-b)^2+c} \, d\bB_t^H,
\end{aligned}
\right.
\end{equation}
where \(S_0=s_0>0\), \(Z_0=z_0\), and \(a,b,c,\lambda,\eta\) are positive constants. In this model, the instantaneous variance process is given by
\[
V_t=a(Z_t-b)^2+c.
\]
Thus, the auxiliary volatility is driven only by fractional Brownian motion, with neither a drift term nor a standard Brownian motion term. The motivation for choosing this driftless model is that it allows us to recover the fractional Brownian motion from the auxiliary volatility process by first estimating the Hurst index and then leveraging the \(p-\)variation along the partition on which the data are observed.

After the model selection, let us discuss the numerical simulation, estimation of the primary process, learning of the extension map and calibration in the subsequent subsections. 

\subsection{Simulation}
To simulate the price and variance process, we select the parameters in the model to be \(a= 0.002, b = 0.1, c = 0.08, \lambda = 0.15,\) and \( \eta = 0.15\). We simulate paths of  Brownian motion \(\bB_t\) and fractional Brownian motion \(\bB^H_t\) with Hurst parameter \(H=0.3\) of a length \(N\) with \(N=1,000,000\) on the interval \([0,1]\). Fractional Brownian motion path is simulated  using Davies-Harte method \cite{davies1987tests}. Furthermore, we use simple Euler-Maruyama scheme to simulate the price and variance path i.e., after fixing \(S_0 = 1.0\), \(V_0=0.8\), and  \(Z_0=\sqrt{\frac{V_0-c}{a}}+b\), we run the following for \(n=0,1,\cdots,N-1\).

\begin{equation*}\label{euler}
\left\{
\begin{aligned}
S_{n+1} &= S_n + f_1(S_n,Z_n,t_n)\,\Delta t + f(S_n,Z_n,t_n)\,\Delta \bB_{n+1},\\
Z_{n+1} &= Z_n + g(S_n,Z_n,t_n)\,\Delta \bB^H_{n+1},\\
V_{n+1} & = a(Z_{n+1}-b)^2 +c
\end{aligned}
\right.
\end{equation*}
where \(t_n = n\,\Delta t, \; \Delta \bB_{n+1} := \bB_{n+1} - \bB_{n}, \; \Delta \bB^H_{n+1} := \bB^H_{n+1} - \bB^H_{n}\) and the functions \(f_1, f\) and \(g\) are defined as in Eq.~\eqref{rough vol}.

\subsection{Estimating primary process}
To begin with, we construct the \(\mathbb Q\)-Brownian motion increments and hence the \(\mathbb Q\)-Brownian \(\bB^{\mathbb Q}_t\). On a small interval \([t,t+\Delta]\), we estimate the diffusion function \(f(S,Z,t)\) by using the high frequency estimate of the quadratic variation of the price process \(S_t\) along a partition $\Pi := \{\Pi^M\}_{M=0}^\infty:=\{t=t^M_0<t^M_1\cdots < t^M_i<\cdots<t^M_{\#(\Pi^M)}=t+\Delta\}_{M=0}^\infty$ with vanishing mesh (i.e. $|\Pi| := |t^M_{i+1}-t^M_i|\xrightarrow[]{M\to \infty} 0$). The estimation is given as
\begin{align}
    (f(S_t,Z_t,t))^2 \approx \frac{1}{\Delta} \sum_{\Pi^M\cap [t,t+\Delta]} (S(t^M_{i+1}) - S(t^M_i))^2.
\end{align}
Once the vector of diffusion values \(f(S_t,Z_t,t)\) has been estimated, we can construct the \(\mathbb Q\)-Brownian motion increments over any interval \([t,t+\Delta]\), with \(\Delta>0\), via
\begin{align}
\bB_{t+\Delta}^{\mathbb Q}-\bB_t^{\mathbb Q}
=
\int_t^{t+\Delta}\frac{dS_u}{f(S_u,Z_u,u)}
\approx
\frac{\sqrt{\Delta}\left(S_{t+\Delta} - S_t\right)}{\sqrt{\sum_{\Pi^M\cap [t,t+\Delta]} (S(t^M_{i+1}) - S(t^M_i))^2}}.
\end{align}
Assuming \(\mathbb Q\)-Brownian motion \(\bB^{\mathbb Q}_t\) starts at zero, we recover the path of \(\bB^{\mathbb Q}_t\) using its increments \(\bB_{t+\Delta}^{\mathbb Q}-\bB_t^{\mathbb Q}\).

Secondly, we construct the fractional Brownian motion from the auxiliary volatility process \(Z_t\). To do this, we first estimate the Hurst index \(H\) associated with \(Z_t\), and then use it to recover the fractional Brownian motion \(\bB_t^H\) driving the process. The Hurst index \(H\) is estimated from the auxiliary process \(Z_t\), since \(Z_t\) is directly driven by the fractional Brownian motion \(\bB^H\). More precisely, \(H\) is obtained from the slope of a log--log regression based on the empirical second-order variations of \(Z_t\) across multiple dyadic scales as discussed in \cite{gairing2020hurst}. The hurst index \(H\) used for the experiments in \(H=0.3\) while the one that we estimated is \(\tilde{H}=0.295\). 

We set \(p=\frac{1}{\tilde H}\), and on a small interval \([t,t+\Delta]\), we estimate the diffusion function \(g(S_t,Z_t,t)\) by using the \(p\)-variation of the process \(Z_t\) along the same partition as the one used for the process \(S_t\). This estimation is motivated by Theorem 4.2(i) in \cite{liu2024power}, according to which the leading term of the normalized \(p\)-variation of the process \(Z_t\) is given by
\[
c_p \int_0^t |g(S_u,Z_u,u)|^p\,du,
\]
where
\[
c_p=\mathbb E|N|^p, \qquad N\sim \mathcal N(0,1).
\]
Thus, the estimation is given by
\begin{align}
    \bigl|g(S_t,Z_t,t)\bigr|^p
    \approx
    \frac{1}{c_p\Delta}
    \sum_{\Pi^M\cap [t,t+\Delta]}
    \left| Z(t_{i+1}^M)-Z(t_i^M)\right|^p.
\end{align}
After estimating the coefficient values \(g(S_t,Z_t,t)\), we approximate the increments of the fractional Brownian motion on any interval \([t,t+\Delta]\), where \(\Delta>0\), by
\begin{align}
\bB_{t+\Delta}^H-\bB_t^H
\approx
\frac{(c_p\Delta)^{\tilde H}\bigl(Z_{t+\Delta}-Z_t\bigr)}
{\left(\displaystyle\sum_{\Pi^M\cap [t,t+\Delta]}
\left| Z(t_{i+1}^M)-Z(t_i^M)\right|^p\right)^{\tilde H}}.
\end{align}
Starting from the initial condition \(\bB_0^H=0\), we then construct the full path of \(\bB_t^H\) by successively accumulating these increments.

Finally, for the purpose of calibration in our experiments, we take the time-extended path \(\bX(t) = (t, \bB_t^{\mathbb Q}, \bB_t^H)\) as the primary process.

\subsection{Learning the extension map}
To learn the extended path \(t \mapsto \bar{\bX}_\theta(t)\), we set \(\bbX_\theta(t) := (\bX(t), \bfX_{\theta}(t))\), where the coordinates corresponding to the extended path \(\bar{\bX}_\theta(t)\) are produced by a multi-layer perceptron (MLP) \(\bar{\bX}_\theta(t) \colon \RR^{3}\to\RR^{m}\) as a function of the \(\bX(t)\). In our implementation, \(\bar{\bX}_\theta(t)\) uses six hidden layers with widths \(512\text{--}256\text{--}128\text{--}64\text{--}32\text{--}16\), \(\tanh\) activations throughout, and output dimension \(m=9\). Prediction proceeds sequentially via \emph{layer-wise signature models} i.e., rather than forming a single high-depth signature (which would be computationally expensive as we discussed in subsection \ref{layerwise section}), we apply a depth-\(N_1=2\) signature model to the features \((\bX(t), \bfX_{\theta}(t))_{[0,t]}\) and map it through a linear layer to an scalar \(\widehat{X}_t\); we then augment the true input path with this output from the signature model, apply a second depth-\(N_2=2\) signature model of \((\bX,\widehat{X})_{[0,t]}\), and pass it through a second linear layer to obtain \(\widehat{S}_t\), our estimate of the target stock price path \(S_t\).

The training process minimizes a cost function that is a linear combination of a pathwise calibration loss ~\eqref{physics_informed loss} and a shuffle product loss ~\eqref{shuffle product loss}. Here \(\widehat{S}_{t_i}\) is given by the term \(\sig^{k}(\bbX_\theta)_{0t_i}\) where one signature model of level-\(k\) is replaced by two signature models of each level-\(2\) combined with two linear layers. At each partition time \(t_i\), we pass the current input through the multilayer perceptron, then through the signature models and linear layers. The signature is computed on the prefix path, with previously observed values retained so the model explicitly incorporates the history of the data. Finally, we optimize the total loss \(\mathcal{L}\) using Adams (initial step size \(10^{-3}\) with step decay). The loss corresponding to each epoch is shown in Figure \ref{fig_losses} for number of epochs to be 2000. Backpropagation is performed on all the learnable parameters. Evaluation is performed by a full sequential pass, reporting the path-wise MSE i.e., \(\frac{1}{N+1}\sum_{i=0}^{N}(\widehat{S}_{t_i}-S_{t_i})^2\) and the terminal shuffle product residual on the learned extension. The shuffle product residual matrix corresponding to all the components is shown in Figure \ref{fig_shuffle}.

\begin{figure}[H]
  \centering
  \includegraphics[width=0.82\linewidth]{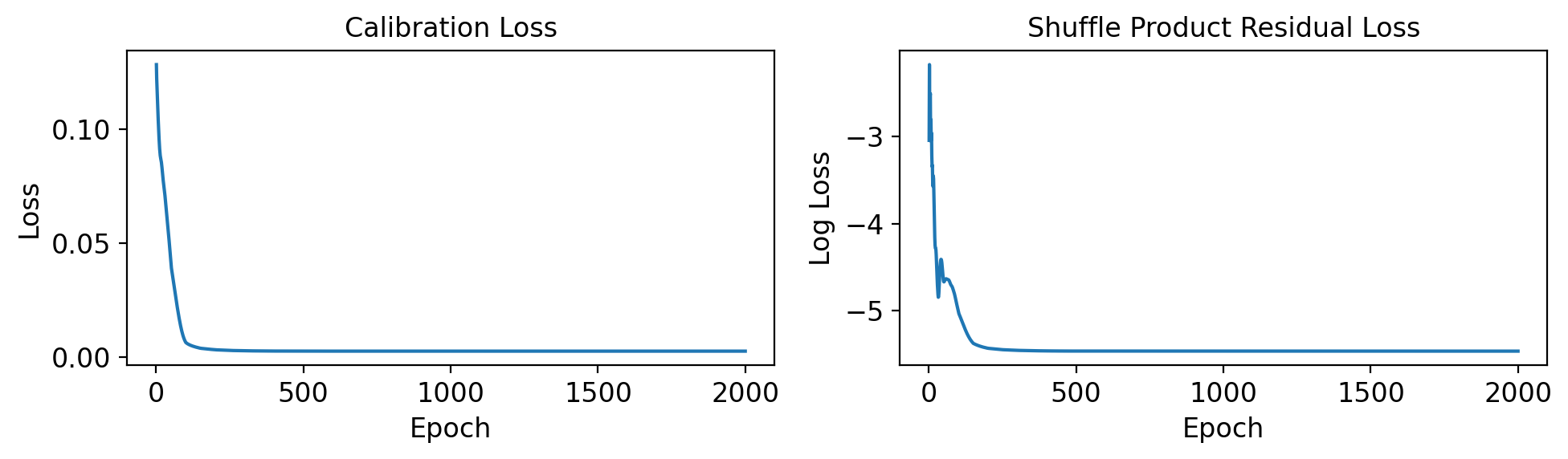}
  \caption{\textit{\textbf{Left:} calibration (path) loss per epoch. \textbf{Right:} shuffle-product residual loss per epoch. The right panel’s vertical axis is shown on a base-10 logarithmic scale.}}
  \label{fig_losses}
\end{figure}

\begin{figure}[H]
  \centering
  \includegraphics[width=0.4\linewidth]{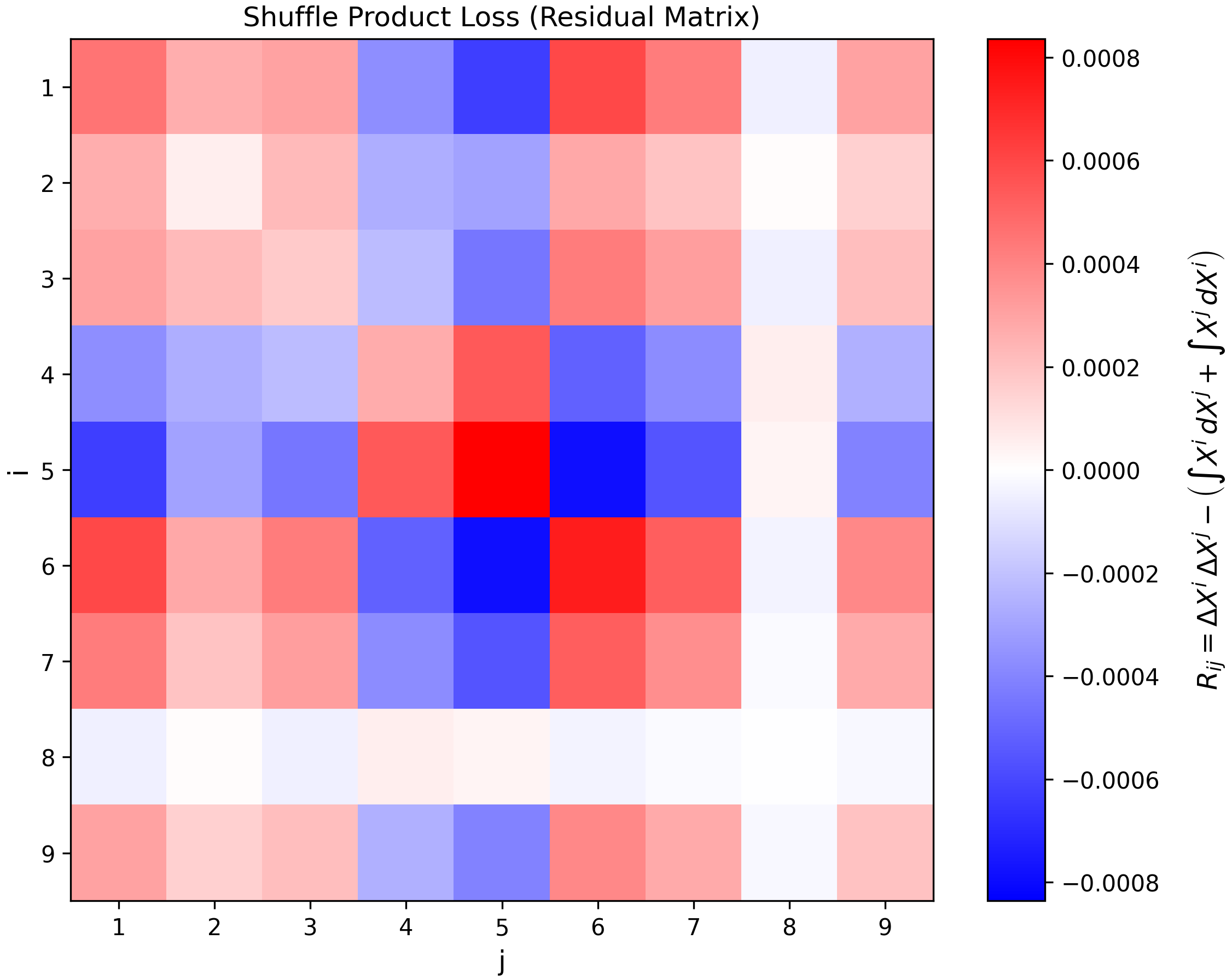}
  \caption{\textit{Shuffle product loss residual matrix. Each cell $(i,j)$ displays 
$R_{ij}=\Delta X^{i}_T\Delta X^{j}_T-\!\left(\int_0^T X^{i}_t\,dX^{j}_t+\int_0^T X^{j}_t\,dX^{i}_t\right)$ 
on the plotted interval. Color encodes sign and magnitude red for positive, blue for negative and white near to $0$; rows/columns are coordinate indices $i$ and $j$. Near-white regions indicate good numerical adherence to the identity, while darker patches show where deviations are large. Final shuffle product MSE mean over all components is \(6.375 \times 10^{-6}\).}}
  \label{fig_shuffle}
\end{figure}

\subsection{Calibration}
To assess the performance of the learned extension to calibrate the underlying path that was used for training i.e., the stock price process path, we stop the further training of the network at certain epochs (say 2000) and assume that we recovered the appropriate parameters \(\theta^*\). We denote the learned extension by \(\bfX_{\theta^*}(t)\) and concatenate it with the actual path as before to get \(\bbX_{\theta^*}(t) = (\bX(t), \bfX_{\theta^*}(t))\). To extract the features from this extended path, we apply a signature of depth 2, fit a linear regression (with intercept) to calibrate the observed process \(S_t\) and record the MSE. As a baseline, we repeat the same procedure using signature of the original path 
\(\bX(t)\) (without the learned extension). We report the mean-squared-error MSE in this case too. Figure \ref{volatility calibration figure} shows both the full-path fits and zoomed-in windows, where the extended model consistently tracks local fluctuations more accurately than the baseline model i.e., the one without the extension.

Similarly, the instantaneous volatility process path \(V_t\) is also regressed against the actual path \(\bX(t)\) and one with extension \(\bbX_{\theta^*}(t) = (\bX(t), \bfX_{\theta^*}(t))\) and MSE is recorded in both cases. In Figure \ref{stock calibration figure}, both the global fits and the zoomed panels demonstrate that the extended model consistently follows local movements more closely than the baseline model i.e., the one without the extension.

\begin{figure}[H]
  \centering
  \begin{subfigure}{0.48\linewidth}
    \centering
    \includegraphics[width=1\linewidth]{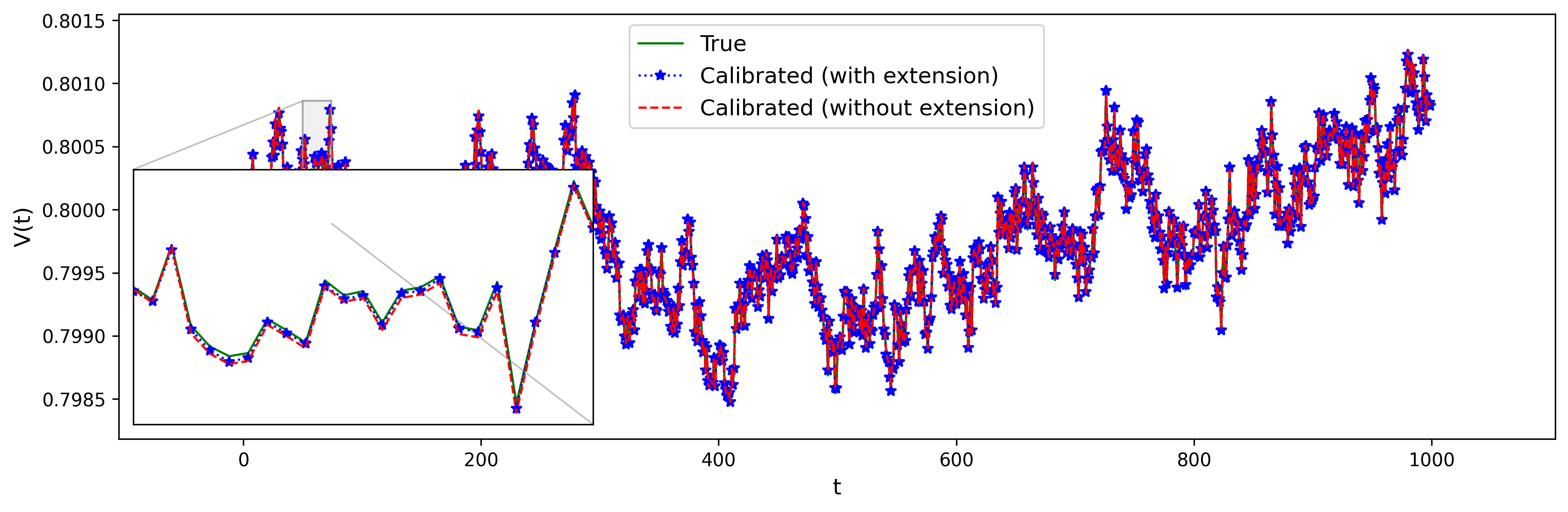}
    \caption{\textit{True vs.\ calibrated volatility process \(V(t)\) using signatures \emph{with} and \emph{without} the learned extension. The inset provides a zoomed view that highlights the finer discrepancies between the two calibrated paths. Reported errors corresponding to with and without extension are \(\mathrm{MSE} = 2.5361 \times 10^{-10}\) and \(\mathrm{MSE} = 3.6666 \times 10^{-10}\) respectively.}}
    \label{volatility calibration figure}
  \end{subfigure}\hfill
  \begin{subfigure}{0.48\linewidth}
    \centering
    \includegraphics[width=1\linewidth]{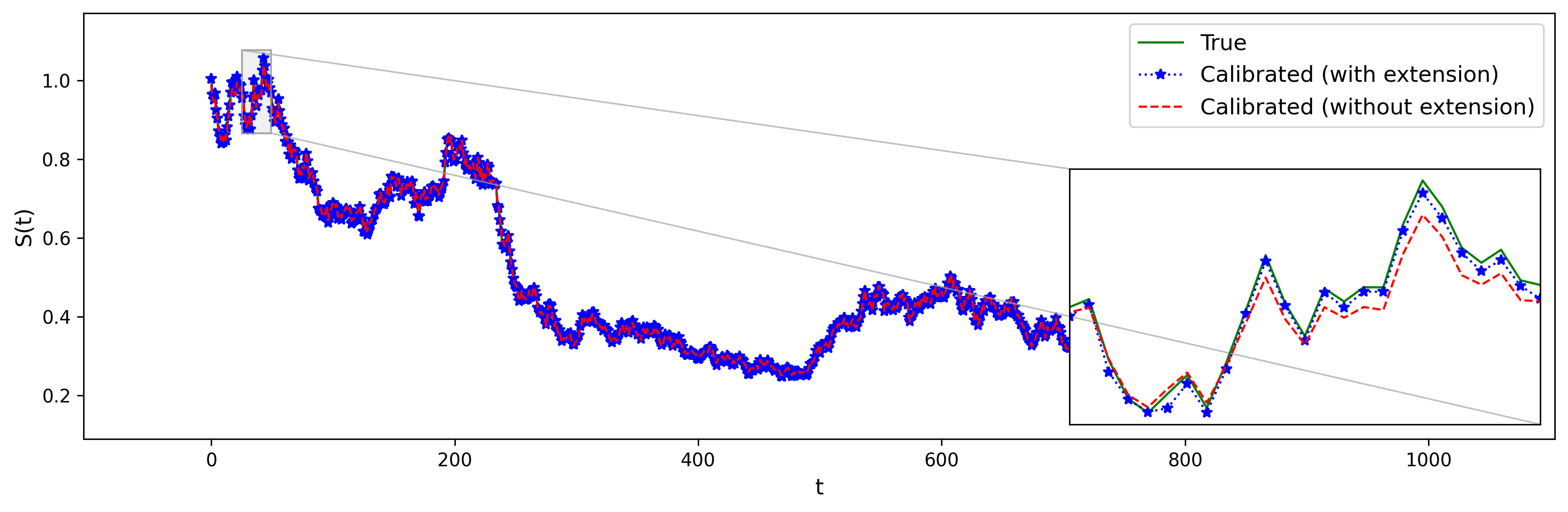}
    \caption{\textit{True vs.\ calibrated stock price path \(S(t)\) using signatures \emph{with} and \emph{without} the learned extension. The inset provides a zoomed view that highlights the finer discrepancies between the two calibrated paths. Reported errors corresponding to with and without extension are \(\mathrm{MSE} = 3.6800 \times 10^{-6}\) and \(\mathrm{MSE} = 2.1877 \times 10^{-5}\) respectively.}}
    \label{stock calibration figure}
  \end{subfigure}
\end{figure}

\section{Application II: Optimal execution problem} \label{optimal execution section}
A common problem in finance is deciding how to sell or acquire a large number of shares while minimizing the risk of large price movements caused by the trader’s own orders. In the literature, this is known as the optimal execution problem. It was originally studied by \cite{bertsimas1998optimal} and \cite{almgren2001optimal}.

More specifically, suppose a trader wants to liquidate \(i_0\) shares of an asset over the time interval \([0,T]\). Rather than selling all shares at once, the trader splits the order over time to reduce adverse price impact. The goal is to determine the optimal trading speed \(u_t\) on \([0,T]\), that is, the rate at which the trader should sell in order to minimize unwanted market impact and execution risk.

To formulate the problem, we first specify a model for the asset dynamics. Let \(X_t\) denote the midprice process and \(V_t\) its (stochastic) volatility process, and assume they satisfy the coupled dynamics given in Eq. ~\eqref{volitility model} with \(S_t\) replaced by \(X_t\) i.e.,
\begin{equation} \label{midprice process}
\left\{
\begin{aligned}
dX_t & = f_1(X_t,V_t,t)\,dt \;+\; f(X_t,V_t,t)\,d\bB_t ,\\[6pt]
dV_t & = g_1(X_t,V_t,t)\,dt \;+\; g_2(X_t,V_t,t)\,d\bB_t \;+\; g(X_t,V_t,t)\,d\bB^H_t,
\end{aligned}
\right.
\end{equation}
with \(X_0=1\) and \(V_0=0\). Here \(\bB_t\) is Brownian motion and \(\bB^H_t\) is an independent fractional Brownian motion with Hurst parameter \(H\in(0,1)\).

Since the system is driven by \(\bX_t = (t,\bB_t,\bB^H_t)\), the solution \((X_t,V_t)\) can be viewed as a functional of the corresponding time-extended path and, in particular, \(X_t\) inherits this dependence through its coupling with \(V_t\). To make this dependence more clear, we consider a simplified driftless, time-homogeneous case by assuming
\[
f_1 \equiv g_1 \equiv g_2 \equiv 0,
\qquad
f(X_t,V_t,t)=f(X_t,V_t),
\qquad
g(X_t,V_t,t)=g(X_t,V_t).
\]
Then \eqref{midprice process} reduces to
\begin{equation}\label{simple midprice process}
\left\{
\begin{aligned}
dX_t & =  f(X_t,V_t)\,d\bB_t ,\\[6pt]
dV_t & =  g(X_t,V_t)\,d\bB^H_t.
\end{aligned}
\right.
\end{equation}

Next, we introduce the trading strategies and the admissible class of controls. A trading speed \(u=(u_t)_{0\le t\le T}\) is required to be adapted to the information available to the trader. In particular, if the trader observes only the process \(\bX\), then we work with the filtration
\[
\mathcal F_t^{\bX}:=\sigma(\bX_s:\,0\le s\le t),
\]
and require \(u_t\) to be \(\mathcal F_t^{\bX}\)-adapted for all \(t\in[0,T]\). To make the problem more accessible, we consider the process \(\bX_t=(t, X_t,V_t)\) instead of working with the primary process. This choice simplifies the analysis and avoids introducing multiple approximations at different stages before arriving at the final result. Furthermore, branched signatures are employed for \(\bX_t=(t, X_t,V_t)\) because they are naturally suited to capturing nonlinear stochastic dynamics. The processes \(X_t\) and \(V_t\) are constructed using It\^o integrals. This choice is also financially motivated, since It\^o integrals are non-anticipative and adapted to the natural filtration, which makes them the standard framework in mathematical finance. 

We denote by \(\mathbb S\) the class of admissible trading speeds. Within \(\mathbb S\), we consider a subclass of controls generated by branched signatures, which we call the class of branched signature trading speeds:
\[
\SS_{\bsig}
:=
\left\{
\sum_{|\tau|\le N}\ell_{\tau}\,\innerprod{\bsig(\bX)_{0t},\tau}
\;:\;
N\in\NN,\ \ell_\tau\in\mathbb R,\ \tau\in\hopf
\right\}.
\]
Since each coordinate \(\innerprod{\bsig(\bX)_{0t},\tau}\) depends only on the path up to time \(t\), every \(u_t\in \SS_{\bsig}\) is \(\mathcal F_t^{\bX}\)-adapted, and hence \(\SS_{\bsig}\subseteq \mathbb S\).

Selling or purchasing a large number of shares of an asset can move the price adversely through temporary and permanent market impact. We model both effects as functionals of the branched signature of the underlying process. The execution price received by the trader, after accounting for market impact, is defined as follows.
\begin{defn}[Execution price process]
If the trader wishes to follow a branched signature trading speed \(u \in \SS_{\bsig}\), the price the trader receives, called execution price, is given by 
\[P^{u}_{t} := X_t - \ \innerprod{\bsig(\bX)_{0t}, \ell^{\tau}},\]
where \(\ell^{\tau} \in \hopf\) and models the \emph{market impact} and depends on the trading speed \(u\).
\end{defn}

The wealth of the trader at any time \(t \in [0,T]\) depends on the execution price \(P^{u}_{t}\) and the trading speed \(u_t\) and is defined as follows.
\begin{defn}[Wealth process]
For a branched signature trading speed \(u \in \SS_{\bsig}\), the wealth process \(W^u_t\) of a trader is defined as
\[W^u_t := \int_0^t P^u_s u_s ds.\]
\end{defn}

The trader can only liquidate a fixed number of the shares of an asset say \(i_0\). She needs to keep track of the shares that are left with her, called inventory, and defined as follows.

\begin{defn}[Inventory]
If \(i_0 > 0\) shows the initial inventory then the remaining inventory at time \(t\) is given as 
\[I_t^u := i_0 - \int_0^t u_s ds.\]
\end{defn}

The goal is to find an optimal trading speed \(u_t\) that maximizes the terminal wealth \(W_T^u\), while penalizing inventory holdings over time via the running cost
\(\phi \int_0^T (I_t^u)^2\,dt\), valuing the terminal inventory at the (impacted) terminal price through \(I_T^u P_T^u\), and penalizing any remaining inventory via the terminal cost \(\alpha (I_T^u)^2\). Accordingly, we define the objective function \(\mathcal{L}^u\) by
\[
\mathcal{L}^u := W_T^u \;-\; \phi \int_0^T (I_t^u)^2\,dt \;+\; I_T^u P_T^u \;-\; \alpha (I_T^u)^2,
\]
where \(\phi,\alpha \ge 0\) are constants. The optimal execution problem is then
\begin{equation} \label{optimization problem}
\sup_{u \in \mathbb{S}_{\bsig}} \mathbb{E}\!\left[\mathcal{L}^u\right].
\end{equation}
The next lemma implies that the objective function can be expressed as a functional of the branched signature of $\bX$. This is analogous to Lemma~3.1 in \cite{kalsi2020optimal}, where they consider the objective function as a functional of the classical signature.

\begin{lem} \label{notation}
For \(u \in \SS_{\bsig}\) given by \(\innerprod{\bsig(\bX)_{0t}, \tau}\) for some \(\tau \in \hopf\), we have the following
\begin{enumerate}
\item \(W_t^u = \innerprod{\bsig(\bX)_{0t}, [(\Forest{decorate[2]} + \mathbf{1}- \ell^{\tau})\tau ]_{\mathbf{1}}}\);
\item \(I_t^u = \innerprod{\bsig(\bX)_{0t}, q_0\mathbf{1} - [\tau]_{\mathbf{1}}}\);
\item \(\int_0^t (I^u_s)^2 ds = \innerprod{\bsig(\bX)_{0t}, \left[ \left(q_0 \mathbf{1} - [\tau]_{\mathbf{1}}\right) \left(q_0 \mathbf{1} - [\tau]_{\mathbf{1}}\right)\right]_{\mathbf{1}}}\);
\item \(I_t^u (P_t^u - \alpha I_t^u) = \innerprod{\bsig(\bX)_{0t}, (q_0\mathbf{1} - [\tau]_{\mathbf{1}})(\Forest{decorate[2]} + \mathbf{1}- \ell^{\tau}) - \alpha (q_0\mathbf{1} - [\tau]_{\mathbf{1}})(q_0\mathbf{1} - [\tau]_{\mathbf{1}})}\).
\end{enumerate}
\end{lem}

\begin{proof}
\begin{enumerate}
\item Using \(X_0=1\) we have \(X_t = \innerprod{\bsig(\bX)_{0t}, \Forest{decorate[2]} + \mathbf{1} }\). Therefore
\[P^u_t = X_t - \ \innerprod{\bsig(\bX)_{0t}, \ell^{\tau}} = \innerprod{\bsig(\bX)_{0t}, \Forest{decorate[2]} + \mathbf{1} - \ell^{\tau} }.\]
Hence, the wealth process is given as
\begin{align*}
W^u_t & := \int_0^t P^u_s u_s ds\\
& = \int_0^t \innerprod{\bsig(\bX)_{0s}, \Forest{decorate[2]} + \mathbf{1} - \ell^{\tau} } \innerprod{\bsig(\bX)_{0s}, \tau} ds\\
& = \int_0^t \innerprod{\bsig(\bX)_{0s}, (\Forest{decorate[2]} + \mathbf{1} - \ell^{\tau} ) \tau } ds\\
& =  \innerprod{\bsig(\bX)_{0s}, \left[(\Forest{decorate[2]} + \mathbf{1} - \ell^{\tau} ) \tau \right]_{\mathbf{1}} },
\end{align*}
where \(\left[(\Forest{decorate[2]} + \mathbf{1} - \ell^{\tau} ) \tau \right]_{\mathbf{1}}\) means grafting the root \(\mathbf{1}\) to the forest \((\Forest{decorate[2]} + \mathbf{1} - \ell^{\tau} ) \tau \).
\item Using the definition of the \(I^u_t\) we have
\begin{align*}
I_t^u   & := q_0 - \int_0^t u_s ds\\
        & = q_0 \innerprod{\bsig(\bX)_{0t}, \mathbf{1}} - \int_0^t \innerprod{\bsig(\bX)_{0s}, \tau} ds\\
        & =  \innerprod{\bsig(\bX)_{0t}, q_0 \mathbf{1}} - \innerprod{\bsig(\bX)_{0t}, [\tau]_{\mathbf{1}}}\\
        & =  \innerprod{\bsig(\bX)_{0t}, q_0 \mathbf{1} - [\tau]_{\mathbf{1}} }.
\end{align*}
\item Using the expression for \(I^u_t\) from part \((2)\) we get
\begin{align*}
\int_0^t (I^u_s)^2 ds & = \int_0^t \left(\innerprod{\bsig(\bX)_{0s}, q_0 \mathbf{1} - [\tau]_{\mathbf{1}} }\right)^2 ds\\
& = \int_0^t \innerprod{\bsig(\bX)_{0s}, \left(q_0 \mathbf{1} - [\tau]_{\mathbf{1}}\right) \left(q_0 \mathbf{1} - [\tau]_{\mathbf{1}}\right) } ds\\
& =  \innerprod{\bsig(\bX)_{0t}, \left[ \left(q_0 \mathbf{1} - [\tau]_{\mathbf{1}}\right) \left(q_0 \mathbf{1} - [\tau]_{\mathbf{1}}\right)\right]_{\mathbf{1}}}\\
\end{align*}
\item Using the expression for \(P^u_t\) and \(I^u_t\) from \((1)\) and \((2)\) we get
\[I_t^u (P_t^u - \alpha I_t^u) = \innerprod{\bsig(\bX)_{0t}, (q_0\mathbf{1} - [\tau]_{\mathbf{1}})(\Forest{decorate[2]} + \mathbf{1}- \ell^{\tau}) - \lambda (q_0\mathbf{1} - [\tau]_{\mathbf{1}})(q_0\mathbf{1} - [\tau]_{\mathbf{1}})}.\]
\end{enumerate}
\end{proof} 

Following the notation from Lemma~\ref{notation}, the objective function can be written as
\begin{align*}
\mathcal{L}^u
&= \bigl\langle \bsig(\bX)_{0T},\,
\left[(\Forest{decorate[2]} + \mathbf{1} - \ell^{\tau} ) \tau \right]_{\mathbf{1}}
-\phi \left[ \left(q_0 \mathbf{1} - [\tau]_{\mathbf{1}}\right)
\left(q_0 \mathbf{1} - [\tau]_{\mathbf{1}}\right)\right]_{\mathbf{1}} \\
&\qquad + (q_0\mathbf{1} - [\tau]_{\mathbf{1}})(\Forest{decorate[2]} + \mathbf{1}- \ell^{\tau})
- \lambda (q_0\mathbf{1} - [\tau]_{\mathbf{1}})(q_0\mathbf{1} - [\tau]_{\mathbf{1}})
\bigr\rangle.
\end{align*}
The optimization problem \ref{optimization problem} then reduces to
\begin{align} 
\sup_{\tau \in \hopf}
& \bigl\langle \mathbb{E} \left[\bsig(\bX)_{0T}\right],\,
\left[(\Forest{decorate[2]} + \mathbf{1} - \ell^{\tau} ) \tau \right]_{\mathbf{1}}
-\phi \left[ \left(q_0 \mathbf{1} - [\tau]_{\mathbf{1}}\right)
\left(q_0 \mathbf{1} - [\tau]_{\mathbf{1}}\right)\right]_{\mathbf{1}} \nonumber \\
&\qquad + (q_0\mathbf{1} - [\tau]_{\mathbf{1}})(\Forest{decorate[2]} + \mathbf{1}- \ell^{\tau})
- \lambda (q_0\mathbf{1} - [\tau]_{\mathbf{1}})(q_0\mathbf{1} - [\tau]_{\mathbf{1}})
\bigr\rangle, \label{opt}
\end{align}
where the \(\mathbb{E} \left[\bsig(\bX)_{0T}\right]\) is the expectation of the branched signature. 

Solving \eqref{opt} requires access to the expected branched signature. Since we cannot use any available optimized packages like \cite{reizenstein2018iisignature}  and \cite{kidger2020signatory} to compute the branched signature, therefore we learn the extended path \(\bar{\bX}_{\theta}\) of the underlying process and then compute the classical signature of the concatenated path
\[
\bbX_{\theta} := (\bX,\bar{\bX}_{\theta}).
\]

Using the relationship between branched and geometric rough paths, established through the Hopf algebra morphism \(\Psi\) introduced in \cite{hairer2015introduction}, we replace the branched signature of the primary process \(\bX\) by the classical signature of the extended path \(\bbX_{\theta}\). More precisely, the trading speed \(u_t\) is defined by
\[
u_t = \innerprod{\sig(\bbX_{\theta})_{0t}, \Psi(\tau)},
\]
while the corresponding execution price process is given by
\[
P_t^{u} := X_t - \innerprod{\sig(\bbX_{\theta})_{0t}, \Psi(\ell^{\tau})}.
\]
The components of the classical signature of \(\bbX_{\theta}\) corresponding to \(\Psi(\tau)\) and \(\Psi(\ell^{\tau})\) lie in the tensor algebra generated by the rooted trees of the Hopf algebra. Moreover, the classical signature satisfies the shuffle property. Consequently, the optimization problem \eqref{opt} can be rewritten in terms of the classical signature as follows.
\begin{align} 
\sup_{\tau \in \hopf}
& \bigl\langle \mathbb{E} \left[\sig^N(\bbX_{\theta})_{0T}\right],\,
\left[(\mathbf{2} + \Phi - \Psi(\ell^{\tau}) ) \shuffle \Psi(\tau) \right]  \mathbf{1}
-\phi \left[ \left(q_0 \Phi - \Psi(\tau)\mathbf{1}\right)
\left(q_0 \Phi - \Psi(\tau)\mathbf{1}\right)\right]\mathbf{1} \Bigr. \nonumber \\
& \Bigl. \qquad + \left(q_0 \Phi - \Psi(\tau)\mathbf{1}\right)(\mathbf{2} + \Phi - \Psi(\ell^{\tau}) )
- \lambda \left(q_0 \Phi - \Psi(\tau)\mathbf{1}\right)\left(q_0 \Phi - \Psi(\tau)\mathbf{1}\right) 
\bigr\rangle. \label{opt classical sig}
\end{align}
The optimization problem \ref{opt classical sig} can be solved by using classical optimization tools. To solve this problem we first learned the extended path \(\Bar{\bX}_{\theta}\) by using \(1{,}000\) independently generated paths of the primary process. The extension \(\Bar{\bX}_{\theta}\) is learned via minimizing the loss defined in \ref{total loss}. This loss the sum of the calibration loss \ref{physics_informed loss} and shuffle product loss \ref{shuffle product loss}. To keep the problem less computationally expensive we set the dimension of extended path to be equal to \(2\). We simulated \(50{,}000\) paths of the primary process and passed through the learned network to get the extended path. In both cases, we estimated the expected terminal classical signatures of the processes using Monte Carlo and solved \ref{opt classical sig}. For all of our experiments we consider temporary and permanent market impact and define the execution price  \(P_t^u\) as follows 
\[
P_t^u =  X_t - k u_t - \lambda \int_0^t u_s ds,
\]
where \(k u_t\) corresponds to temporary impact and \(\lambda \int_0^t u_s ds\) corresponds to permanent impact.

Finally, for testing, we simulate \(50{,}000\) new realizations of the process and obtain the corresponding signature-based and branched-signature-based trading strategies. We conduct numerical experiments in three different settings: when the price process is driven by Brownian motion, fractional Brownian motion, and coupled dynamics. Each case is discussed in the following subsections, together with the corresponding comparison metrics and parameter choices.

\subsection{Brownian motion price process}
In this case, the midprice process is given by
\[
X_t = \sigma \bB_t, \qquad \sigma > 0.
\]
In our numerical experiment, we take \(\sigma = 0.02\). To compute the streamed and terminal signatures, we consider the underlying path
\[
\bX_t = (t, \bB_t).
\]
Although the time component is included as an extension so that the underlying path is time-augmented Brownian motion, we still learn an extension network \(\bar{\bX}_{\theta}\). Other parameters used for the experiment are \(i_0=1\), \(\lambda=10^{-3}\), \(k=10^{-3}\), \(\alpha=100\) and \(T=1\). We consider different values of parameter \(\phi\). Corresponding to each value the running inventory is given in the Figure \ref{inventory_all_BM} . Also, Figure \ref{inventory_0.005_BM} gives inventory for \(100\) realizations of the price process in case of \(\phi= 0.005\) . It is clear that if there is no penalty from running inventory i.e., \(\phi=0\), then the trading speed is constant.

\begin{figure}[H]
  \centering
  \begin{subfigure}{0.48\linewidth}
    \centering
    \includegraphics[width=0.9\linewidth]{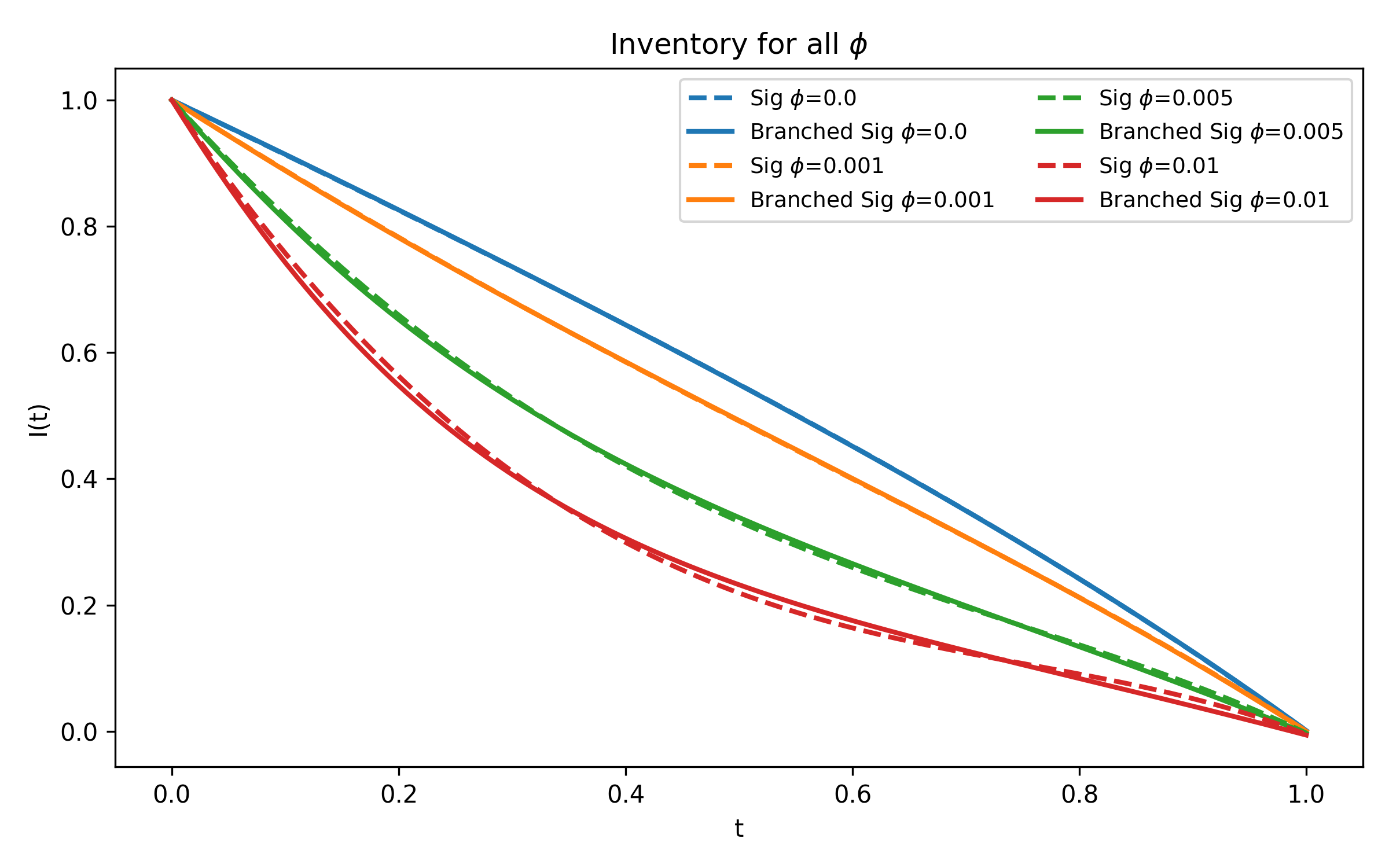}
    \caption{\textit{Mean signature and branched signature inventory for different values of \(\phi\) .}}
    \label{inventory_all_BM}
  \end{subfigure}\hfill
  \begin{subfigure}{0.48\linewidth}
    \centering
    \includegraphics[width=1\linewidth]{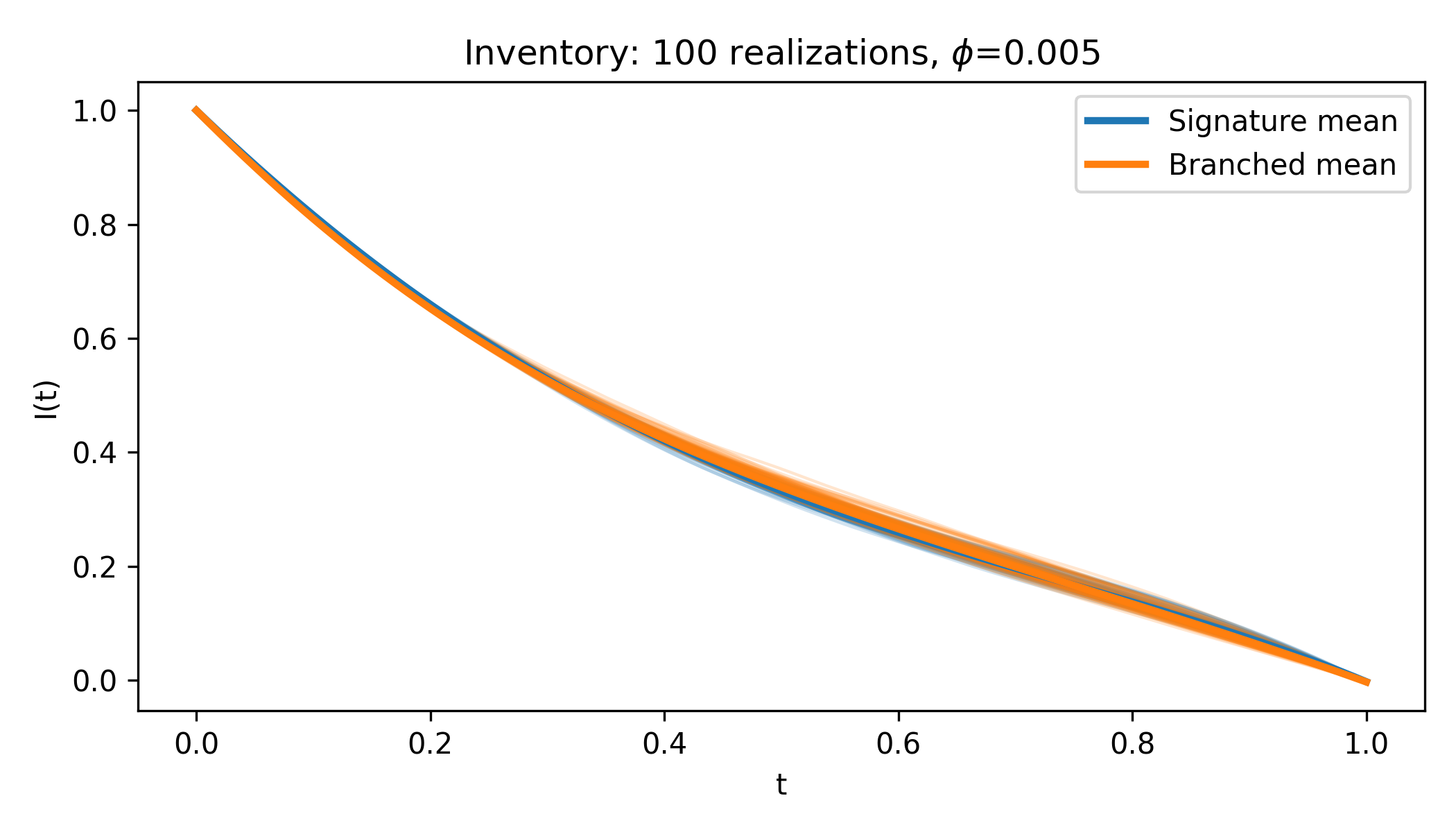}
    \caption{\textit{Signature and branched signature inventories for \(100\) price path realizations in the case of \(\phi=0.005\) .}}
    \label{inventory_0.005_BM}
  \end{subfigure}
  \caption{}
\end{figure}

Moreover, we also report the standard deviation and the histogram of the terminal wealth associated with the learned trading strategy. The standard deviation is computed for different values of \(\phi\) and is shown in Figure \ref{std_BM}, while the histogram of terminal wealth is presented only for \(\phi = 0.005\) in Figure \ref{wealth_BM}. We observe that there is little difference between the signature-based and branched-signature-based trading strategies. This is because the price process is simply Brownian motion, and we already augment the path with the time component, which is known to provide the natural extension in this setting.

\begin{figure}[H]
  \centering
  \begin{subfigure}{0.48\linewidth}
    \centering
    \includegraphics[width=1\linewidth]{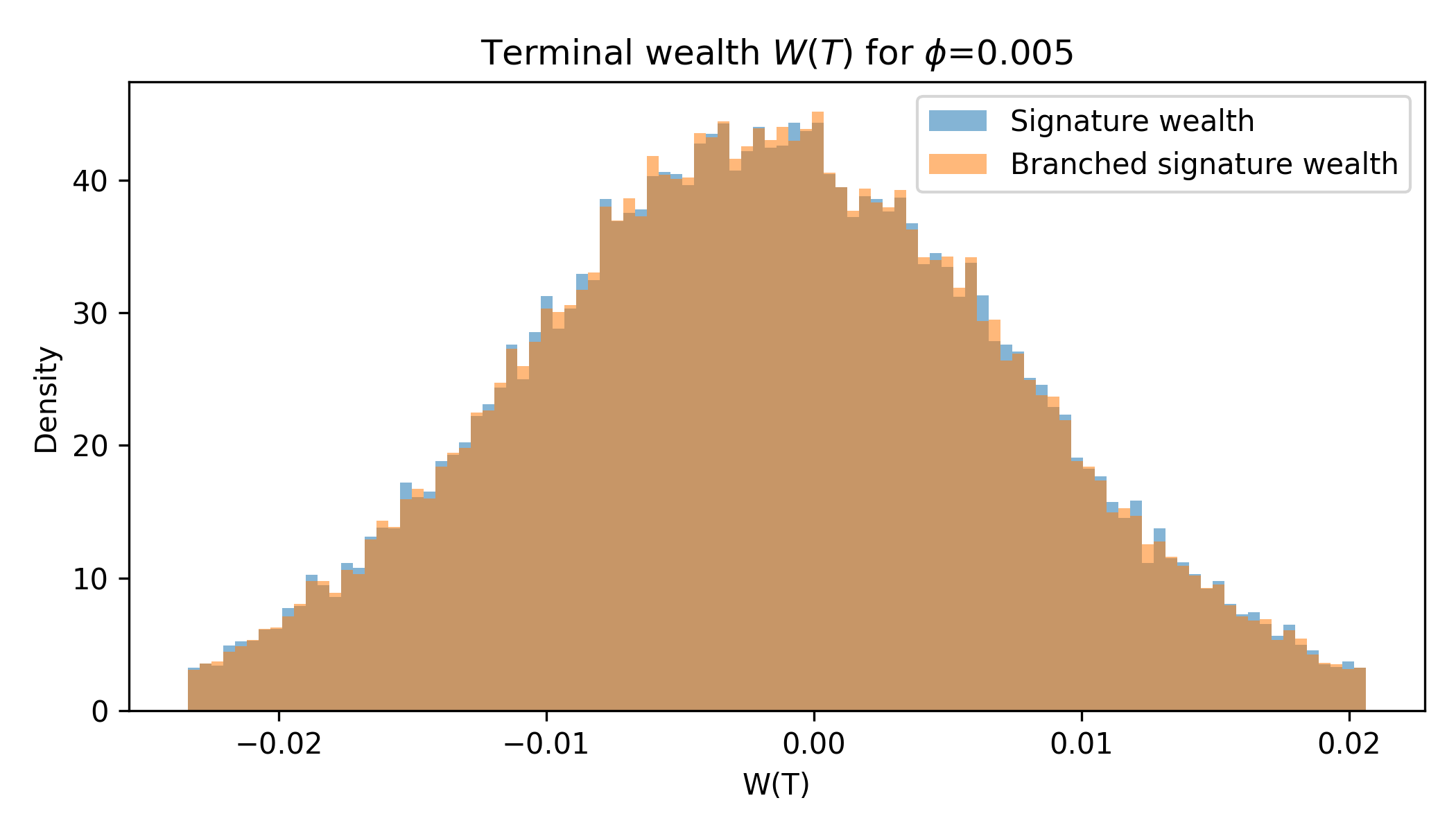}
    \caption{\textit{Signature and branched signature wealth distribution in the case of Brownian motion price process for \(\phi=0.005\).}}
    \label{wealth_BM}
  \end{subfigure}\hfill
  \begin{subfigure}{0.48\linewidth}
    \centering
    \includegraphics[width=1\linewidth]{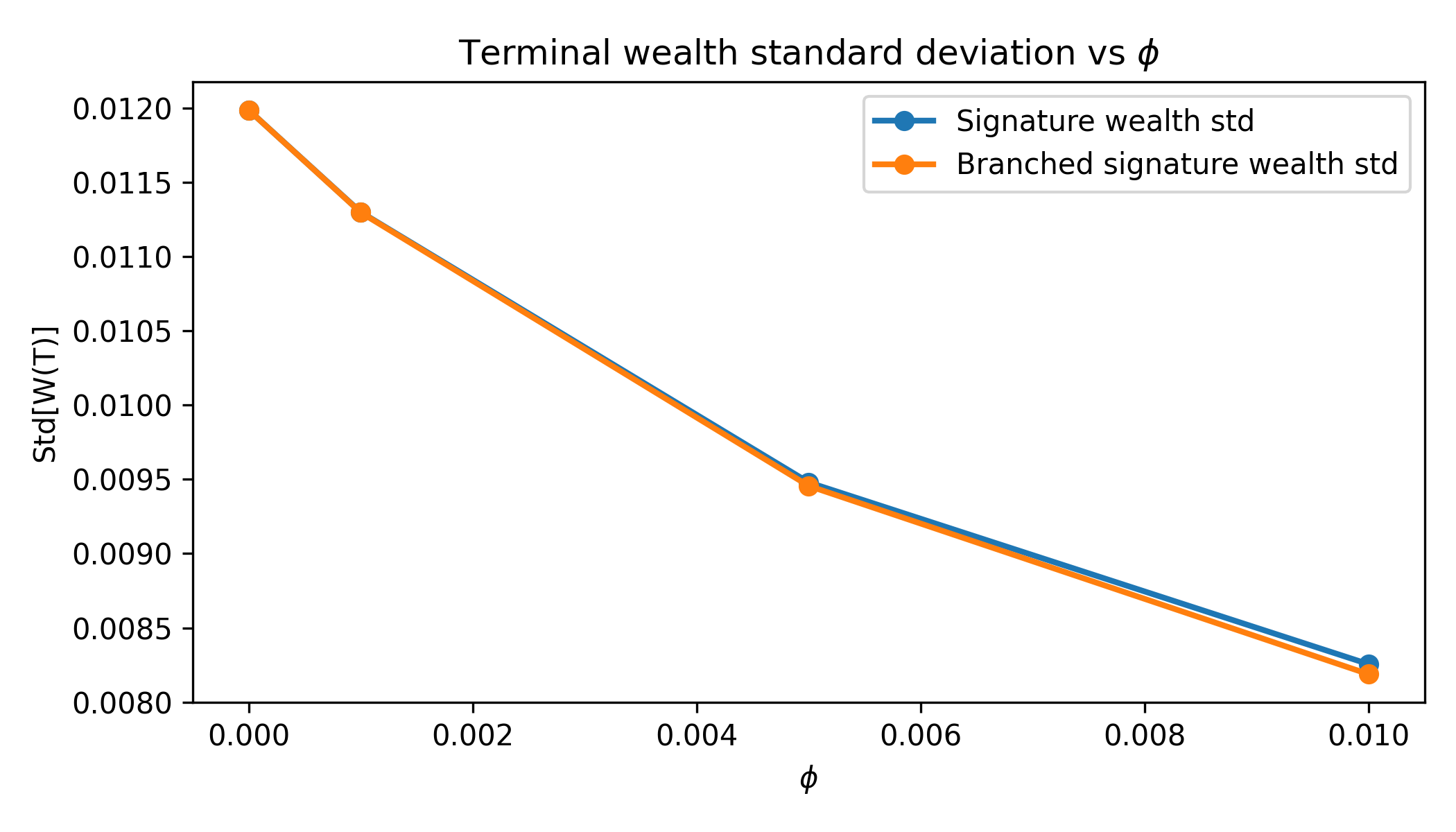}
    \caption{\textit{Standard deviation of signature and branched signature terminal wealth distribution for different values of \(\phi\).}}
    \label{std_BM}
  \end{subfigure}
  \caption{}
\end{figure}

\subsection{Fractional Brownian motion price process}
In this second experiment,  we consider the following midprice process
\[
X_t = \sigma \bB^H_t, \qquad \sigma > 0.
\]
Again we take \(\sigma = 0.02\) and to compute the streamed and terminal signatures, we consider the underlying path
\[
\bX_t = (t, \bB^H_t).
\]
As in the previous setting, we train an extension network \(\bar{\bX}_{\theta}\). In this experiment, we use the parameter values
\(i_0 = 1,  \lambda = 10^{-3},  k = 10^{-3},  \alpha = 100\), and \(T = 1\).
We again examine several choices of the parameter \(\phi\), following the same range considered in the Brownian motion case. For each value of \(\phi\), the corresponding mean inventory trajectories are displayed in Figure \ref{inventory_all_fBM}. In addition, Figure \ref{inventory_0.005_fBM} shows the inventory paths for \(100\) simulated realizations of the price process when \(\phi = 0.005\). It can be observed that the mean remains close to that in the Brownian motion case, whereas the individual sample paths deviate much more from the mean.
\begin{figure}[H]
  \centering
  \begin{subfigure}{0.48\linewidth}
    \centering
    \includegraphics[width=0.9\linewidth]{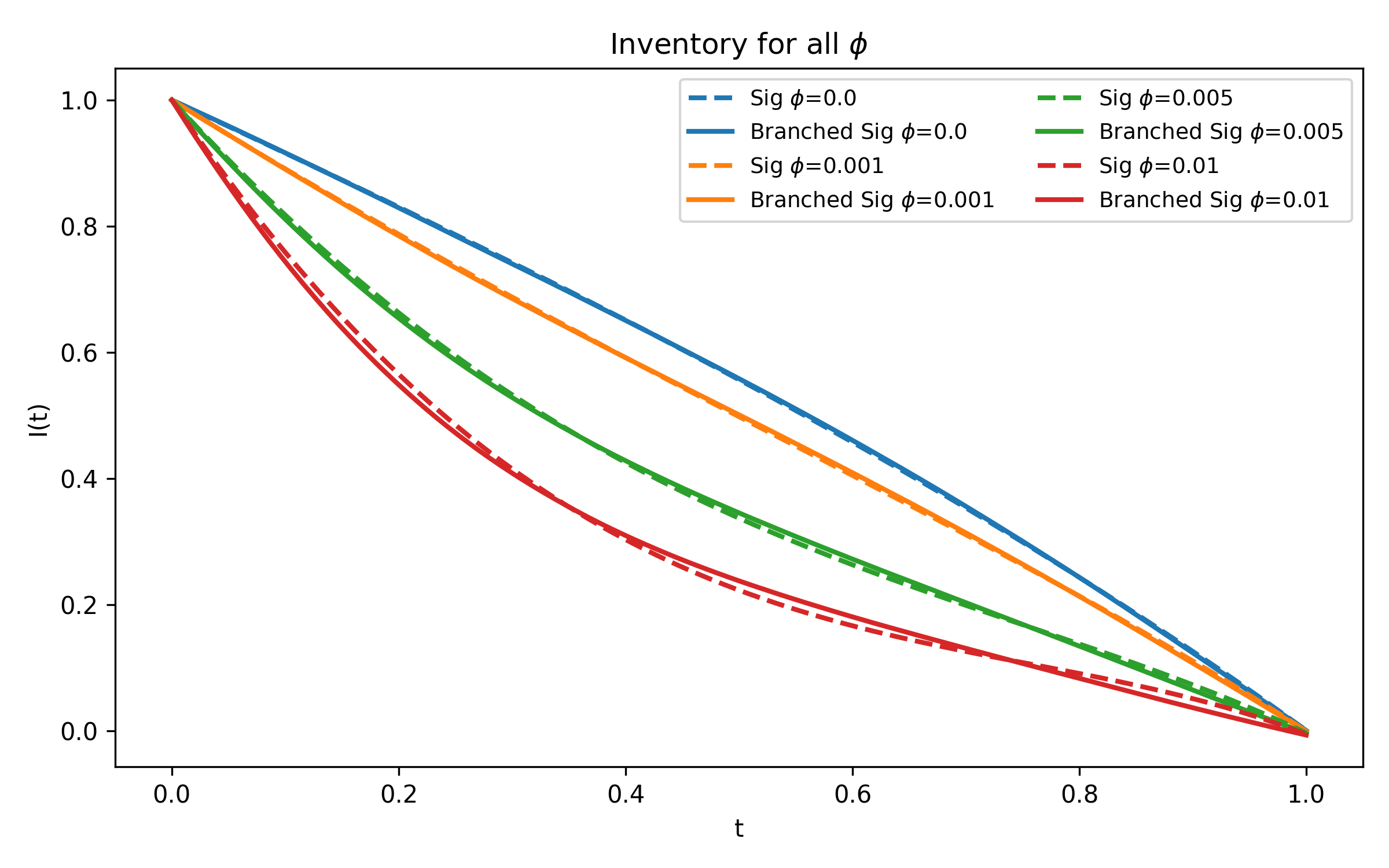}
    \caption{\textit{Mean signature and branched signature inventory for different values of \(\phi\) when price process follows fractional Brownian motion.}}
    \label{inventory_all_fBM}
  \end{subfigure}\hfill
  \begin{subfigure}{0.48\linewidth}
    \centering
    \includegraphics[width=1\linewidth]{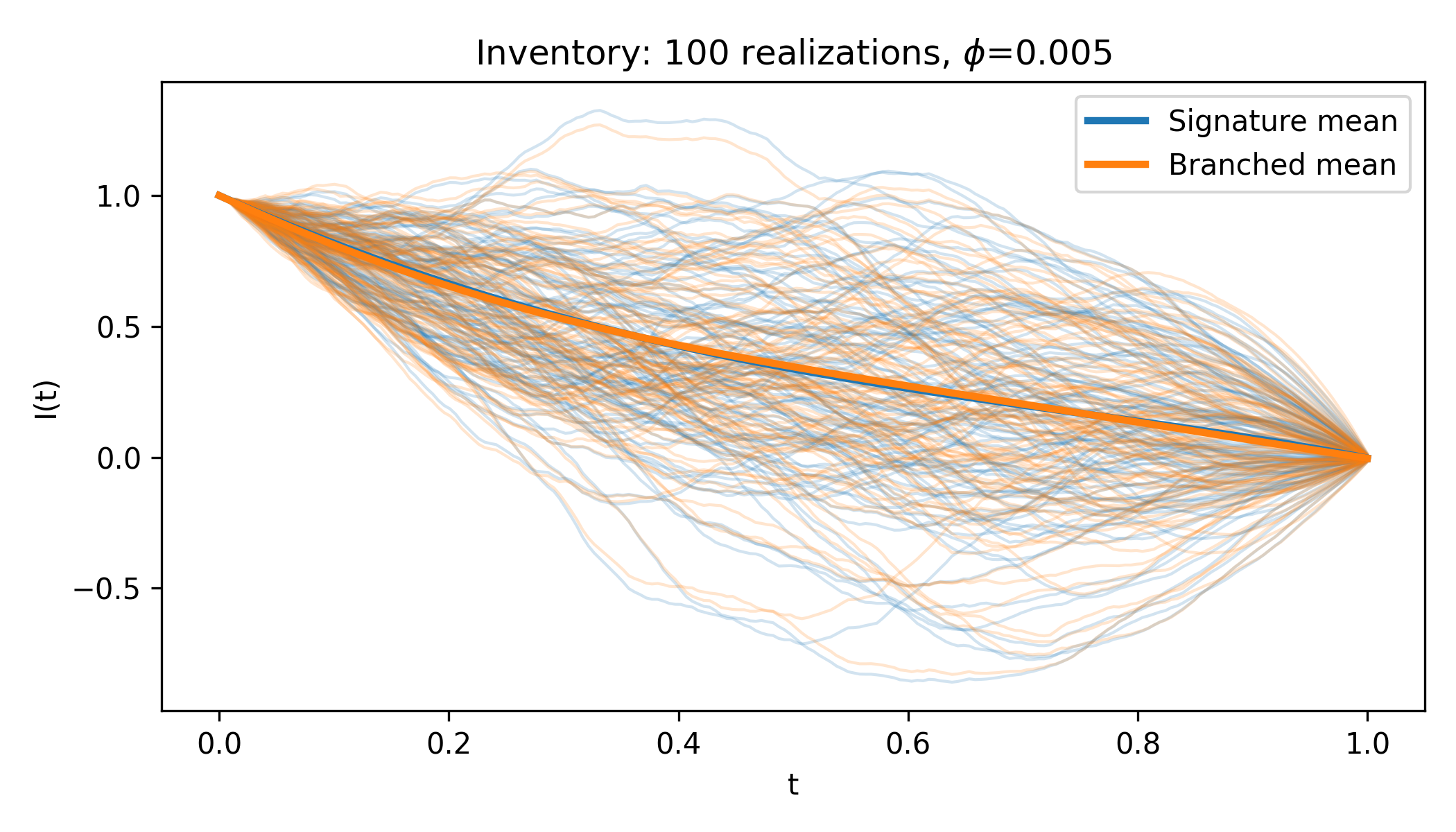}
    \caption{\textit{Signature and branched signature inventories for \(100\) price path realizations in the case of \(\phi=0.005\) for fractional Brownian motion price process.}}
    \label{inventory_0.005_fBM}
  \end{subfigure}
  \caption{}
\end{figure}

Moreover, we report the standard deviation and the histogram of the terminal wealth corresponding to the learned trading strategies in the fractional Brownian motion setting. The standard deviation is computed for different values of \(\phi\) and displayed in Figure \ref{std_fBM}, while the histogram of terminal wealth for the choice \(\phi = 0.005\) is shown in Figure \ref{wealth_fBM}. From Figure \ref{std_fBM}, we observe that the standard deviation decreases as \(\phi\) increases for both the signature and branched signature strategies. In addition, the branched signature strategy yields a smaller standard deviation than the signature one across all values of \(\phi\). Overall, the difference between the two approaches remains small, but the branched signature strategy exhibits a slight advantage in terms of reduced standard deviation, which is natural in the fractional Brownian motion setting where the learned extension may capture additional nontrivial information.

\begin{figure}[H]
  \centering
  \begin{subfigure}{0.48\linewidth}
    \centering
    \includegraphics[width=1\linewidth]{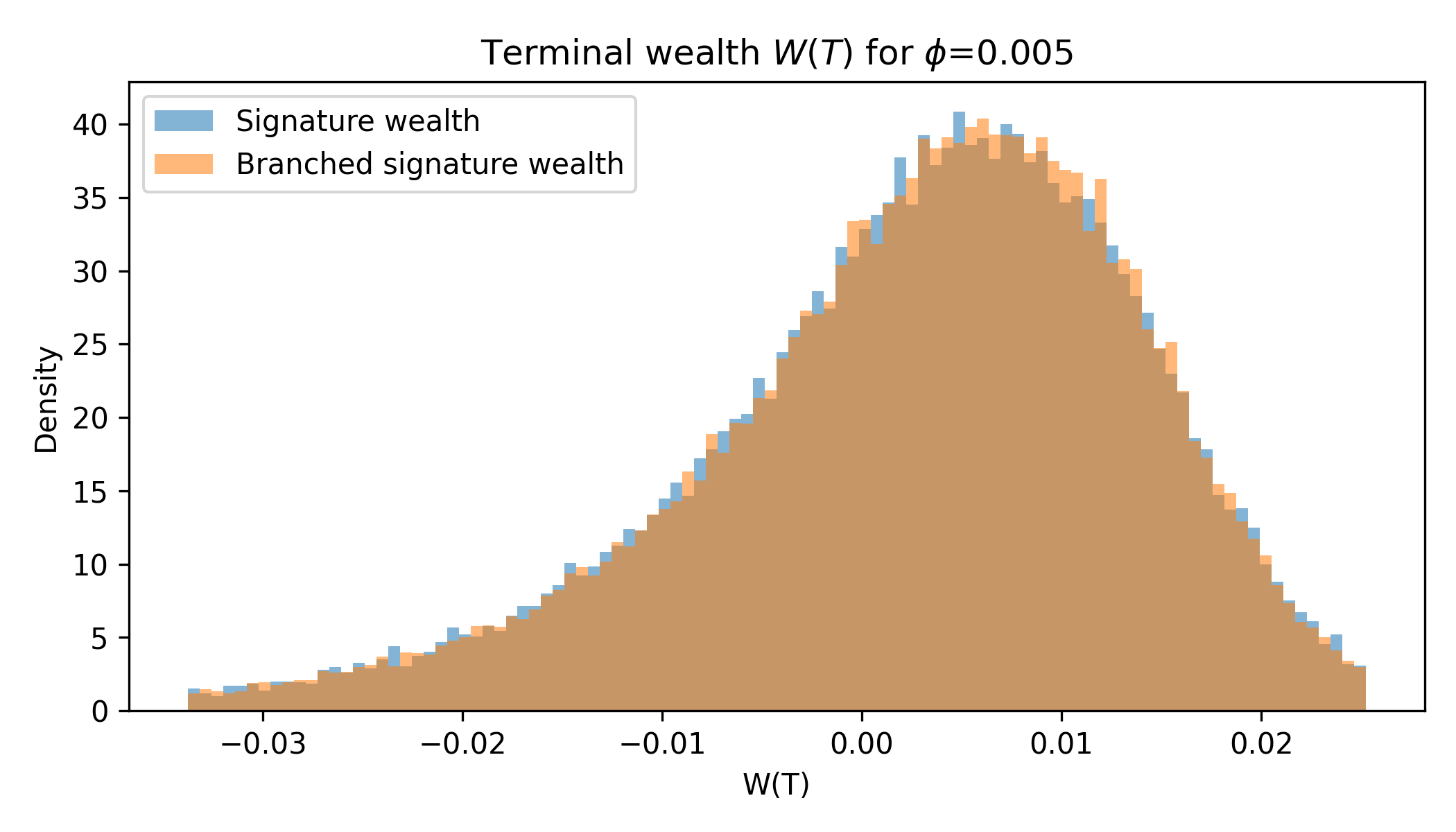}
    \caption{\textit{Signature and branched signature wealth distribution in the case of fractional Brownian motion price process for \(\phi=0.005\).}}
    \label{wealth_fBM}
  \end{subfigure}\hfill
  \begin{subfigure}{0.48\linewidth}
    \centering
    \includegraphics[width=1\linewidth]{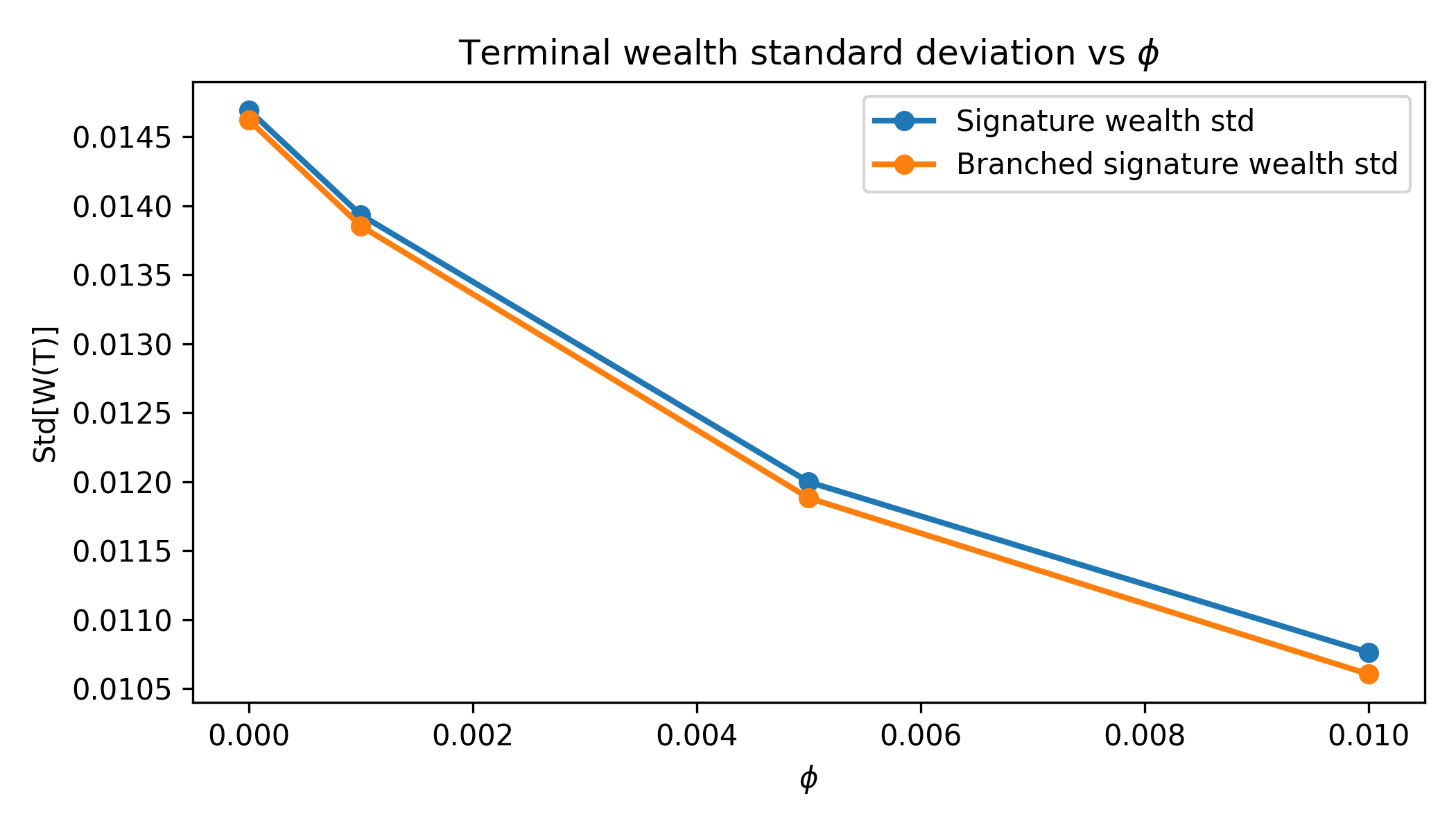}
    \caption{\textit{Standard deviation of signature and branched signature terminal wealth distribution for different values of \(\phi\).}}
    \label{std_fBM}
  \end{subfigure}
  \caption{}
\end{figure}

\subsubsection{Coupled dynamics}
In the final experiment, we consider the price and volatility processes governed by the coupled dynamics in \eqref{simple midprice process}, with
\[
f(X_t,V_t)=aV_tX_t, \qquad g(X_t,V_t)=bV_t,
\]
so that
\begin{equation}\label{specific dynamics}
\left\{
\begin{aligned}
dX_t &= a\,V_tX_t\,d\bB_t,\\[4pt]
dV_t &= b\,V_t\,d\bB_t^H,
\end{aligned}
\right.
\end{equation}
where \(a=0.02\) and \(b=0.02\). This experiment is different from previous two in the way that we have simulate both Brownian and fractional Brownian motion processes and then use Euler-Maruyama to simulate the coupled process. We take the underlying path to be
\[
\bX_t=(t,X_t,V_t),
\]
and train the extension network \(\bar{\bX}_{\theta}\) as before. The parameters are \(i_0=1, \lambda=10^{-3}, k=10^{-3}, \alpha=100,\) and  \(T=1\). We consider the same range of \(\phi\) as in the Brownian motion case. The mean inventory trajectories are shown in Figure \ref{inventory_all_fBM}, while Figure \ref{inventory_0.005_fBM} displays the inventory paths for \(100\) simulated realizations when \(\phi=0.005\).

\begin{figure}[H]
  \centering
  \begin{subfigure}{0.48\linewidth}
    \centering
    \includegraphics[width=0.9\linewidth]{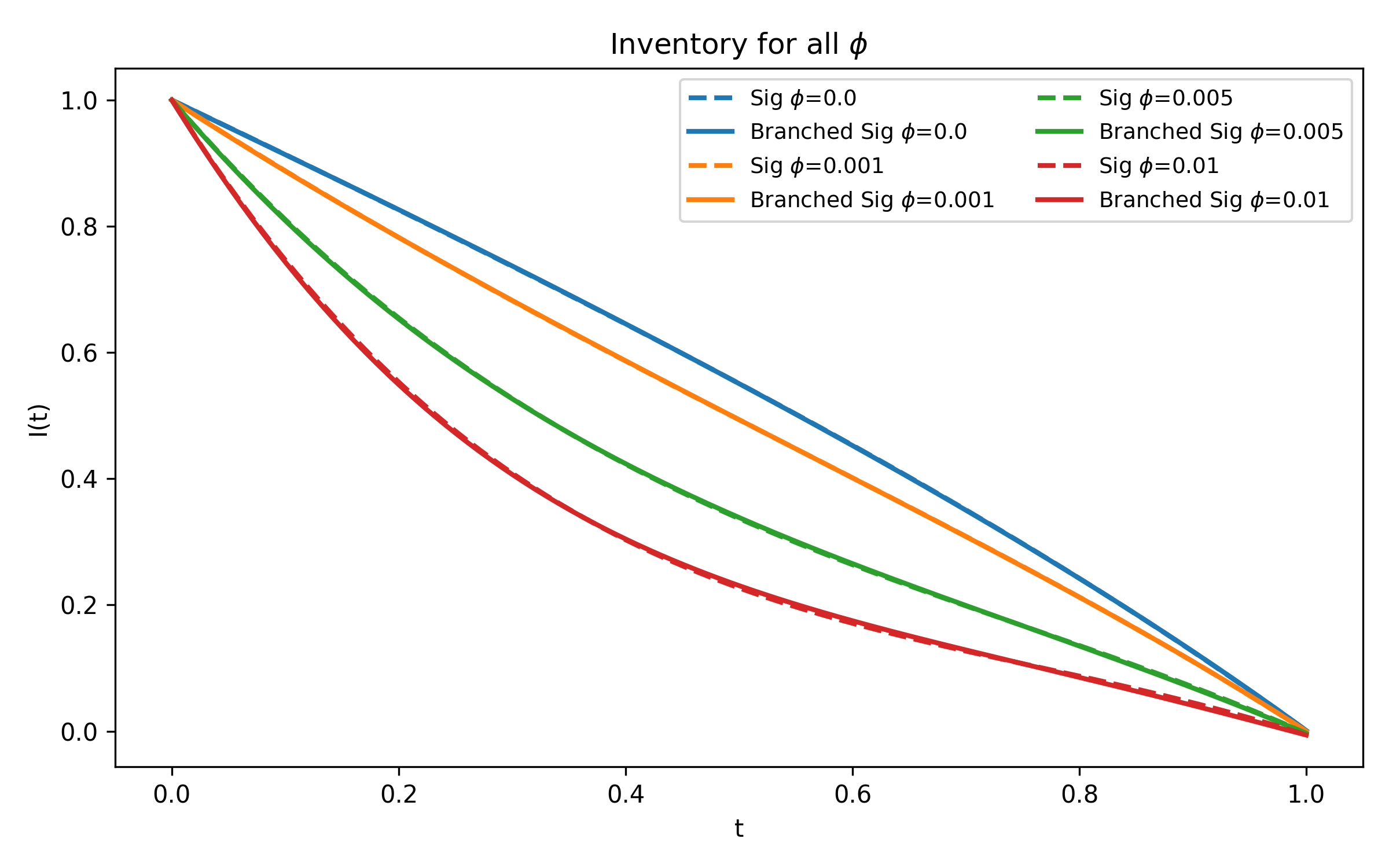}
    \caption{\textit{Mean signature and branched signature inventory for different values of \(\phi\) when price process follows copled dynamics \ref{specific dynamics}.}}
    \label{inventory_all_fBM}
  \end{subfigure}\hfill
  \begin{subfigure}{0.48\linewidth}
    \centering
    \includegraphics[width=1\linewidth]{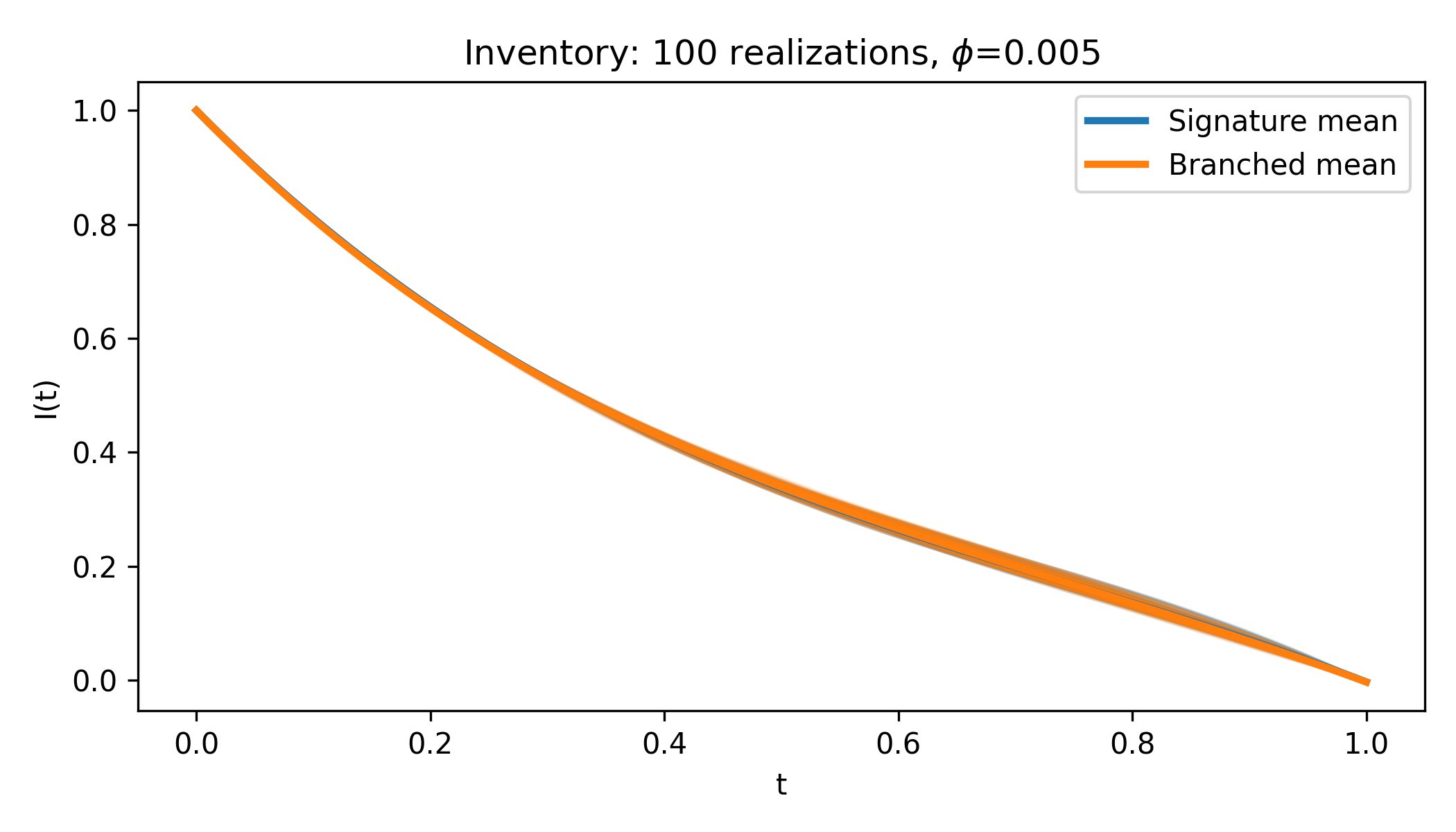}
    \caption{\textit{Signature and branched signature inventories for \(100\) price path realizations in the case of \(\phi=0.005\) for coupled dynamics \ref{specific dynamics}.}}
    \label{inventory_0.005_fBM}
  \end{subfigure}
  \caption{}
\end{figure}

Finally, we report the standard deviation and the histogram of the terminal wealth for the learned trading strategies in the coupled dynamics. The standard deviation is computed for different values of \(\phi\) and displayed in Figure \ref{std_coupled}, while the histogram of terminal wealth for \(\phi=0.005\) is shown in Figure \ref{wealth_coupled}. From Figure \ref{std_coupled}, we observe that the standard deviation decreases as \(\phi\) increases. At the same time, the signature and branched signature strategies produce almost identical values of the standard deviation, and their terminal wealth histograms nearly overlap. Although this experiment does not show a significant advantage of one method over the other, it still shows that both approaches remain stable and consistent even in this coupled setting. A possible explanation is that the extended path has dimension only \(2\), which may limit its ability to encode the nonlinear structure of the coupled system. Moreover, if the extension network is not trained for sufficiently many iterations, the learned extension may fail to approximate the desired dynamics well.

\begin{figure}[H]
  \centering
  \begin{subfigure}{0.48\linewidth}
    \centering
    \includegraphics[width=1\linewidth]{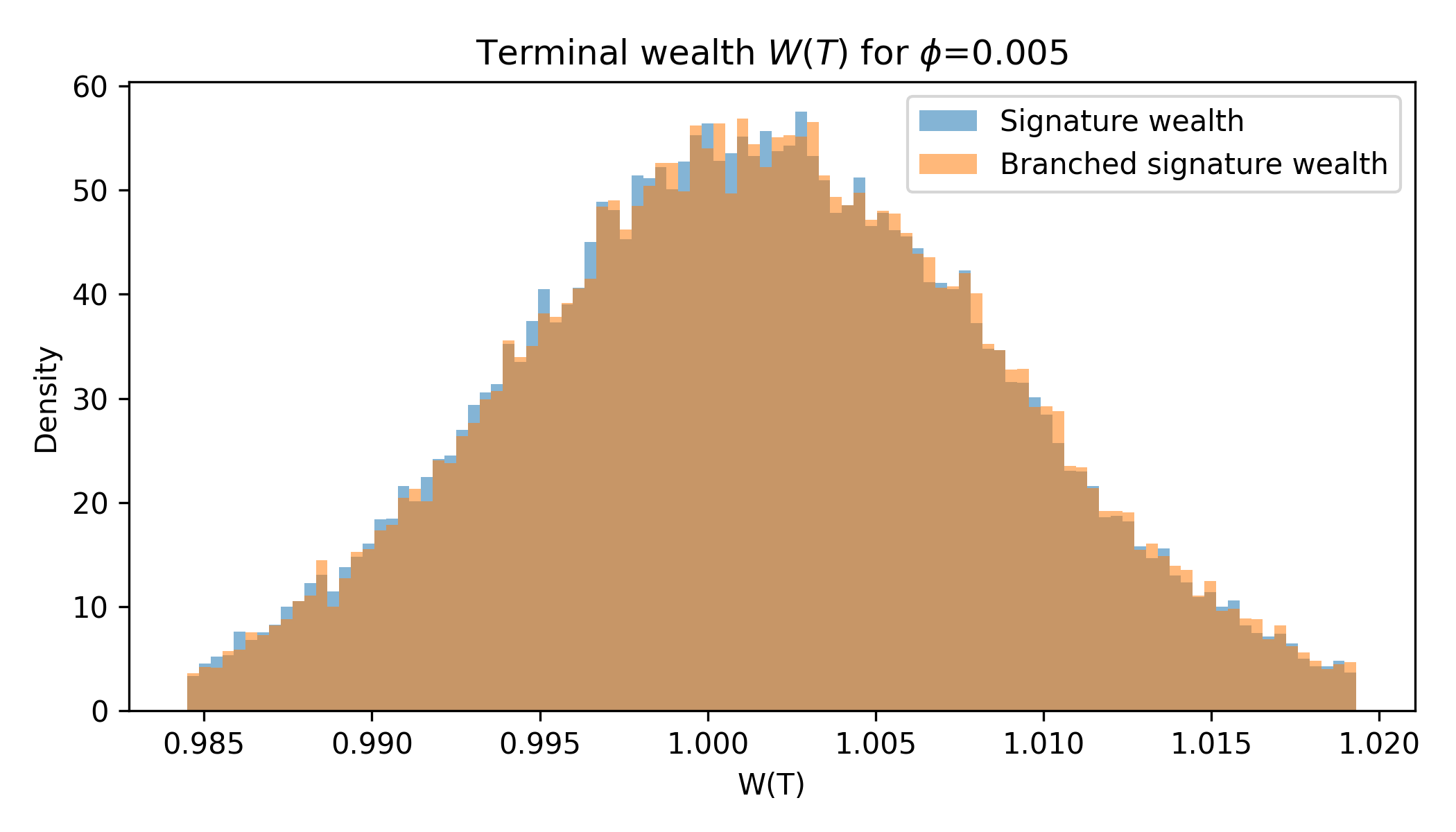}
    \caption{\textit{Signature and branched signature wealth distribution in the case of coupled dynamics for \(\phi=0.005\).}}
    \label{wealth_coupled}
  \end{subfigure}\hfill
  \begin{subfigure}{0.48\linewidth}
    \centering
    \includegraphics[width=1\linewidth]{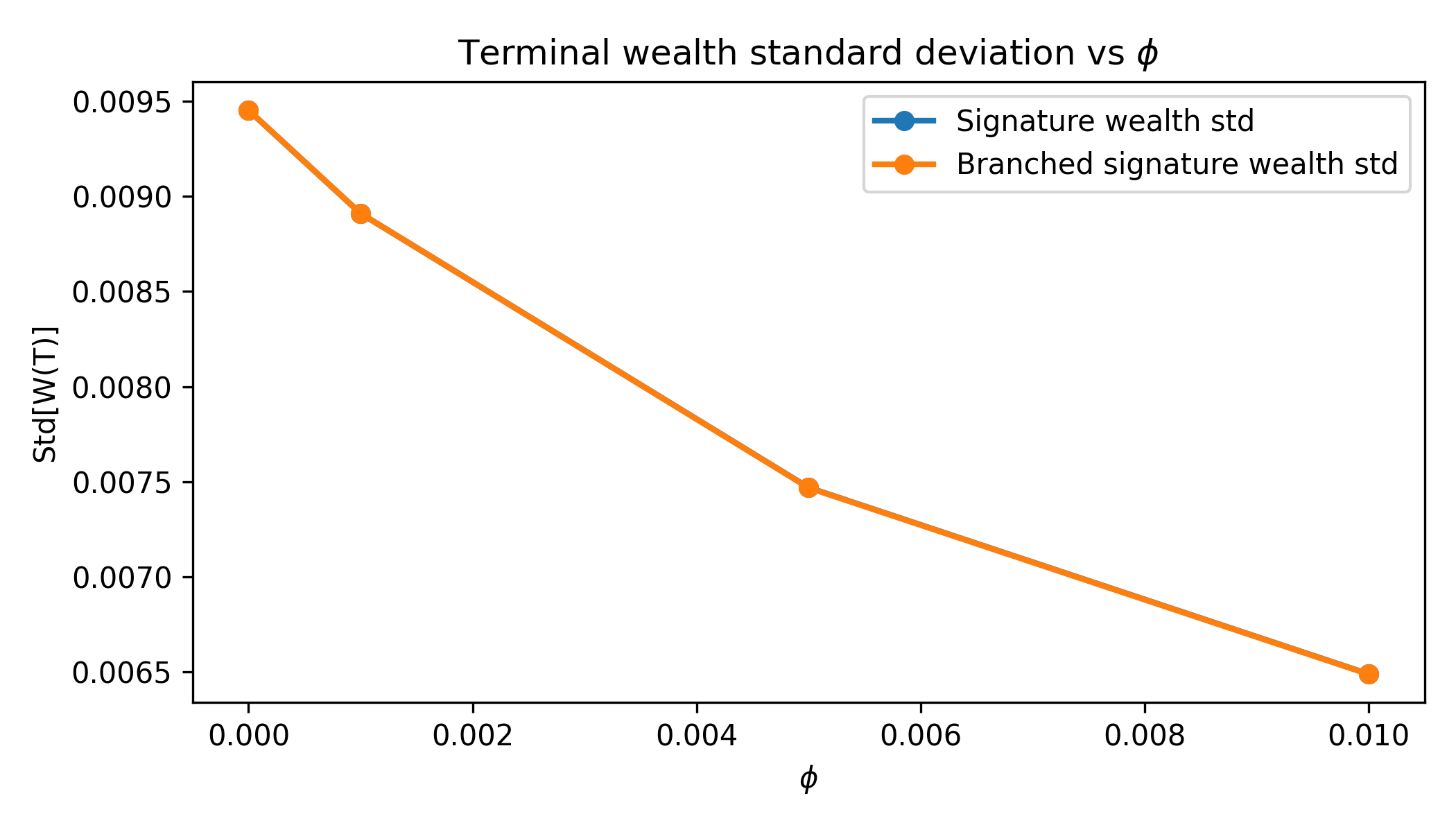}
    \caption{\textit{Standard deviation of signature and branched signature terminal wealth for different values of \(\phi\).}}
    \label{std_coupled}
  \end{subfigure}
  \caption{}
\end{figure}

\bibliographystyle{apalike}
\bibliography{myrefs}
\end{document}